\documentclass[11pt]{articlefederico}
\usepackage{amsmath,amssymb,amsfonts}
\usepackage{amsthm}
\usepackage{enumerate}
\usepackage{stmaryrd}
\usepackage{young}
\usepackage{xcolor}
\usepackage{graphicx}
\usepackage{epsfig}
\usepackage{epstopdf}
\usepackage{color}
\usepackage{float}
\usepackage[a4paper]{geometry}
\usepackage{ytableau}
\usepackage{mathtools}
\usepackage{hyperref}
\usepackage[none]{hyphenat}
\usepackage{lscape}
\usepackage{verbatim}
\usepackage{wasysym}
\usepackage{tikz}
\usetikzlibrary{arrows, matrix, fit}
\usepackage{mathtools}
\usepackage{colortbl}

\usepackage{tikz-cd} 

\usepackage{tikz}

\newcommand{\Pentagon}{
  \tikz[baseline=-0.5ex, scale=0.2]{
    \draw (90:1) -- (162:1) -- (234:1) -- (306:1) -- (18:1) -- cycle;
    \fill (0,0) circle (0.15);
  }
}

\newtheorem{theorem}{Theorem}[section]
\newtheorem{corollary}[theorem]{Corollary}

\newtheorem{lemma}[theorem]{Lemma}
\newtheorem{remark}[theorem]{Remark}

\newtheorem{proposition}[theorem]{Proposition}

\newtheorem{conjecture}[theorem]{Conjecture}
\newtheorem{definition}[theorem]{Definition}
\newtheorem{example}[theorem]{Example}

\renewcommand{\a}{\mathsf{a}}
\renewcommand{\b}{\mathsf{b}}

\newcommand{\A}{\mathsf{A}}
\newcommand{\B}{\mathsf{B}}
\newcommand{\ABHY}{\mathrm{ABHY}}
\newcommand{\AFV}{\mathrm{AFV}}

\newcommand{\1}{\mathbf{1}}

\newcommand{\C}{\mathcal{C}}
\newcommand{\F}{\mathcal{F}}

\newcommand{\N}{\mathcal{N}}

\renewcommand{\S}{\mathcal{S}}

\newcommand{\X}{\mathcal{X}}

\newcommand{\U}{\mathcal{U}}

\newcommand\Curves{\operatorname{Curves}}
\newcommand\Chords{\operatorname{Chords}}
\newcommand\Spokes{\operatorname{Spokes}}

\renewcommand{\c}{\mathsf{c}}

\newcommand\Brackets{\operatorname{Brackets}}
\newcommand\Cosmo{\operatorname{Cosmo}}
\newcommand\CosmoFan{\operatorname{CosmoFan}}
\newcommand\Loday{\operatorname{Loday}}
\newcommand\Assoc{\operatorname{Assoc}}
\newcommand\BrAssoc{\operatorname{BrAssoc}}
\newcommand\AssocFan{\operatorname{AssocFan}}
\newcommand\BrAssocFan{\operatorname{BrAssocFan}}
\newcommand\BdCorrela{\operatorname{BdCorrela}}
\newcommand\Cycle{\operatorname{Cycle}}
\newcommand\Br{\operatorname{Br}}
\newcommand\PreBr{\operatorname{PreBr}}
\newcommand\Mat{\operatorname{Mat}}
\newcommand\Tr{\operatorname{Tr}}
\newcommand\Tree{\operatorname{Tree}}
\newcommand\BrTree{\operatorname{BrTree}}
\newcommand\bracket{\operatorname{bracket}}
\newcommand\codim{\operatorname{codim}}
\newcommand\cone{\operatorname{cone}}
\newcommand\conv{\operatorname{conv}}
\newcommand\diag{\operatorname{diag}}
\newcommand\interior{\operatorname{int}}
\newcommand\tree{\operatorname{tree}}
\newcommand\Perm{\operatorname{Perm}}

\newcommand\Newt{\operatorname{Newt}}

\newcommand{\R}{\mathbb{R}}
\newcommand{\Z}{\mathbb{Z}}

\usepackage{enumitem}
\usepackage{xcolor}

\definecolor{darkgreen}{rgb}{0,.5,0}
\definecolor{brown}{rgb}{0.5,0.3,0}

\title{\textsf{Combinatorics of the cosmohedron }  \includegraphics[scale=0.03]{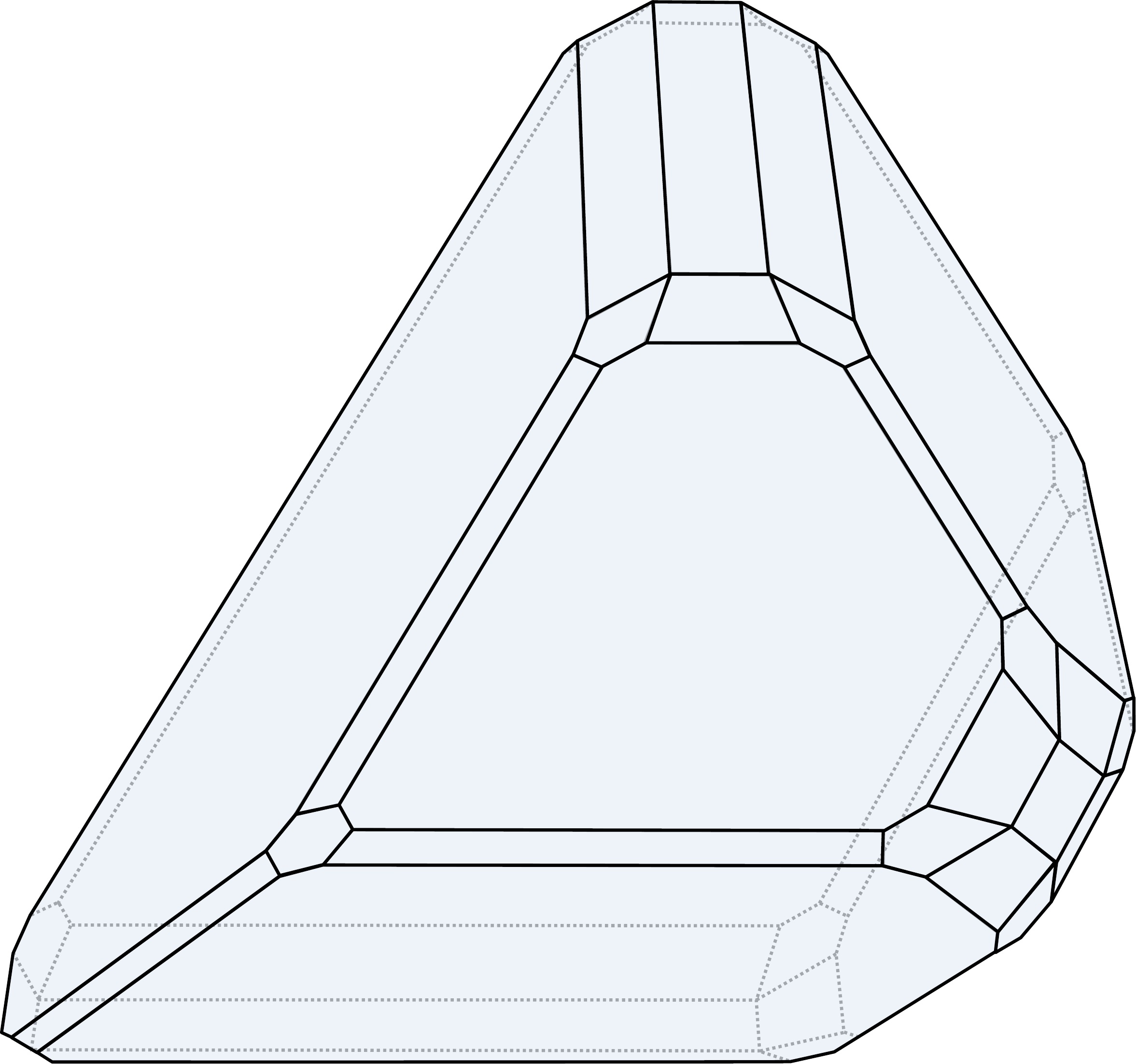}}

\date{}
\author{ 
\textsf{Federico Ardila-Mantilla}\footnote{\noindent \textsf{Queen Mary University of London and San Francisco State University; \texttt{f.ardila@qmul.ac.uk}.}}
\qquad \and \qquad
\textsf{Nima Arkani-Hamed}\footnote{\noindent \textsf{Institute for Advanced Study, Princeton; \texttt{arkani@ias.edu.}}}
\qquad \and \qquad
\textsf{Carolina Figueiredo}\footnote{\noindent \textsf{Jadwin Hall, Princeton University; \texttt{cfigueiredo@princeton.edu}. }}
\qquad \and \qquad
\textsf{Francisco Vazão}\footnote{\noindent \textsf{Institute for Advanced Study, Princeton; \texttt{fvvazao@ias.edu}.}}
\qquad \qquad }
\date{}
\begin{document}

\maketitle

\begin{abstract}
The cosmohedron was recently proposed as a polytope underlying the cosmological wavefunction for $\text{Tr}(\Phi^3)$ theory. Its faces were conjectured to be in bijection with Matryoshkas, which are obtained from a subdivision of a polygon by sequentially wrapping groups of polygons into larger polygons. In this paper, we prove the correctness of this construction and elucidate its combinatorial structure. Cosmohedra generalize to a wider class of $\mathcal{X}$ in $Y$ polytopes, where we chisel a polytope from the family $\mathcal{X}$ at each vertex of a polytope $Y$. 
We sketch a new application of these chiseled polytopes to the physics of ultraviolet divergences in loop-integrated Feynman amplitudes. 
\end{abstract}

\section{\textsf{Introduction}} \label{sec:intro}

It has been a long-standing question in cosmology to find a geometric object that encodes the cosmological wavefunction of  Tr$(\phi^3)$ theory. 
The combinatorics of this wavefunction is governed by certain nesting subdivisions of a polygon
that we call \emph{Matryoshkas}.

\begin{figure}[h]
\begin{center}
\includegraphics[width=0.24\linewidth]{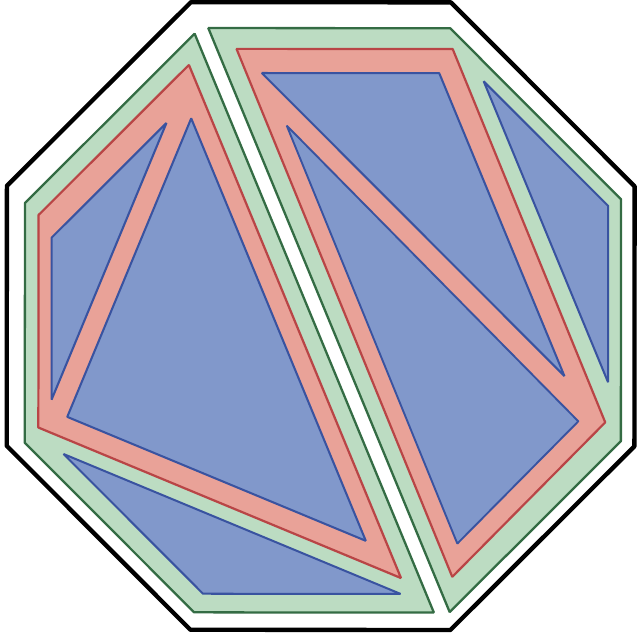}
\end{center}
\end{figure}
\vspace{-.5cm}
\noindent Arkani-Hamed, Figueiredo, and Vaz\~ao  \cite{AFV} proposed a construction of the \emph{cosmohedron} as a solution to this cosmological question. 
The goal of this article is to prove the correctness of that subtle construction. 
Along the way we recover and discover numerous physical, combinatorial, polyhedral, and enumerative properties of Matryoshkas and cosmohedra, that we briefly outline in this introduction.

\subsection{\textsf{Combinatorial properties of the cosmohedron}}

The combinatorial structure of the cosmological wavefunction of $\Tr(\phi^3)$ theory
is described by the \emph{Matryoshkas} of an $(n+2)$-gon $\pentagon_{n+2}$, so these are the objects we wish to model geometrically in the cosmohedron.
We call them Matryoshkas because their nesting structure is reminiscent of, though much richer than, that of the Matryoshka dolls created by S. Malyutin and V. Zvyozdochkin in Russia in the 1890s.

One way of constructing a Matryoshka, ``from the inside out," is to choose a subdivision of $\pentagon_{n+2}$, and then repeatedly choosing a group of adjacent outermost polygons and wrapping them into a larger polygon engulfing them. 
Alternatively, ``from the outside in", a Matryoshka is obtained recursively by choosing a subdivision of $\pentagon_{n+2}$, and then placing a Matryoshka inside each of its polygons.

\begin{figure}[h]
\begin{center}
\includegraphics[width=0.6\linewidth]{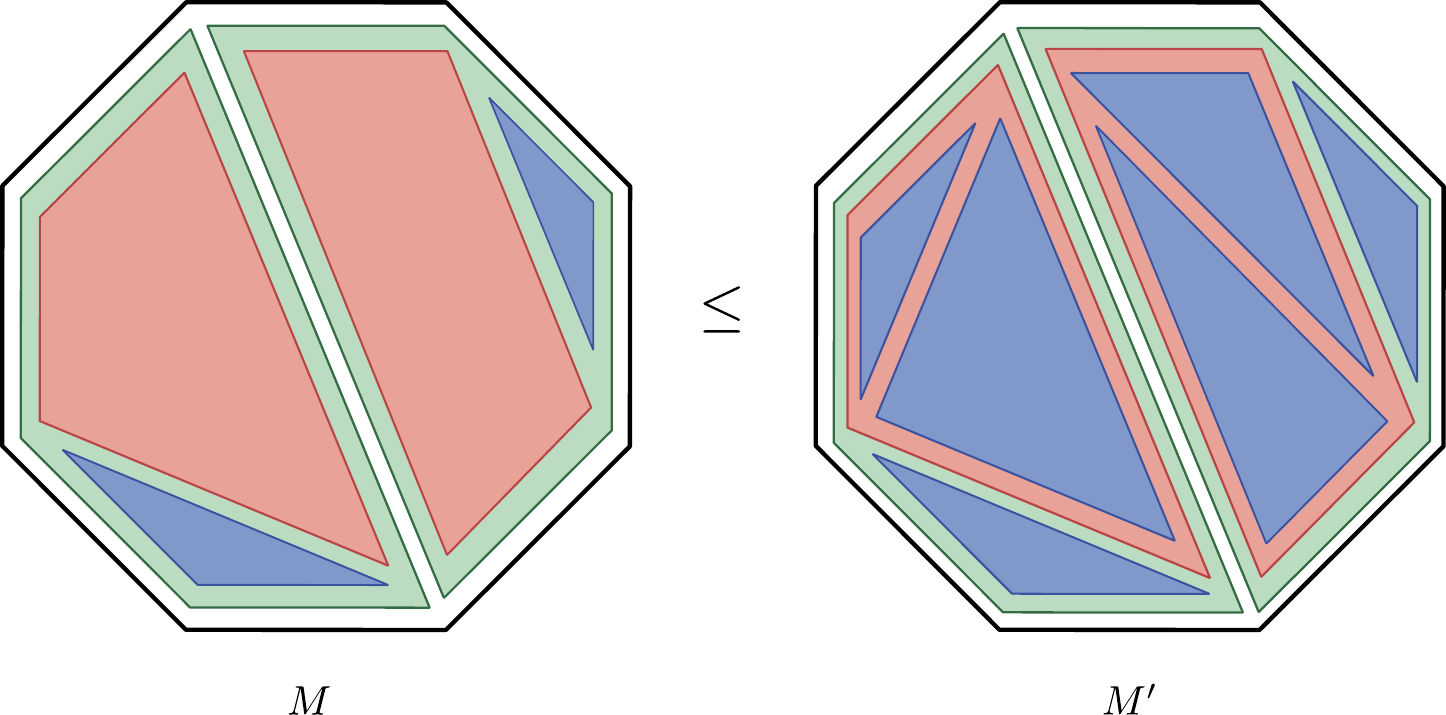}
\end{center}
\caption{An octagonal Matryoshka $M$ and a maximal Matryoshka $M'$ containing it.} \label{fig:doll8}
\end{figure}

Naturally, we can think of a Matryoshka as a set of polygons. This gives a   partial order on the set of Matryoshkas on $\pentagon_{n+2}$ by  \emph{containment}: we have $M \leq M'$ if every polygon of $M$ is also in $M'$.

Arkani-Hamed, Figueiredo, and Vaz\~ao introduced an $(n-1)$-dimensional polytope called the \emph{cosmohedron}
to encode this combinatorial structure; see Definition \ref{def:AFVcosmo}. Our main result is a proof of the following statement from \cite{AFV}.

\begin{theorem}  \label{thm:cosmofan1}
The face lattice of the $(n-1)$-cosmohedron is anti-isomorphic to the poset of  Matryoshkas of an $(n+2)$-gon.
\end{theorem}

\subsection{\textsf{Polyhedral properties of the cosmohedron}}\label{sec:intropoly}

The search for a cosmohedron that models the combinatorics of Matryoshkas is analogous, and in fact closely related, to earlier constructions of polyhedra with a desired combinatorial structure. An important motivating example is the \emph{associahedron}, whose faces are in bijection with the plane trees of a given size.

The cosmohedron is obtained from the associahedron through a careful chiseling procedure:
at each vertex $a_T$ corresponding to the binary tree $T$, we chisel a copy of the \emph{bracket associahedron} $\Br_T$.
The chiselings at the different vertices interact with each other, so they must be carried out very precisely. We must place the right realization of each $\Br_T$ in the correct position, angle, and  size, so that the different bracket associahedra glue correctly into one polytope with the desired combinatorics. The $3$-cosmohedron is shown in Figure \ref{fig:cosmohedron}. The precise construction is given in Section \ref{sec:cosmohedron}.

\begin{figure}[h]
\begin{center}
\includegraphics[width=.95\linewidth]{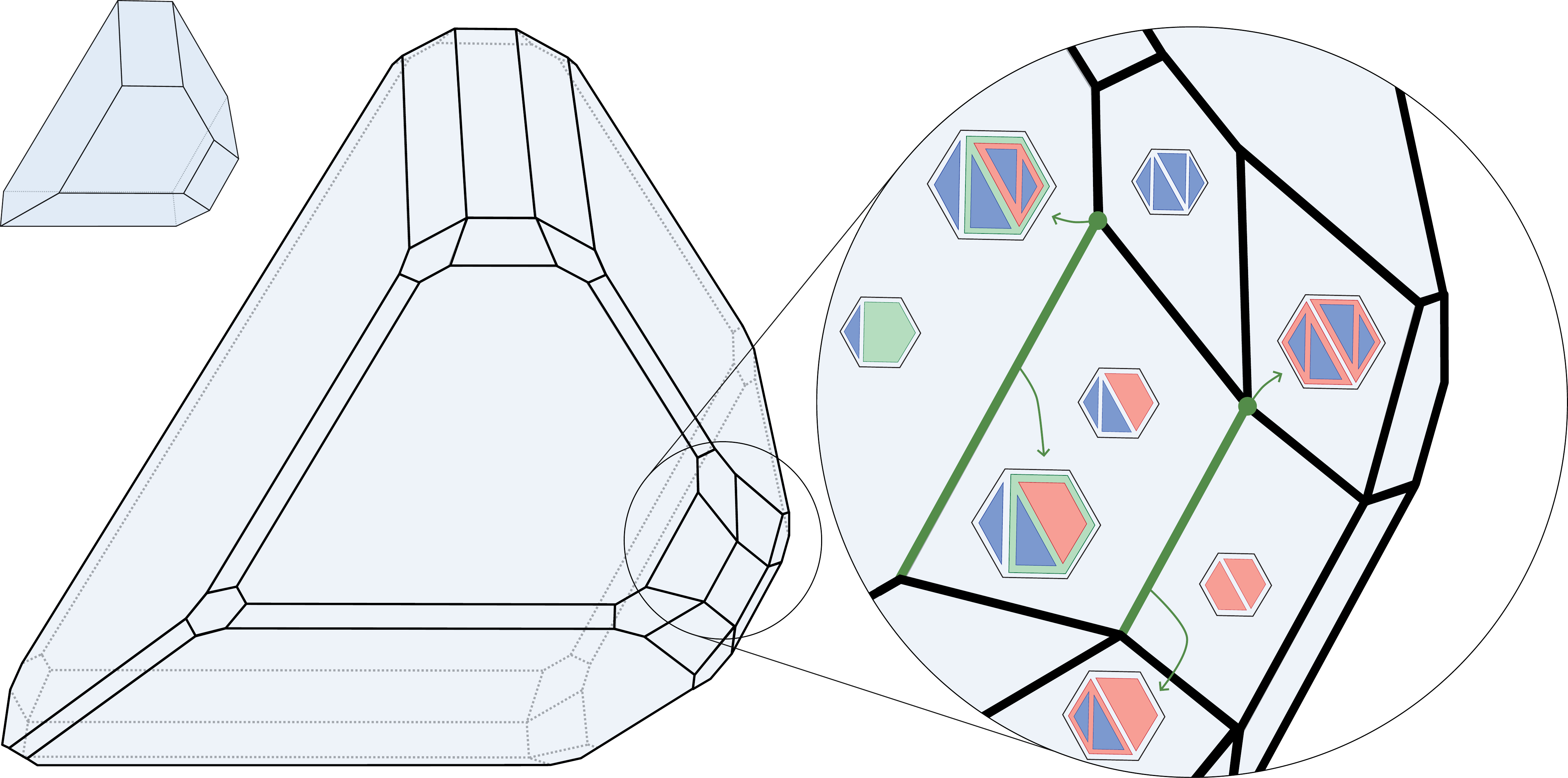}
\end{center}
\caption{The $3$-cosmohedron as a blowup of the $3$-associahedron, and the Matryoshkas labeling some of its faces.  \label{fig:cosmohedron}}
\end{figure}

Arkani--Hamed, Figueiredo, and Vaz\~ao \cite{AFV} proposed a construction of the cosmohedron $\AFV_{n-1}$ as an $(n-1)$-dimensional polytope in $\R^{(n+2)(n-1)/2}$; see Definition \ref{def:AFVcosmo}. This embedding gives the polytope an elegant inequality description, but it does not offer a direct path to prove its correctness, as we explain in Sections \ref{sec:introsubtleties} and \ref{sec:subtleties}.

Our approach is to construct a different embedding $\Cosmo_{n-1}$ of the cosmohedron in (a hyperplane in) $\R^n$ which is closer in spirit to the family of generalized permutahedra. 
To do so, we begin by building its normal fan $\CosmoFan_{n-1}$ in $\R^n/\R$, and proving that it has the correct combinatorial structure of Matryoshkas; we do this in Definition \ref{def:cosmofan} and Theorem \ref{thm:cosmofan3}. 
Next, we construct explicit coordinates for the vertices of $\Cosmo_{n-1}$, and prove that its normal fan is $\CosmoFan_{n-1}$ in Theorem \ref{thm:cosmohedron}.
Since we know the rays of the fan, we can readily use them to compute the inequality description of $\Cosmo_{n-1}$ in Theorem \ref{thm:cosmoineqs}. 
Finally, with inequality descriptions at hand for both of them, we find a linear isomorphism between $\Cosmo_{n-1}$ in $\R^n$ and $\AFV_{n-1}$ in $\R^{(n+2)(n-1)/2}$ in Theorem \ref{thm:samecosmo}. This proves the combinatorial structure of the cosmohedron $\AFV_{n-1}$ as conjectured in \cite{AFV}.

\subsection{\textsf{Enumerative properties of the cosmohedron}}

In Section \ref{sec:enum} we turn to enumerative questions. The facets of the cosmohedron are in bijection with the non-trivial subdivisions of $\pentagon_{n+2}$, which are enumerated by the Hipparchus--Schr\"oder numbers. This sequence is well understood, and is in fact one of the earliest combinatorial sequences, dating to the second century BC; for an interesting historical account, see \cite{EC2, StanleyHipparchus}.

The vertices correspond to the maximal Matryoshkas; we prove that they are enumerated by sequence A177384 in the OEIS. We give a recursive formula for the sequence and a simple polynomial differential equation for the power series:

\begin{theorem} \label{thm:D-algebraic}
The generating function $M(x) = x^2 + 2x^3 + 10x^4 + 72x^5 + \cdots$ for the number of maximal Matryoshkas is characterized by the polynomial differential equation 
\[
M(x) = x^2 + M(x)M'(x).
\]
\end{theorem}
\noindent Thus Matryoshkal numbers form one of the simplest sequences in the notoriously difficult family of \emph{$D$-algebraic power series} \cite{Melczer}. For instance, we do not know how to prove the asymptotic behavior of this sequence; see Conjecture \ref{conj:asymptotic}.

We also give an intriguing description for the $f$-vectors of cosmohedra: their generating function equals its own compositional inverse up to a simple transformation. Consider the $f$-polynomials of the cosmohedral fans $f_n(t)$ as well as the simple transformation $g_n(t) = ((1+t)f_n(t)-1)/t$. The coefficients of $g_n(t)$ are  obtained by adding consecutive coefficients of $f_n(t)$:
\begin{eqnarray*}
&& f_1(t)=1, \, f_2(t)=1+2t,\,  f_3(t)=1+10t+10t^2, \,  f_4(t) = 1 + 44t + 114t^2 + 72t^3, \ldots \\
&& g_1(t)=1, \, g_2(t)=3+2t, \, g_3(t)=11+20t+10t^2, \, g_4(t) = 45+158t+186t^2+72t^3, \ldots 
\end{eqnarray*} 

\begin{theorem} \label{thm:cosmofvector0} The $f$-polynomials of cosmohedra form the unique sequence $\{f_n(t)\}_{n \geq 1}$ such that the power series in $(\R[t])[x]$
\begin{eqnarray*}
&& x - f_1(t)x^2 - f_2(t)x^3 - f_3(t)x^4 - \ldots \\
&& x + g_1(t)x^2 + g_2(t)x^3 + g_3(t)x^3 + \ldots 
\end{eqnarray*}
are compositional inverses, where $g_n(t) = ((1+t)f_n(t)-1)/t$ for $n \geq 1$.
\end{theorem}

We also give a strikingly similar formula for the $f$-vector of the correlatron.

\subsection{\textsf{Elusive qualities of the cosmohedron}}\label{sec:introsubtleties}

As mentioned above, Arkani--Hamed, Figueiredo, and Vaz\~ao's cosmohedron $\AFV_{n-1}$ is an $(n-1)$-dimensional polytope in $\R^{(n+2)(n-1)/2}$, given by a combinatorially rich system of inequalities; see Definition \ref{def:AFVcosmo}. How might one prove its combinatorial structure?
The standard methods don't give a simple approach. 

One might try to prove that $\AFV_{n-1}$ has the correct incidences between vertices and facets. However, we did not have a formula for the vertices of $\AFV_{n-1}$ before; and now that we do thanks to Proposition \ref{prop:cosmovertex}, we see they are rather intricate. Proving the vertex-facet incidences would be a highly non-trivial matter. 

To circumvent this issue, one might instead describe the normal fan of $\AFV_{n-1}$, and prove the corresponding wall-crossing inequalities, as described for example in \cite[Lemma 2.10]{ACEP} or \cite[Theorems 6.1.5–6.1.7]{toricvarieties}. We know at the outset that this computation would be much more cumbersome than usual, because the normal fan is not simplicial. 
Furthermore, in this high dimensional embedding, the $(n-1)$-dimensional normal fan lives in a combinatorially intricate quotient of an $\frac12(n+2)(n-1)$--dimensional space, so the description of the normal fan is  far from unique. These obstacles make this computation very subtle as well.

These are the reasons why we proceeded by giving a different construction of the cosmohedron $\Cosmo_{n-1}$ in $\R^n$, in the spirit of generalized permutahedra, and proved its combinatorial structure. Then, as we somewhat optimistically hoped at the outset, the polytopes $\Cosmo_{n-1}$ and $\AFV_{n-1}$ turned out to be linearly, and thus combinatorially equivalent.

However, proving the correctness of $\Cosmo_{n-1}$ is still a delicate matter.  
Cosmohedra are related to several polytopes in the mathematical literature, some of which are known to be relevant in physical applications, and many of which are generalized permutahedra. These include permutahedra \cite{BooleStott, Schoute}, associahedra \cite{CSZ, Loday, Stasheff, Tamari}, bracket associahedra \cite{Bonatto, Laplante} (which are special cases of graph associahedra \cite{CarrDevadoss, Devadoss, Postnikov}), permutoassociahedra \cite{CastilloLiu2, Kapranov}, and permutonestohedra \cite{Gaiffi}, among others. There are numerous methods to prove the combinatorial structure of such polytopes.

In Section \ref{sec:subtleties}  we explain that the cosmohedron $\Cosmo_{n-1}$ is more subtle than its predecessors in several ways: it is not very symmetric, it is not simple, it depends very delicately on the chiseling, we do not know how to express it as a deformation of an easy polytope, we do not know how to express it as a Minkowski sum of easy polytopes, and its faces factor combinatorially but not geometrically. For these reasons, proving its combinatorial structure requires new ideas, as described above.
%which we described in Section \ref{sec:intropoly}.

\subsection{\textsf{Cosmological properties of the cosmohedron}}

Let us close this introduction by briefly explaining the physical origins of this project. 
The content of this section is not essential to the mathematical understanding of this paper, but it is essential to the discovery of these mathematical objects, and it gives rise to many interesting new directions.

The past decade has seen the discovery of startling and deep new connections between fundamental physics and new mathematical structures in combinatorics and geometry. 
The story began with the most basic and fundamental physical processes in nature--the collision and scattering of elementary particles. Because of quantum mechanics, the outcome of any given scattering process is uncertain; only the probability for different states can be predicted. Given an initial state $i$ and a final state $f$, the probability is $|{\cal A}_{i,f}|^2$ where ${\cal A}_{i,f}$ is the \emph{scattering amplitude}. The textbook method for computing these amplitudes is to
sum over 
\emph{Feynman diagrams} representing the different ways in which the collision events can take place in space and time.

The connection to combinatorics and geometry can most easily and vividly be illustrated in an especially useful toy model, the so-called \emph{Tr($\phi^3$)} theory, for ``colored" scalar particles represented by $N \times N$ square matrices. In the simplest limit (``tree-level"), the amplitude turns into the sum over all planar trivalent (known in physics as \emph{cubic}) tree diagrams, or what is the same, a sum over the triangulations of an $n$-gon. 

A priori to a physicist, there is no obvious combinatorics here: we simply have a collection of diagrams that have to be collected and summed over, in no particular order, to arrive at the final result. But a combinatorialist knows that there is something very special about this particular collection of objects: the triangulations of an $n$-gon are captured by the vertices of a polytope, the associahedron. 

This link turns out to be the tip of a large iceberg of much larger significance to physics beyond ``merely" the combinatorics of diagrams. A particular embedding of the associahedron allows us to extract mathematical expressions for the scattering amplitude from the geometry~\cite{Arkani-Hamed:2017mur}. This particular embedding makes it natural to think of the associahedron as a Minkowski sum of the simplest possible pieces -- a collection of simplices -- and quite remarkably, this fact is intimately related to (and indeed lets us naturally discover from the bottom-up) the generalization from particle scattering to scattering processes in string theory. 
Over the past few years, these developments have led to the discovery of a number of surprising connections between seemingly totally different physical theories~\cite{Arkani-Hamed:2023swr} and, almost as an incidental byproduct, have provided new ways to simplify enormously complicated Feynman diagram calculations, sometimes into a few lines.

The connection to physics also led to the discovery of new mathematical objects. For instance, the associahedron captures all tree diagrams, but what about diagrams that can be drawn on surfaces with arbitrary topology? The consistency of the physical picture demands that geometric objects encoding such diagrams must exist, and they have indeed been described in the physics literature \cite{Arkani-Hamed:2023lbd}.

This emerging connection between physics and combinatorics suggests a radical new picture for the laws of physics, where the notion of particle evolution in spacetime is not primary. Instead, the unifying combinatorial-geometric objects that control how all spacetime processes are ``glued together" play a more fundamental role. This dovetails nicely with an expectation widely shared by most theoretical physicists that the notion of spacetime cannot be fundamental but must eventually be supplanted by more elementary and abstract ideas. The new connection to combinatorics suggests a concrete strategy for pursuing this vision. 

Nowhere is the need for a new picture of physics supplanting spacetime more urgent than in cosmology. Amongst other things, the notion of time itself breaks down at the big bang singularity, so the standard picture of time evolution must likely be abandoned. 

Cosmology also offers us a rich set of physical observables to compute, and in principle measure, that turn out to be interesting generalizations of scattering amplitudes. It is widely believed that the very early history of the universe involved a period of exponentially rapid expansion known as \emph{inflation}, and that quantum fluctuations in this early inflationary phase were blown up to enormous scales, even larger than our observable universe, giving rise to all the ``clumpiness" in matter and all the interesting structures we see in the universe today. The details of these quantum fluctuations are reflected in the statistics of the distribution of mass and energy in the universe today and are known as \emph{cosmological correlation functions}. The cosmological correlators produced during inflation are 
computed in much the same way as particle scattering amplitudes, drawing diagrams for all possible ways the correlations could have been produced. 
But there is an important difference: while for scattering processes a single diagram represents a single elementary process, in cosmology, even a single diagram has much more structure, precisely having to do with the different time-ordering of ``histories" a given diagram can represent. 

If we return to the toy Tr$(\phi^3)$ theory, to compute a cosmological correlator, a given diagram must be further decorated by all the possible \emph{bracketings} of the graph, or equivalently, what we will call the possible \emph{Matryoshkas} associated with the corresponding triangulation. This arises in a very natural way when solving the Schr\"odinger equation $H \Psi = E \Psi$ to compute the wavefunction $\Psi(\phi)$ for $\phi$, from which we determine probabilities and correlation functions. The Schr\"odinger equation for $\Psi(\phi) = {\rm exp}[\-\psi(\phi)]$ can be solved in perturbation theory.  The exponential relating $\Psi(\phi)$ to $\psi(\phi)$ reflects the familiar exponential relationship between all graphs and connected graphs. Every connected cubic diagram with $n$ external legs, where the external legs carry spatial momentum $\vec{k}_i$, makes a contribution to the coefficient of  $\phi(\vec{k}_1) \cdots \phi(\vec{k}_n)$ in the Taylor expansion of $\psi(\phi)$ in powers of $\phi$. 
The contribution of a diagram to $\psi$ depends on ``energies" entering each subgraph -- this is the total sum of the magnitude of the momenta entering the subgraph. In Figure \ref{fig:Wavefunction},  these energies are given by perimeters of subpolygons.

\begin{figure}[h]
    \centering
    \includegraphics[width=\linewidth]{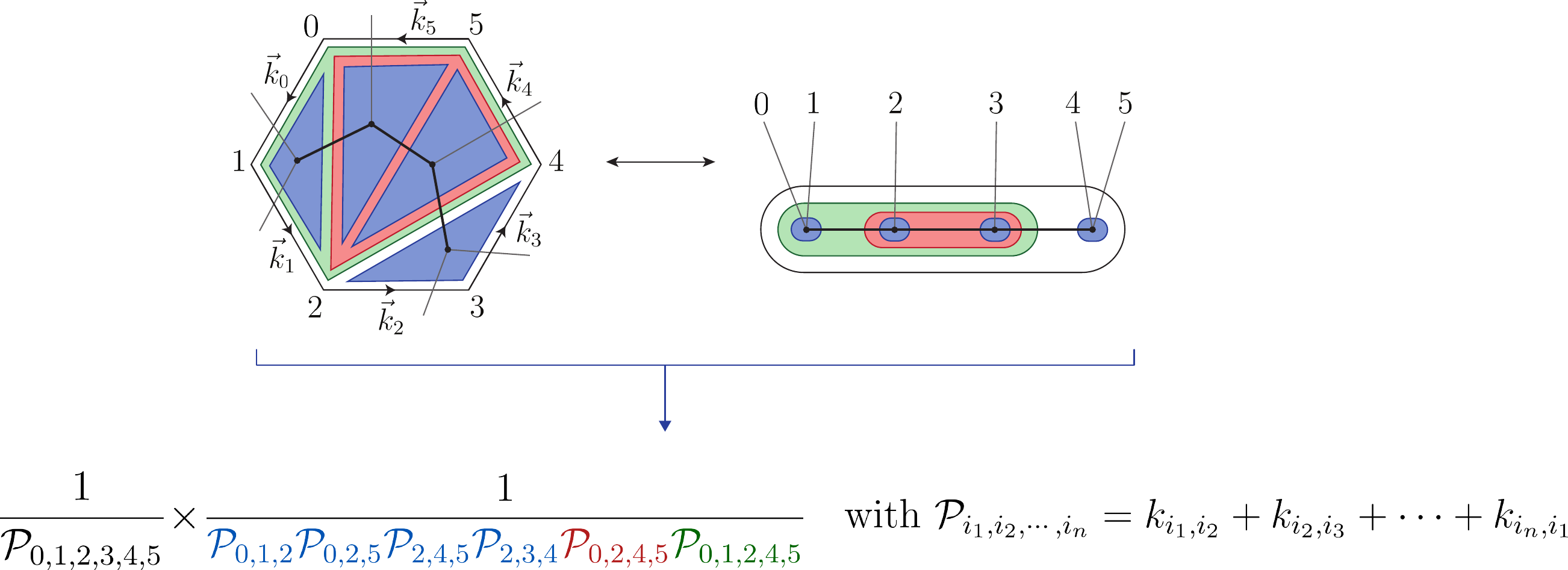}
    \caption{The contribution of a tree graph (or equivalently a triangulation) to the wavefunction is determined by the bracketings  of the graph (or equivalently the Matryoshkas of the triangulation). Here $k_{i,j}$ denotes the length between vertices $i$ and $j$, which  corresponds to the length of the momenta flowing through the edges of the graph, and $\mathcal{P}$ stands for the perimeter of the subpolygons entering the Matryoshka.}
    \label{fig:Wavefunction}
\end{figure}

Written in terms of $\psi$, the Schr\"odinger equation turns into a simple recursion relation:  the contribution from a given graph can be computed by first writing a factor of the inverse of the sum over all energies on the vertices, and then summing over all the ways of deleting an edge from the graph, and adding the energy of the edge to the vertices it connects to in the smaller graph with the edge deleted. We can run this recursion until there are no more edges to be cut, and every term encountered corresponds to a bracketing of the graph -- where to each bracket subgraph we associate the factor given by the total energy entering that subgraph. Phrased in the language of the triangulations, each bracket corresponds to a sub-polygon in the triangulation, to which we associate a factor given by its inverse perimeter. The terms contributing to the wavefunction are then associated with a maximal collection of nested sub-polygons in the triangulation, corresponding to a Matryoshka; see Figure \ref{fig:Wavefunction}.

Given the close connection between scattering amplitudes and cosmological correlators, physics suggested that some analog of the associahedron must exist, capturing not the combinatorics of all triangulations, but the factorially larger combinatorics of all possible Matryoshkas. 
This was the motivation leading to the discovery and definition of cosmohedra, which are polytopes precisely realizing this picture. 

Our main goal in this paper is to better understand the combinatorics and geometry of cosmohedra, provide proofs of many of their claimed properties, and also describe them in a way that makes more natural contact with similar objects that have recently been studied in combinatorics.

\subsection{\textsf{Outlook: $\X$ in $Y$ polytopes and related constructions in physics}}

Apart from its intrinsic mathematical interest, our treatment makes it clear that cosmohedra belong to a wider class of objects whose relevance to physics should extend far beyond cosmology.  We will see that cosmohedra generalize to what we can call \emph{$\X$ in $Y$} polytopes, where we begin with a polytope $Y$ and ``chisel" it in a way controlled by a family $\X$ of polytopes to get a new \emph{chiseled polytope}. At the level of fans, the normal fan of a (simple) polytope $Y$ is canonically refined using the normal fans of the polytopes in $\X$. We expect these $\X$ in $Y$ polytopes to be relevant to many settings in physics,  where, in addition to summing over all diagrams, we must keep track of various properties of subgraphs in each diagram.

An especially important example of this arises in the context of scattering amplitudes beyond ``tree" level. In diagrams with loops, the internal momenta are not determined uniquely and must be integrated over. This is associated with a rich pattern of possible divergences when loop momenta become large (``ultraviolet") or small (``infrared"). For an individual Feynman diagram, these divergences are controlled by the Newton polytopes of certain graph polynomials -- the so-called \emph{Symanzik polynomials} --associated
with the graph. In \cite{Arkani-Hamed:2023swr}, all Feynman diagrams are naturally realized as cones glued together into a complete ``Feynman fan", but this raises a question posed in \cite{Arkani-Hamed:2023swr}: can the combinatorics of the  ``Symanzik polytopes" associated with any graph be encoded in some way inside the full Feynman fan that glues all diagrams together? 
This is exactly what $\X$ in $Y$ polytopes allows us to do. In the Appendix, we sketch this construction in the first interesting case:  the  Symanzik $\U$-polynomials at one-loop level in the Tr($\phi^3$) theory.

\section{\textsf{Matryoshkas, bracketed trees,  and the cosmohedral fan}} \label{sec:fanM}

\subsection{\textsf{Matryoshkas}} \label{sec:Matryoshkas}

Consider a convex $(n+2)$-gon $P$. A \emph{subdivision} $\S$ is a set of sub-polygons whose union is $P$ and whose interiors are pairwise disjoint. A subdivision $\S$ can be naturally identified with the corresponding set $\C$ of non-crossing diagonals of $P$. We say $\S$ \emph{refines} $\S'$, and write $\S \geq \S'$, if the following three equivalent conditions hold:  
every polygon in $\S$ is contained in a polygon in $\S'$, 
every polygon in $\S'$ is a union of polygons in $\S$, and
the corresponding sets of chords satisfy $\C \supseteq \C'$.

A \emph{Matryoshka} or \emph{Russian doll} $M$ is a set of sub-polygons of $P$ such that for any non-minimal polygon $X$ in $M$, the maximal subpolygons $(M_{<X})_{\text{max}}$ of $X$ that are in the Matryoshka $M$  subdivide $X$.  
In particular, the set $M_{\text{min}}$ of minimal polygons in $M$ form the \emph{underlying subdivision} of $M$. We say $M$ \emph{refines} $M'$, and write $M \geq M'$, if the sets of polygons satisfy $M \supseteq M'$.

For example, Figure \ref{fig:Matryoshkas} shows two maximal Matryoshkas $M_1$ and $M_2$ on a hexagon. Removing a quadrilateral from $M_1$ gives a Matryoshka smaller than it. Removing a single triangle from $M_1$ does not give a non-Matryoshka, because the corresponding quadrilateral is not subdivided by its maximal subpolygons.

\begin{figure}[h]
    \centering
    \includegraphics[width=0.6\linewidth]{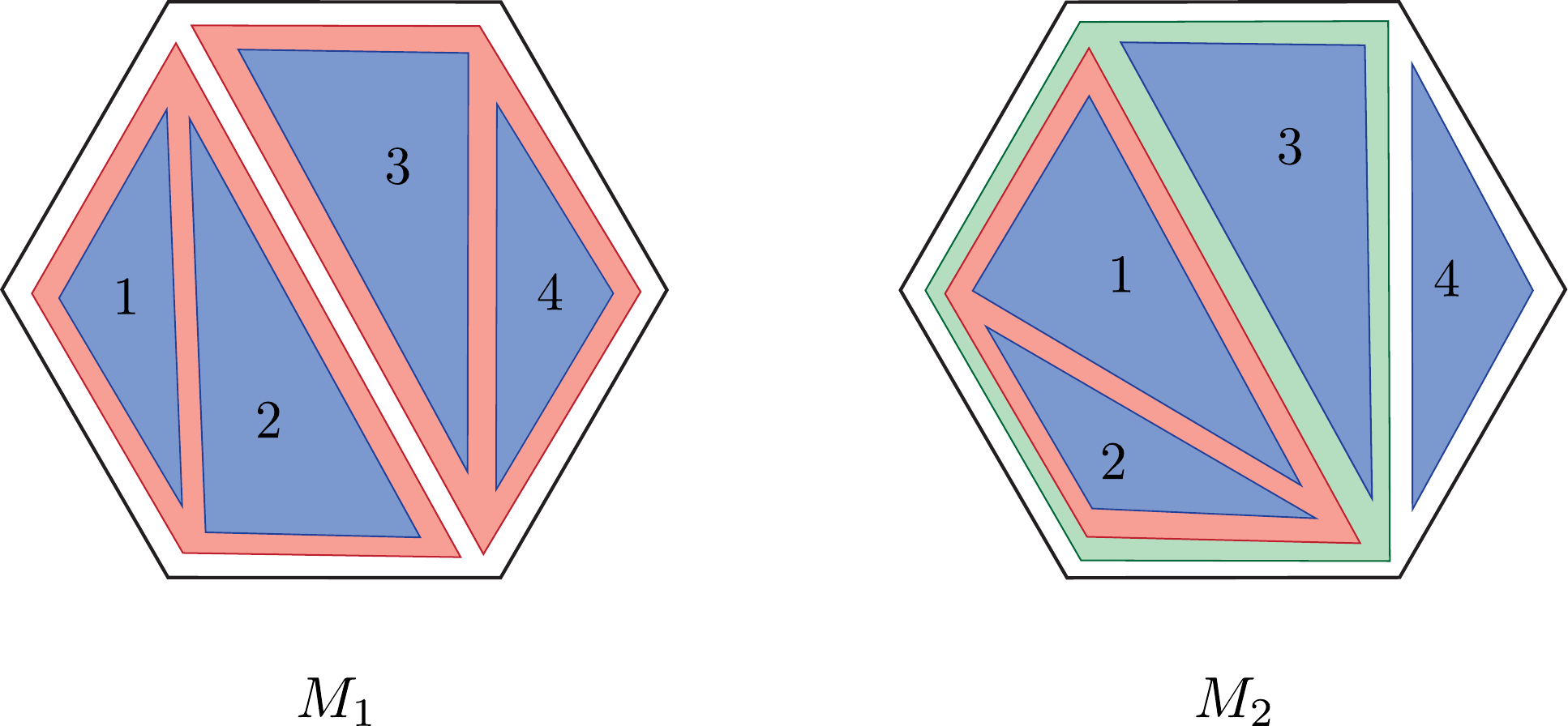}
    \caption{Two maximal Matryoshkas $M_1$ and $M_2$ on a hexagon. \label{fig:Matryoshkas}}
\end{figure}

Recursively, a Matryoshka is either:

\noindent
$\bullet$ $M=\{P\}$, the full polygon, or

\noindent
$\bullet$ $M = \S \cup  \bigcup_{S \in \S} M_S$: a subdivision $\S$ of $P$ and a Matryoshka $M_S$ on each part $S$ of $\S$.

\begin{definition}
The \emph{poset of Matryoshkas} $\Mat_{n-1}$ is the partially ordered set of Matryoshkas on the $(n+2)$-gon ordered by containment, with an additional maximum element $\widehat{1}$.
\end{definition}

\subsection{\textsf{The cosmohedral fan in terms of Matryoshkas}} 

Let us give the vertices of our $(n+2)$-gon $P=\pentagon_{n+2}$ the labels $0, 1, \ldots, n+1$ counterclockwise, with the edge $e$ joining $0$ and $n+1$ placed horizontally at the top.

\smallskip

Consider a triangulation (or subdivision) $\S$ of $P$.

\smallskip

\noindent
\textsf{Triangle (or polygon) labels.} We have a canonical labeling of the $n$ triangles of $\S$ with the vertex labels $1, \ldots, n$: We give the label $v$ to the first triangle of $\S$ that 
one encounters on a straight path from vertex $v$ to edge $e$. (In a subdivision, a $k$-gon gets $k-2$ such labels.)

\smallskip

\noindent
\textsf{Diagonal labels.} 
Each diagonal of $\S$ separates two triangles labeled $v$ and $w$ (or polygons labeled $V$ and $W$), with triangle $v$ between triangle $w$ and edge $e$; we give that diagonal the label $vw^\perp$ (or $VW^\perp$).

\bigskip

\noindent 
Now consider a Matryoshka $M$ of $P$, with underlying triangulation (or subdivision)  $\S$.

\smallskip

\noindent
\textsf{Polygon labels.} 
Each diagonal $vw^\perp$ is contained inside a unique smallest polygon of $M$, which we denote $\pentagon_M(vw^\perp)$. When $M$ is maximal, this is a bijection between the diagonals of $\S$ and the non-minimal polygons of $M$.

\smallskip

\noindent
\textsf{Containment poset.} 
The diagonals $vw^\perp$ are the elements of the \emph{containment poset} $\tau$ of $M$, whose order relations are given by containment of the corresponding polygons: 
$ij^\perp \geq kl^\perp$ if 
$\pentagon_M(ij^\perp) \supseteq 
\pentagon_M(kl^\perp)$.
This poset is a tree, which is binary if $M$ is maximal.

\smallskip

\noindent
\textsf{Matryoshkal cone.} To the Matryoshka $M$ we assign the cone
\begin{equation}\label{eq:cone(M)}
\cone(M) = \{x \in \R^n \, : \, x_i - x_j \geq x_k - x_l  \geq 0 \text{ when $\pentagon_M(ij^\perp) \supseteq \pentagon_M(kl^\perp)$}\}
\end{equation}
naturally associated to the containment poset.

\begin{figure}[h]
\begin{center}
\includegraphics[width=4in]{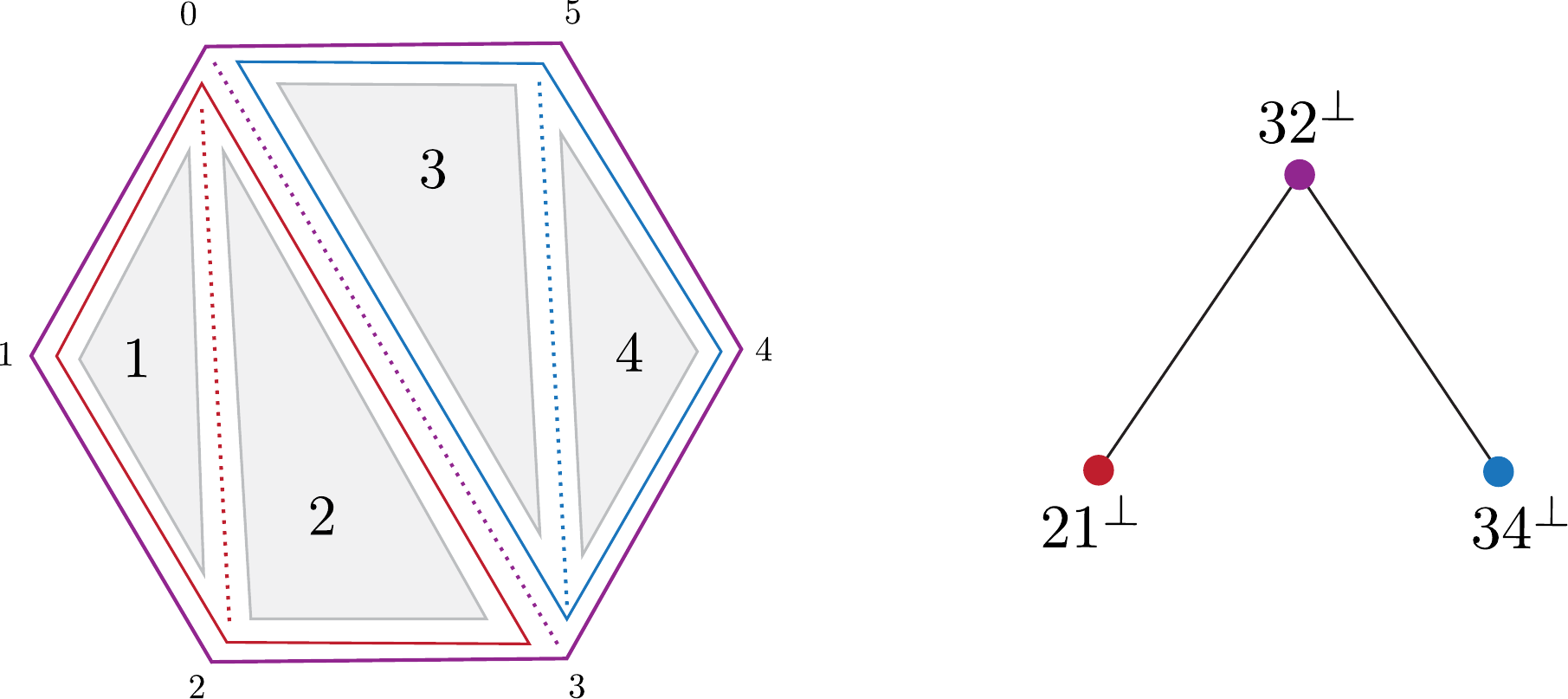}
\end{center}
\caption{The Matryoshka $M_1$ and its containment poset $\tau_1$.\label{fig:Matryoshkaposet}}
\end{figure}

\begin{example} 
In the Matryoshka $M_1$ of Figure \ref{fig:Matryoshkaposet} we have $\pentagon(21^\perp) \subset \pentagon(32^\perp)$ and $\pentagon(34^\perp) \subset \pentagon(32
^\perp)$, so $\cone(M) = \{x \in \R^{4} \, : \, x_3 - x_2 \geq \{x_2-x_1, x_3-x_4\}\geq 0 \}$. 
\end{example}

\begin{definition}\label{def:cosmofan} The \emph{cosmohedral fan} is
\[
\CosmoFan_{n-1} = \{\cone(M) \, : \, M \text{ is a Matryoshka of the $(n+2)$-gon}\}.
\]
\end{definition}

We will prove in Section \ref{sec:prooffan} that this is indeed a fan, and that it has the desired combinatorial structure:

\begin{theorem} \label{thm:cosmofan2}
The cosmohedral fan is a complete fan in $\R^n$.
The map $M \mapsto \cone(M)$ is an order-preserving bijection between the poset $\Mat_{n-1}$ of Matryoshkas of an $(n+2)$-gon and the face lattice $L(\CosmoFan_{n-1})$ of the cosmohedral fan.
\end{theorem}

\subsection{\textsf{Bracketed trees and the cosmohedral fan}} \label{sec:fanT}

Let us give a different combinatorial description of Matryoshkas, which gives rise to a description of the cosmohedral fan in terms of \emph{bracketed trees}.

\noindent
\subsubsection{\textsf{Preliminaries: plane trees, triangulations, and the associahedral fan.}}

A \emph{plane tree} is a rooted tree where the children of each node are linearly ordered from left to right. The plane trees on $n+1$ leaves form a poset $\Tree_{n-1}$ where $T \leq T'$ if $T$ can be obtained from $T'$ by contracting edges. We include an additional maximum element $\widehat{1}$ that makes this poset into a lattice. 

The coatoms of this poset are the \emph{binary trees}, which have $n$ internal nodes and $n+1$ leaves.  It is sometimes convenient to erase the leaves of a binary tree, obtaining a plane tree with $n$ vertices.

If $\S$ is a subdivision of $\pentagon_{n+2}$, then the \emph{dual graph} $T(\S)$ has a vertex for each polygon of $\S$ and an edge between each pair of polygons that share an edge. This is a tree rooted at the polygon adjacent to the top edge $e$. Figure \ref{fig:twoMatryoshkas} shows the binary trees dual to the subdivisions of Matryoshkas $M_1$ and $M_2$.

If $\S$ is a triangulation then $T(\S)$  is a binary tree, and this is a bijection between the triangulations of $\pentagon_{n+2}$ and the trees in $\Tree_{n-1}$; the number of such objects is the Catalan number $C_n = \frac{1}{n+1}\binom{2n}{n}$.

The internal vertices of a plane tree $T$ with $n+1$ leaves are naturally labeled with sets that form a partition of $[n]$: we give the gaps between the leaves the labels $1, 2, \ldots, n$ from left to right, and label each internal node of $T$ with the gap(s) that are visible from it. 
We regard $T$ as the Hasse diagram of a pre-poset on $[n]$ whose root is the maximum element, and define
\[
\cone(T) = \{x \in \R^n/\R\1 \, : \, x_i \geq x_j \text{ for each } i \geq j \text{ in } T_{\interior}\}
\]
where $\1=(1,\ldots,1)$. In particular $x_i=x_j$ in $\cone(T)$ whenever $i$ and $j$ label the same vertex of $T$.

\begin{definition}
The \emph{associahedral fan} is
\[
\AssocFan_{n-1} = \{\cone(T) \, : \, T \text{ is a plane tree on $n+1$ leaves}\}.
\]
\end{definition}

Figure \ref{fig:assocfan} shows the associahedral fan $\AssocFan_3$.

\begin{figure}[h]
    \centering
\includegraphics[width=0.6\linewidth]{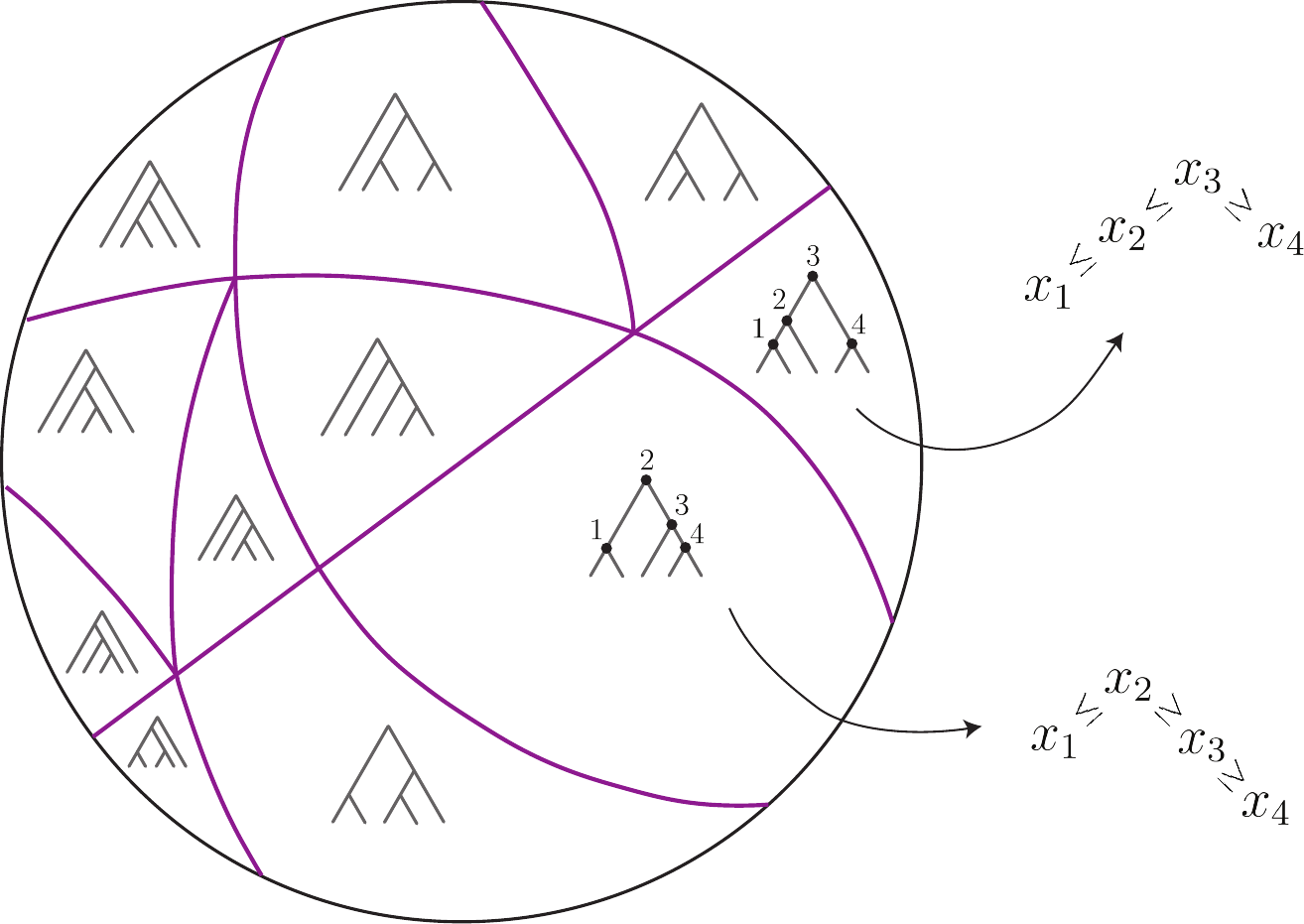}
    \caption{The three-dimensional associahedral fan.}
    \label{fig:assocfan}
\end{figure}

\begin{theorem} \cite{Loday} \label{thm:assocfan2}
The associahedral fan $\AssocFan_{n-1}$ is a complete fan in $\R^n$. 
The map $T \mapsto \cone(T)$ is an order-preserving bijection between the poset $\Tree_{n-1}$ of plane trees with $n+1$ leaves and the face lattice $L(\AssocFan_{n-1})$ of the associahedral fan.
\end{theorem}

\noindent
\subsubsection{\textsf{Preliminaries: graph bracketings and bracket associahedral fans.}}

Let $H=(V,E)$ be a connected graph. A \emph{bracket} of $H$ is a connected subgraph of $H$ with at least one edge. We identify it with its set of edges. 
A \emph{bracketing} of $H$ is a collection $B$ of brackets such that any two brackets are either nested (one is contained in the other) or disjoint (they have no edges or vertices in common). We require every bracketing to contain the full graph $H$.

The \emph{poset of bracketings} $\Br(H)$  is the set of bracketings of $H$ ordered by containment, with an additional maximum element $\widehat{1}$. 

In a bracketing $B$, each edge $e$ is contained in a unique smallest bracket of $B$, which we denote $\bracket_B(e)$ or simply $\bracket(e)$.
This defines a pre-poset $\tree(B)$ on the edge set $E$, where $e \geq_B f$ if $e$ is \emph{outermore} than $f$ in the sense that $\bracket(e) \supseteq \bracket(f)$; or equivalently, every bracket containing $e$ also contains $f$. Because every bracket is contained in a unique smallest bracket, the poset $\tree(B)$ is a tree rooted at a unique maximal element. 
We obtain a cone
\[
\cone_H(B) = \{y \in \R^{E(H)} \, : \, y_e \geq y_f \text{ for each } e \geq_B f\}.
\]
Its dimension is $|B|$.

\begin{definition}
The \emph{bracket associahedral fan} of a graph $H$ is
\[
\BrAssocFan(H) = \{\cone_H(B) \, : \, B \text{ is a bracketing of $H$}\}
\]
\end{definition}

Figure \ref{fig:brassocfan} illustrates the bracket associahedral fan of a path of length 4  with edges $e,f,g$. The fan is three-dimensional, with all cones containing the line $y_e=y_f=y_g$, so we display a 2-dimensional slice.

\begin{theorem}  \cite{AFV, CarrDevadoss}
\label{thm:assocfan}
The bracket associahedral fan $\BrAssocFan(H)$ is a complete fan in $\R^{E(H)}$. 
The map $B \mapsto \cone_H(B)$ is an order-preserving bijection between the poset of bracketings $\Br(H)$ and the face lattice $L(\BrAssocFan(H))$ of the bracket associahedral fan.
\end{theorem}

\begin{figure}[h]
    \centering
\includegraphics[width=0.6\linewidth]{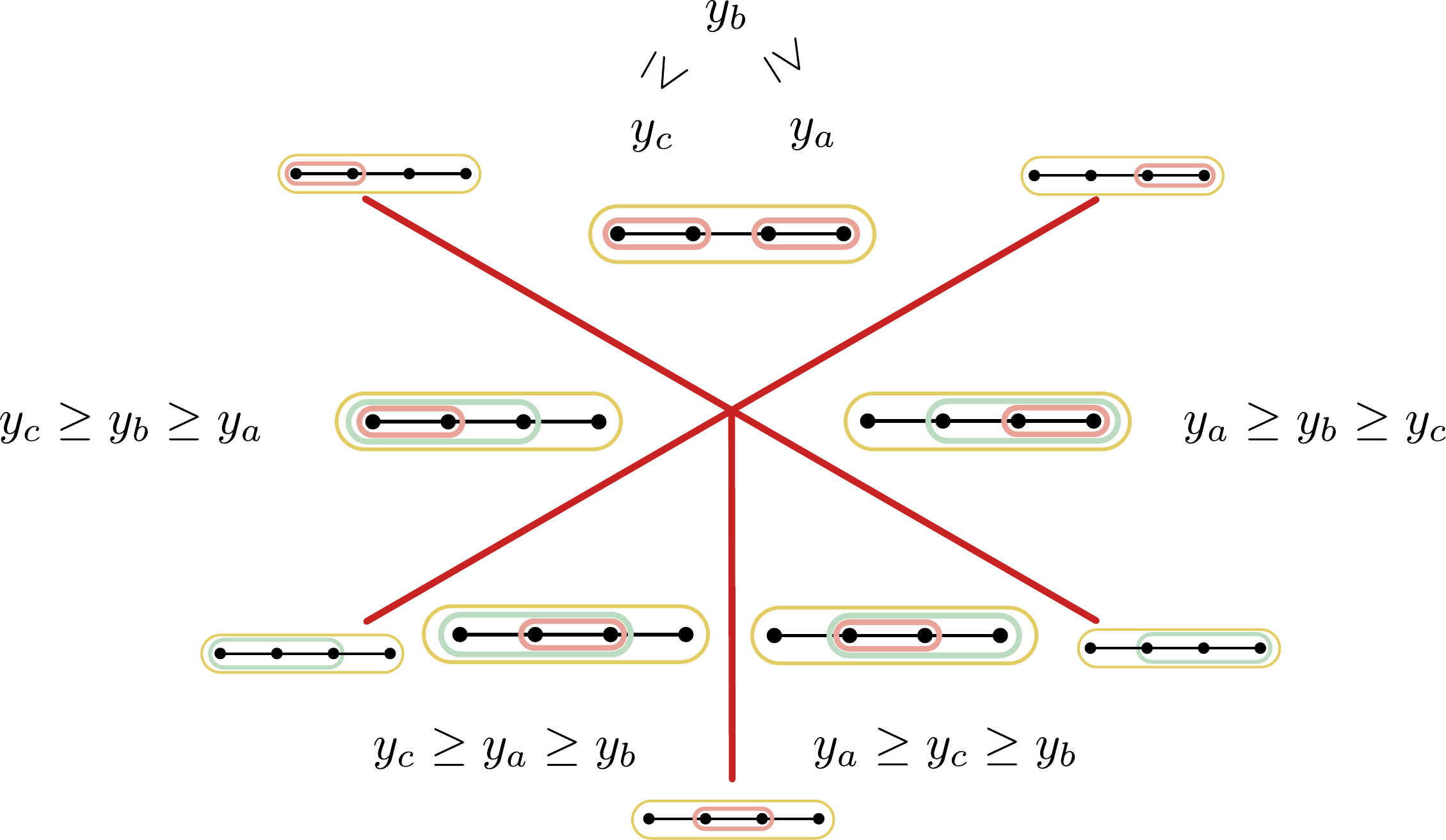}
    \caption{The bracket associahedral fan of a path.}
    \label{fig:brassocfan}
\end{figure}

\noindent
\textsf{Bracket associahedral fans and graph associahedral fans.}
The \emph{line graph} $L(H)$ of $H=(V,E)$ is the graph of edge adjacencies of $H$: it has a vertex for each edge $e$ of $H$, and the vertices corresponding to edges $e,f$ are connected in $L(H)$ if $e$ and $f$ share a vertex in $H$. The bracketings of $H$ are in order-preserving bijection with the tubings of  $L(H)$ in the sense of \cite{CarrDevadoss}, so 
\[
\{\text{bracket associahedral fans}\} \subsetneq \{\text{graph associahedral fans}\}.
\] 

\begin{figure}[h]
    \centering
\includegraphics[width=0.4\linewidth]{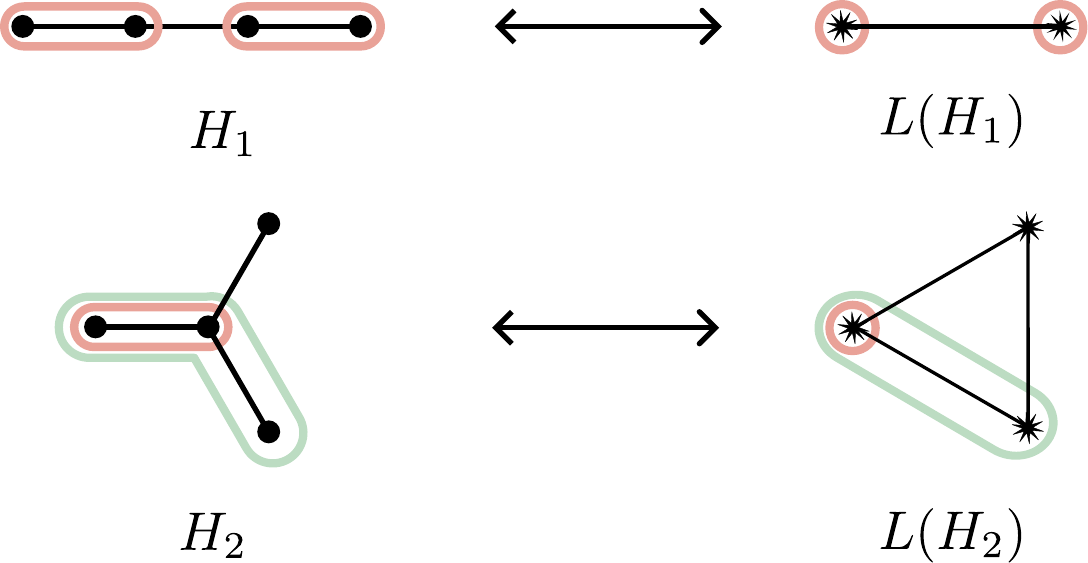}
    \caption{Bracketings of $H$ correspond to tubings of $L(H)$.}
    \label{fig:LineGraph_Bracket}
\end{figure}
The graph associahedral fan of the claw $K_{1,3}$ -- studied and depicted in 
\cite[Figure 3]{ArdilaReinerWilliams}
in its manifestation as the Dynkin diagram $D_4$ -- is a three-dimensional fan with 10 rays, corresponding to the 10 tubes of the graph. The reader may verify that this is the only graph on four vertices that has 10 tubes. Since the claw is not a line graph, this is not a bracket associahedral fan. It is then natural to ask:
\begin{center}
Which graph associahedral fans are bracket associahedral fans?
\end{center}

\smallskip

\noindent
\subsubsection{\textsf{The cosmohedral fan in terms of bracketed trees.} }

A \emph{bracketed tree} $(T,B)$ consists of a plane tree $T$ and a bracketing $B$ of $T$. Let us define the cone of a bracketed tree $(T,B)$ with $n+1$ leaves to be
\[
\cone(T,B) = \{x \in \R^n \, : \, x_i - x_j \geq x_k - x_l  \geq 0 \text{ when $\bracket_B(ij) \supseteq \bracket_B(kl)$}\}.
\]
Note that the inequalities of $\cone(T,B)$ correspond to the order relations of the tree poset $\tree(B)$, as illustrated in Figure \ref{fig:twoMatryoshkas}. 

The goal of this section is to show that the cones of the bracketed trees on $n+1$ leaves are precisely the cones of the cosmohedral fan $\CosmoFan_{n-1}$. We begin with a technical lemma.

\begin{lemma} \label{lem:dim}
For any bracketed tree $(T,B)$ we have
\[
\cone(T,B) \subseteq \cone(T), \qquad \dim \cone(T,B) = \dim \cone(T) = |V(T)_{int}|+1.
\]
\end{lemma}

\begin{proof}
The inclusion and the formula for $
\dim \cone(T)$ are clear from the definitions. To prove that $\cone(T,B)$ has the same dimension, we first claim that for $x \in \cone(T,B)$ we have that $x_i=x_j$ if $i$ and $j$ label the same vertex of $T$. To see this, note that for an offspring  $k$ we have $\bracket_B(ik) = \bracket_B(jk)$ so $x_i-x_k = x_j-x_k$ (or analogously for the parent $k$, if they do not have offspring). 
Setting these equalities aside, the remaining inequalities of $\cone(T,B)$ are given by the acyclic poset $\tree(B)$, so they imply no further equalities.
\end{proof}

\begin{figure}[h]
    \centering
    \includegraphics[width=5.5in]{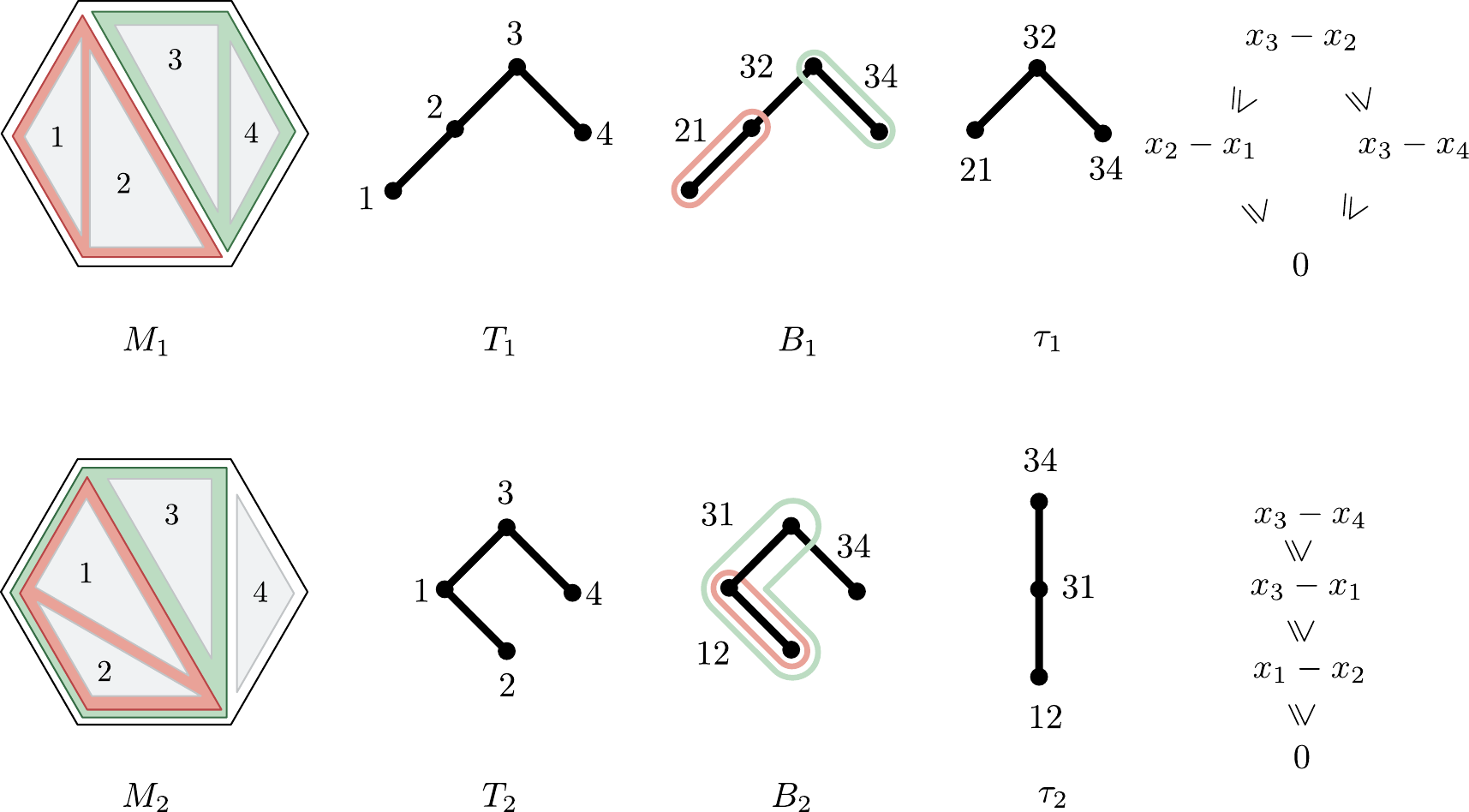}
    \caption{Two Matryoshkas, their corresponding bracketed trees, and their cones.}
    \label{fig:twoMatryoshkas}
\end{figure}

\noindent \textbf{\textsf{The bijection with Matryoshkas.}} 
For a Matryoshka $M$, let $\S=M_{\min}$ be the set of inclusion-minimal polygons in $M$, which form the underlying subdivision of $M$. 
\begin{enumerate}
\item The \emph{rooted planar tree} $T_M$ of $M$ is the dual graph of $\S$, rooted at the top polygon containing edge $e$.
\item Each non-minimal polygon $P$ in $M$ is a disjoint union of minimal polygons in $M_{\min}$, whose dual graph is a bracket of the tree $T_M$,
denoted $\bracket(P)$. The \emph{bracketing} $B_M$ of $M$ is $B_M = \{\bracket(P) \, : \, P \in M_{not\,  min}\}$.
\end{enumerate}

\begin{lemma} \label{lem:bijection}
The map $M \mapsto (T_M, B_M)$ is a bijection between the  Matryoshkas on an $(n+2)$-gon and the bracketed trees $\BrTree_{n-1}$ with $n+1$ leaves.   
\end{lemma}

\begin{proof}
The Matryoshka condition on $M$ implies that $B_M$ is indeed a bracketing of $T_M$. Conversely, for any pair $(T,B)$, the plane tree $T$ is dual to a subdivision $\S$ of the $(n+2)$-gon, and each bracket in $B$ is dual to a polygon $P$ that is the union of polygons in $\S$. The bracketing condition on $B$ implies that these polygons form a Matryoshka. 
\end{proof}

\begin{lemma} 
Under the bijection of Lemma \ref{lem:bijection}, the cone of the Matryoshka $M$ equals the cone of the bracketed tree $(T_M,B_M)$:
\[
\cone(M) = \cone(T_M, B_M). 
\]
\end{lemma}

\begin{proof}
This follows readily from the definitions.
\end{proof}

\begin{corollary}
We have
\[
\CosmoFan_{n-1} = \{\cone(T,B) \, : \, (T,B) \text{ is a bracketed tree with $n+1$ leaves}\}.
\]
\end{corollary}

\bigskip

\noindent \textbf{\textsf{The poset of bracketed trees.}} 
Consider a graph $G$ and a bracketing $B$. If $\beta$ is a bracket of $B$, the \emph{deletion} is the bracketing $B-\beta$ is a bracketing of $G$. If $\mu$ is a minimal bracket of $B$, the \emph{contraction} is the bracketing $B/\mu :=  \{\beta - \mu \, : \, \beta \in B, \beta \neq \mu\}$ of the tree $T/\mu$ where edge $\mu$ is contracted -- so the edge is deleted and its vertices are identified; see Figure \ref{fig:bracketing}. If $E$ is a downset in (the tree pre-poset of) $B$, we can define $B/E$ by successively contracting the brackets in $E$ in a non-decreasing order.

\begin{figure}[h]
    \centering
    \includegraphics[width=2.5in]{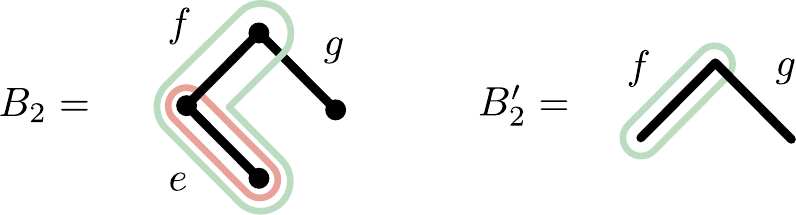}
    \caption{A contraction of a bracketing.}
    \label{fig:bracketing}
\end{figure}

The \emph{poset of bracketed plane trees} with $n+1$ leaves, denoted $\BrTree_{n-1}$, is defined by deletion and contraction as follows. Locally, we have a cover relation $(T,B) \lessdot (T',B')$ if either:\\ 
(i) $T = T'$ and $B = B' - \beta'$ for a bracket $\beta'$ of $B'$, or \\
(ii) $T=T'/\mu'$ and $B = B'/\mu'$ for a minimal bracket $\mu'$ of $B'$.\\ 
Globally, we have a poset relation $(T,B) \leq (T',B')$ if $T = T'/E$ for a set of edges $E$ that is a downset of the tree poset of $B'$,  and  $B \subseteq B'/E$.
We augment this poset with a maximum element $\widehat{1}$.

\begin{proposition} \label{prop:BrTree}
The bijection $M \mapsto (T_M, B_M)$ of Lemma \ref{lem:bijection}, 
is a poset isomorphism between the lattice of Matryoshkas $\Mat_{n+2}$ and the poset of bracketed trees $\BrTree_{n-1}$.
\end{proposition}

\begin{proof}
In a cover relation $M \lessdot M'$ of Matryoshkas, $M$ is obtained from $M'$ by either: \\
\noindent (i) removing a non-minimal polygon $P$, which corresponds to leaving $T_M = T_{M'}$ fixed and making $B_M = B_{M'} - \bracket(P)$, or \\z
\noindent (ii) removing the minimal polygons $P_1, \ldots, P_k$ that subdivided a next-to-minimal polygon $P$ in $M'$, so $P$ becomes minimal in $M$; this corresponds to contracting the minimal $\bracket(P)$ to get $(T_M, B_M) = (T_{M'}/\bracket(P), B_{M'}/\bracket(P))$. 

This proves the poset isomorphism. 
\end{proof}

\noindent
\section{\textsf{Proof of correctness of the cosmohedral fan}\label{sec:prooffan} }

\subsection{\textsf{Positive bracket associahedral fans}}
The \emph{positive bracket associahedral fan} of a graph $H$ is the intersection
\[
\BrAssocFan_{\geq 0}(H) = \BrAssocFan(H) \cap \R_{\geq 0}^{E(H)}
\]
 of the bracket associahedral fan with the positive orthant, with the inherited polyhedral structure.

\begin{figure}[h]
    \centering
    \includegraphics[width=4in]{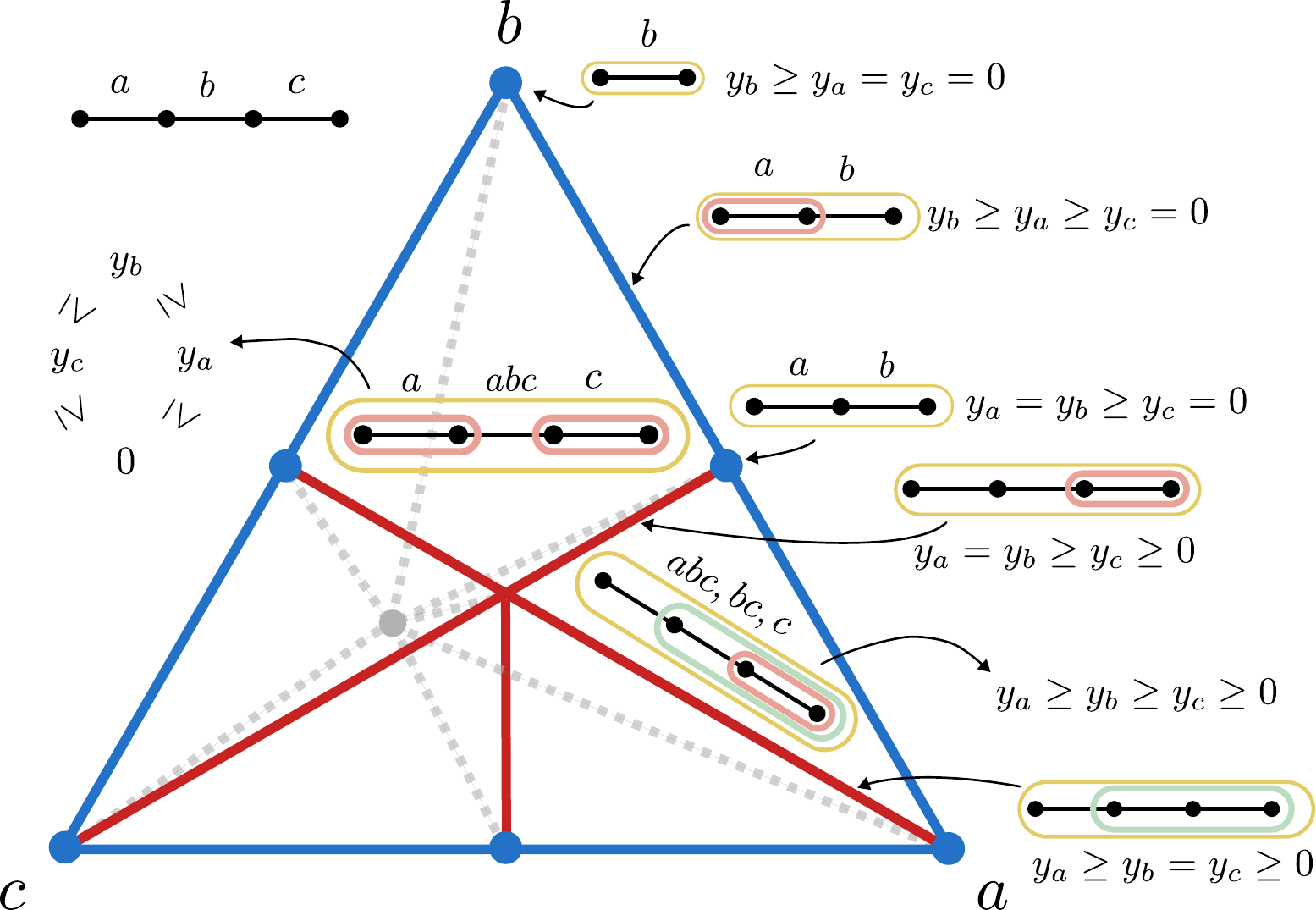}
    \caption{The positive bracket associahedral fan of a path.}
    \label{fig:posbrassocfan}
\end{figure}

\begin{example}
    Figure \ref{fig:posbrassocfan} illustrates the positive bracket associahedral fan of a path of length 4 with edges $e,f,g$; it is the intersection of Figure \ref{fig:brassocfan} -- which we should recall is three-dimensional -- with the positive orthant. The maximal cones are in bijection with those of Figure \ref{fig:brassocfan}, but there are new lower-dimensional cones, corresponding to the bracketings $B$ of the contractions $G$ of the path. We label a few of the faces with a picture of the \emph{pre-bracketing} $(G,B)$, and with the list $B$ of brackets, which is enough to recover $G$. We now prove this description holds in general.
\end{example}

A \emph{pre-bracketing} is a bracketing of a contraction of $H$; that is, a pair $(G,B)$ consisting of a contraction $G$ of $H$ and a bracketing $B$ of $G$.
When $H$ is fixed, we simply write $B$ instead of $(G,B)$, since $G$ is determined automatically as the maximum bracket in $B$.

The \emph{poset of pre-bracketings} $\PreBr(H)$ is again defined by deletion and contraction:
$B \lessdot B'$ if $B$ is obtained from $B'$ by either (i) removing a bracket or (ii) contracting a minimal bracket.

We define the \emph{positive cone of a pre-bracketing} $B$ of $H$ to be
\[
\cone_H(B)_{\geq 0} = 
\{y \in \R^{E(H)} \, : \, y_e \geq y_f \geq 0 \text{ for each } e \geq_B f\, , \,\, \, y_e=0 \text{ for each } e \text{ contracted}\}.
\]

\begin{proposition}   \label{prop:BrAssocFan}
For any graph $H$, the map $B \mapsto \cone_H(B)_{\geq 0}$ is an order-preserving bijection between the poset $\PreBr(H)$ of pre-bracketings of $H$ and the face lattice \newline $L(\BrAssocFan_{\geq 0}(H))$ of the positive bracket associahedral fan of $H$.
\end{proposition}

\begin{proof}
A face of the positive bracket associahedral fan is obtained by intersecting a face $\cone_H(B)$ of $\BrAssocFan(H)$ for a bracketing $B$ with a face $\R_{\geq 0}^{E}$ of $\R_{\geq 0}^{E(H)}$ for a subset $E \subseteq E(H)$. 
Imposing the equations $y_f = 0$ for $f \notin E$ on $\cone_H(B)$ also implies $y_{f'} = 0$ for any $f' <_B f$; so if we let $\overline{F}$ be the downset of (the tree poset of) $B$ generated by $F = E(H)-E$, then 
\begin{equation} \label{eq:brackettoprebracket}
\cone_H(B) \cap \R_{\geq 0}^{E} = \cone_H(B/\overline{F})_{\geq 0}.
\end{equation}
Every pre-bracketing of $H$ arises in this way, proving the description of the cones of $\BrAssoc_{\geq 0}(H)$. 

Now, for a pre-bracketing $B$ of $H$, a facet of $\cone_H(B)_{\geq 0}$ is obtained by either (i) setting $y_e=y_f$ for $e \gtrdot f$ in $\tree(B)$,  which has the effect of deleting the bracket of $f$, or (ii) setting $y_f=0$ for a minimal $f$ in $\tree(B)$, which has the effect of contracting the  bracket of $f$. 
The poset isomorphism follows.
\end{proof}

\subsection{\textsf{The cosmohedral fan and positive bracket associahedral fans}}

Now we describe the cosmohedral fan in terms of positive bracket associahedral fans.

\begin{lemma}\label{lem:f_T}
Let $T$ be a plane rooted tree with $n+1$ leaves, whose internal vertices are canonically labelled by $[n]$. Label each edge in the order $ij$ where $i$ is closer to the root than $j$. 
We have inverse maps:

\[
\begin{aligned}
f_T : \R^{E(T)} &\overset{\cong}\longrightarrow \R^{n}/\R\1
&\qquad
f_T^{-1} : \qquad \qquad \R^{n}/\R\1 &\overset{\cong}\longrightarrow \R^{E(T)} \\
e_{ij} &\longmapsto \sum_{k\in (T-ij)_i} e_k + \R
&\qquad
(x_v)_{v\in V(T)}+\R\1 &\longmapsto (x_i-x_j)_{ij\in E(T)} .
\end{aligned}
\]
where $(T-ij)_i$ and $(T-ij)_j$ are the connected components of the forest $T-ij$ containing $i$ and $j$ respectively. Furthermore, 

\begin{enumerate}

\item 
the map $f_T :  \R^{E(T)} \overset{\cong}\longrightarrow \R^{n}/\R\1$ is an isomorphism of vector spaces,

\item
the restriction $f_T :  \R_{\geq 0}^{E(T)} \overset{\cong}\longrightarrow \cone(T)$ is an isomorphism of cones, and
\item
the restriction $f_T: \cone_T(B)_{\geq 0} \overset{\cong}\longrightarrow \cone(T,B)$ is an isomorphism of cones for any bracketing $B$ of tree $T$.
\item
the restriction $f_T: \cone_T(B)_{\geq 0} \overset{\cong}\longrightarrow \cone(T(B),B)$ is an isomorphism of cones for any pre-bracketing $B$ of tree $T$, where $T(B) \leq T$ is the tree supporting $B$.
\end{enumerate}
\end{lemma}

\begin{proof}
First, we check that these maps are indeed inverse to each other. For each basis vector $e_{ij}$ corresponding to edge $ij$ of $T$, the vector $f_T(e_{ij})$ has coordinates equal to $1$ on component $(T - ij)_i$ and $0$ on component $(T - ij)_j$. Therefore $f_T^{-1}(f_T(e_{ij}))$ is only non-zero on the edge $ij$ -- whose vertices lie in different components -- where it equals 1. Thus $f_T^{-1}(f_T(e_{ij}))=e_{ij}$ indeed. Since both spaces are $(n-1)$-dimensional, this means $f_T$ is  an isomorphism and $f_T^{-1}$ is its inverse, proving 1. 

To show 2, notice that $y=f_T^{-1}(x)$ is nonnegative if and only if $x$ satisfies the inequalities of $\cone(T)$. 
For 3., observe that $y$ satisfies the additional inequalities $y_{ij} \geq y_{kl} \geq 0$ of $\cone_T(B)_{\geq 0}$ if and only if $x$ satisfies the inequalities $x_i - x_j \geq x_k - x_l \geq 0$ of $\cone(T,B)$ for $ij \geq_B kl$. 
For 4. notice that $y$ satisfies $y_{ij}=0$ for each contracted edge $ij$  if and only if $x$ satisfies $x_i=x_j$ when $i$ and $j$ label the same vertex in $T(B)$; these are precisely the equalities that hold in $\cone(T(B),B)$, as we explained in the proof of Lemma \ref{lem:dim}. 
\end{proof}

\begin{example}
Consider the Matryoshka $M_1$ and its corresponding tree $T_1$ with edges $e=21, f=32, g=34$, and bracketing $B_1=\{e,efg, g\}$, as shown in Figure \ref{fig:twoMatryoshkas}. Then
\[
\cone_{T_1}(B_1)_{\geq 0} = \{y \in \R^{efg} \, : \, y_f \geq \{y_e, y_g\} \geq 0\} \subset \R^{E(T)},
\]
and its image under the map $f_{T_1}$ is
\[
\cone(T_1,B_1) = \{x \in \R^{4}/\R\1 \, : \, x_3-x_2 \geq \{x_2-x_1, x_3-x_4\} \geq 0\} \subset \R^n/\R\1
\]
as illustrated in Figure \ref{fig:faninfan}.
\end{example}

\begin{figure}[h]
    \centering
    \includegraphics[width=0.6\linewidth]{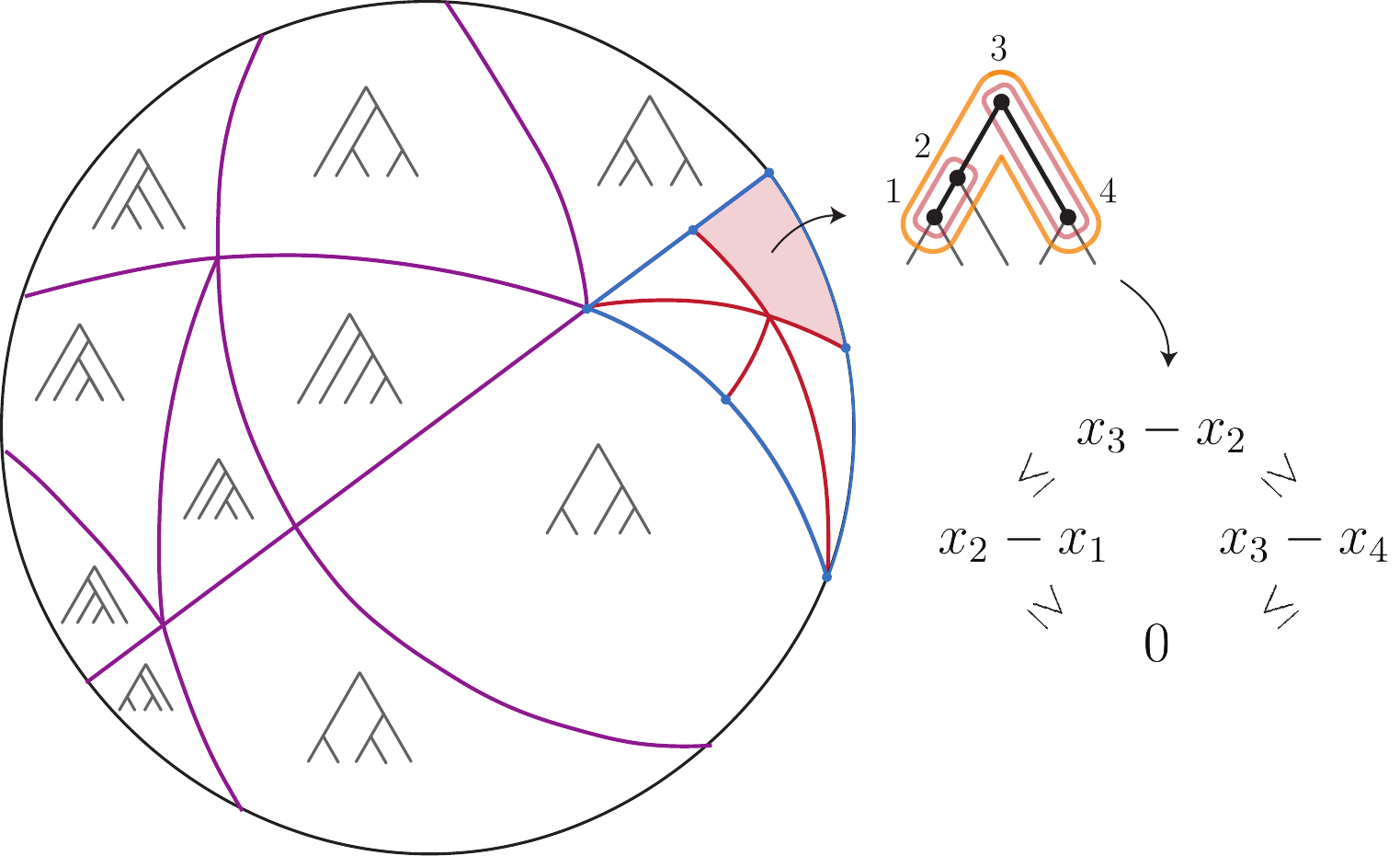}
    \caption{The linear map $f_{T_1}$ embeds the four-sided $\cone_{T_1}(B_1)_{\geq 0}$ of $\BrAssocFan_{\geq0}(T_1)$ in Figure \ref{fig:posbrassocfan} into the $\cone(T_1)$ of $\Assoc_{n-1}$ in Figure \ref{fig:assocfan}. The result is the cone $f_{T_1}(\cone_{T_1}(B_1)_{\geq 0}) = \cone(T_1,B_1)$ of $\Cosmo_{n-1}$.}
    \label{fig:faninfan}
\end{figure}

For the moment, we only know $\CosmoFan_{n-1}$ to be the set of Matryoshkal cones. Our next goal is to show that it is indeed a polyhedral fan, with the correct combinatorial structure. We begin with a simple technical lemma.

\begin{corollary}
We have
\[
\CosmoFan_{n-1} = \bigcup_{T} f_T(\BrAssocFan_{\geq 0}(T))
\]
where the union is over all binary trees $T$ with $n+1$ leaves.
\end{corollary}

\begin{proof}
Any cone of $\CosmoFan_{n-1}$ is of the form $\cone(M) = \cone(T,B)$ for a bracketed tree $(T,B)$. If we choose any binary $T^+ \geq T$, then $B$ is a pre-bracketing of $T^+$, and Lemma \ref{lem:f_T}.4 shows that this cone is in $f_{T^+}(\BrAssocFan_{\geq 0}(T^+))$.
\end{proof}

Notice in particular that the union above is not disjoint.

\subsection{\textsf{The cosmohedral fan is a complete fan}}

\begin{theorem}  \label{thm:cosmofan3} 
The cosmohedral fan $\CosmoFan_{n-1}$ is a complete fan in $\R^n/\R\1$.
The map $M \mapsto \cone(M)$ is an order-preserving bijection between the Matryoshkas of an $(n+2)$-gon and the faces of $\CosmoFan_{n-1}$.
\end{theorem}

\begin{proof}
To prove that the cones intersect properly, let us first analyze how a Matryoshkal $\cone(T_1,B_1)$, which sits in an associahedral $\cone(T_1)$ of the same dimension, intersects a smaller associahedral $\cone(T) \subseteq \cone(T_1)$. Consider any maximal $T^+ \geq T_1$. Then $B_1$ is a pre-bracketing of $T^+$ and we have
 \begin{eqnarray*}
\cone(T_1,B_1) \cap \cone(T) 
&=& f_{T^+}(\cone_{T^+}(B_1)_{\geq 0}) \cap f_{T^+}(\R_{\geq 0}^{E(T)})  \\
&=& f_{T^+}(\cone_{T^+}(B_1)_{\geq 0} \cap \R_{\geq 0}^{E(T)}) \\
&=& f_{T^+}(\cone_{T^+}(B_1/\overline{F})_{\geq 0}) \quad  \text{ for } F=E(T^+)-E(T), \text{ by } \eqref{eq:brackettoprebracket}\\
&=& \cone(T'_1, B'_1) \quad \text{ for } B_1' = B_1/\overline{F}
\end{eqnarray*}
where $T'_1 = T^+/\overline{F} = T/(\overline{F}-F) \leq T$ is the tree supporting $B_1'$. 

For example, in the top row of Figure \ref{fig:intersection} we compute 
$\cone(M_1) \cap \cone(T)$
for the Matryoshka $M_1=(T_1,B_1)$ of Figure \ref{fig:twoMatryoshkas} and the tree $T$ of the shown subdivision. Here $T^+=T_1$, $F=E(T^+)-E(T) = \{21\}$, and the downset in $\tau_1$ is $\overline{F} = F = \{21\}$. Thus we obtain $(T_1', B_1')$ by contracting edge 21 from $(T_1,B_1)$. The resulting inequalities are $x_3-x_{21} \geq x_3-x_4 \geq 0$, where $x_{21}$ denotes the common value of $x_2=x_1$. In the second row we compute $\cone(M_2) \cap \cone(T)$ similarly.

\begin{figure}[h]
    \centering
    \includegraphics[width=5.5in]{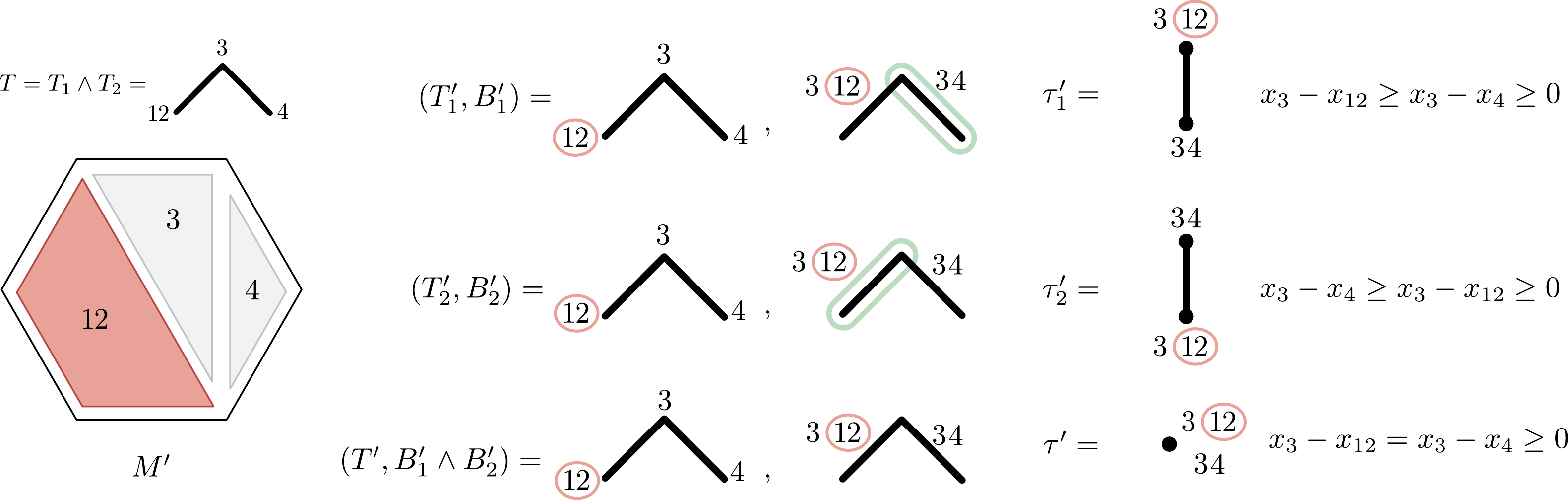}
    \caption{The intersection of the two Matryoshkal cones $M_1$ and $M_2$ described in Figure \ref{fig:twoMatryoshkas}. We begin by computing $T=T_1 \wedge T_2$ and use it to find $T_1', T_2'$ and the inherited bracketings $B_1'$ and $B_2'$. From there we find $B'=B_1' \wedge B_2'$. Then $M'$ is the Matryoshka of $T'$ and $B'$.}
    \label{fig:intersection}
\end{figure}

Now consider two Matryoshkas $M_1=(T_1,B_1)$ and $M_2=(T_2,B_2)$. Their cones intersect inside $\cone(T_1) \cap \cone(T_2) = \cone(T)$ where $T = T_1 \wedge T_2$ in $\Tree_{n-1}$. 
Then
\begin{eqnarray*}
\cone(M_1) \cap \cone(M_2) 
&=&  (\cone(T_1,B_1) \cap \cone(T)) \cap (\cone(T_2,B_2) \cap \cone(T)) \\
&=&  \cone(T'_1,B'_1) \cap \cone(T'_2,B'_2) \text{ where } T_1', T_2' \leq T \\
&=& f_{T^+}(\cone_{T^+}(B'_1)_{\geq 0} \cap \cone_{T^+}(B'_2)_{\geq 0}) \, \, \, \text{ for any maximal } T^+ \geq T \\
&=& f_{T^+}(\cone_{T^+}(B')_{\geq 0}) \quad \text{ for $B' = B'_1 \wedge B'_2$, by Proposition \ref{prop:BrAssocFan}}\\
&=& \cone(M') \text { for the Matryoshka $M'$ of } (T',B')
\end{eqnarray*}
where $T' \leq T^+$ is the tree supporting $B' = B'_1 \wedge B'_2$.

In the third row of Figure \ref{fig:intersection} we compute $\cone(M_1) \cap \cone(M_2)$ for the Matryoshkas of Figure \ref{fig:twoMatryoshkas}, noticing that the tree we had considered is $T=T_1 \wedge T_2$. 
Considering $B_1'$ and $B_2'$ as pre-bracketings of any maximal tree $T^+ \geq T$ we compute their meet $B'=B_1' \wedge B_2'$ and obtain from it the tree $T'$ and the inequalities of the intersection, $x_3-x_{12} = x_3-x_4 \geq 0$; that is, the ray $x_3 \geq x_{124}$ in $\R^4/\R$.

To prove that the fan is complete, we begin with the fact that the maximal tree cones $\cone(T)$ of the associahedral fan cover $\R^n/\R\1$. 
Now, for each $T$, the  cones $\cone_T(B)_{\geq 0}$ of the pre-bracketing fan cover the positive orthant $\R^{E(T)}_{\geq 0}$, so their images $f_T(\cone_T(B)_{\geq 0}) = \cone(T(B),B)$ cover 
$f_T(\R^{E(T)}_{\geq 0}) = \cone(T)$.

Finally, let us prove the poset isomorphism, keeping in mind that we already showed that $\Mat_{n+2} \cong \BrTree_{n+1}$. Clearly $M_1 \subseteq M_2$ implies $\cone(M_1) \subseteq \cone(M_2)$.
Assume that $\cone(M_1) = \cone(T_1,B_1) \subseteq  \cone(T_2,B_2) = \cone(M_2)$. Let $T$ be a maximal tree such that $\cone(T)$ contains these cones; $B_1$ and $B_2$ are pre-bracketings of $T$. Then
$f_T(\cone_T(B_1)_{\geq 0}) = \cone(T_1,B_1) = \cone(T_2,B_2) \subseteq f_T(\cone_T(B_2)_{\geq 0})$ by Lemma \ref{lem:f_T} so 
$\cone_T(B_1)_{\geq 0} \subseteq \cone_T(B_2)_{\geq 0}$ which means that $B_1 \leq B_2$ in $\PreBr(T)$ by Proposition \ref{prop:BrAssocFan}. Therefore $(T_1,B_1) \subseteq (T_2,B_2)$ in $\BrTree_{n-1}$, which means that $M_1 \leq M_2$ in $\Mat_{n-1}$ by Proposition \ref{prop:BrTree}.
\end{proof}

\section{\textsf{The cosmohedron}\label{sec:cosmohedron}}

In this section, we define the cosmohedron $\Cosmo_{n-1}(\a,\b)$ and prove that its normal fan is the cosmohedral fan $\CosmoFan_{n-1}$.

To set conventions, let $e_1, \ldots, e_n$ be the standard basis of $\R^n$. 
Let us write $i<<j$, when integers $i<j$ are not consecutive, that is, $|j-i|>1$. 
We note that the set of pairs $(i,j) $ with $0 \leq i << j \leq n+1$ is in bijection with the set $\diag^+(\pentagon_{n+2}) = \diag(\pentagon_{n+2}) \cup \{e\}$ of diagonals augmented with the base edge $e$.

\subsection{\textsf{Preliminaries: Associahedra and bracket associahedra}}

\noindent
\textbf{\textsf{Loday's associahedron.}} 
Given a sequence of positive numbers $\a = (a_{ij})_{0 \leq i << j \leq n+1}$, the \emph{Loday associahedron} is
\begin{align}
\Loday_{n-1}(\a) &= \sum_{0 \leq i << j \leq n+1} a_{ij} \Delta_{(i,j)} \\
&= \{x \in \R^n 
:  \sum_{1 \leq i \leq n} x_i = \sum_{0 \leq i << j \leq n+1} a_{ij}, \nonumber \\
& 
\qquad \qquad \quad\sum_{r \leq i \leq s} x_i \geq \sum_{r-1 \leq i << j \leq s+1} a_{ij} \text{ for } 1 \leq r \leq s \leq n. \label{eq:Loday}
\}
\end{align}
where $\Delta_{(i,j)} = \Delta_{[i+1, j-1]} = \conv(e_{i+1}, \ldots, e_{j-1})$ is a face of the standard simplex in $\R^n$, and $P+Q = \{p+q \, : \, p \in P, q \in Q\}$ denotes Minkowski sum of polytopes.\footnote{It is more customary to choose $\mathbf{\alpha} = (\alpha_{ij})_{1 \leq i \leq j \leq n}$ and write $\sum_{i \leq j} \alpha_{ij} \Delta_{[i,j]}$ for Loday's associahedron; we use $a_{i-1,j+1} = \alpha_{ij}$ instead. }

\begin{theorem} \cite{AguiarArdila, Postnikov}
The normal fan of the associahedron $\Loday_{n-1}(\a)$ is the associahedral fan $\AssocFan_{n-1}$.
\end{theorem}

Let $\S$ be a triangulation of $\pentagon_{n+2}$ corresponding to the binary tree $T$. It will be useful to express the vertex $a_T$ of Loday's associahedron in terms of $\S$. To do so, say that diagonal $ij$ \emph{transversally crosses} the triangle labeled $v$ with vertices $u<v<w$ if it intersects edges $uv$ and $vw$, not at $v$. This is equivalent to $u \leq i < v < j \leq w$. When this is the case we write $ij \rightarrow v$.
We invite the reader to check that each diagonal of the polygon crosses exactly one triangle of the triangulation $\S$ transversally.

\begin{lemma} \cite{Loday, Postnikov}  \label{lem:vertassoc}
The vertex $a_\S$ of $\Loday_{n-1}(\a)$ corresponding to a triangulation $\S$ of $\pentagon_{n+2}$ has coordinates
\[
(a_\S)_v = \sum_{ij \rightarrow v \text{ in } \S} a_{ij}
\]
summing over the extended diagonals $ij \in \diag^+(\pentagon_{n+2})$ that cross triangle $v$ transversally.
\end{lemma}

\begin{proof}
The vertex $a_T$ of $\Loday_{n-1}(\a)$ corresponding to a binary tree $T$ has coordinates
\[
(a_T)_v = \sum_{i \vee j = v \text{ in } T} a_{i-1,j+1} \qquad \text{ for } v \in [n]
\]
where $i \vee j$ is the lowest common ancestor of $i$ and $j$; 
this is implicit in  \cite{Loday, Postnikov} see for example \cite[Figure 6]{AguiarArdila}. 
The second statement then follows readily from the bijection. 
\end{proof}

\bigskip

\noindent
\textbf{\textsf{Devadoss's bracket associahedron}}.
For a graph $H$, let $\b: \Brackets(H) \rightarrow \Z$ be any function that is \emph{concave} in the sense that
$
b(\beta_1) + b(\beta_2) > b(\beta_1 \cap \beta_2) + b(\beta_1 \cup \beta_2)
$
for any intersecting brackets $\beta_1$ and $\beta_2$. For example, $b(\beta)$ can be any concave function of the number of edges $|\beta|=k$; that is, $b(k) > \frac12(b(k-1)+b(k+1))$.

\begin{definition}
The \emph{Devadoss bracket associahedron} of a graph $H$ is
\begin{align}
\BrAssoc_H(\b) = \{y \in \R^{E(H)} 
& :  
\sum_{e \in E(H)} y_e = b(H), \nonumber \\
& 
\,\,\,\,\,\,\, \sum_{e \in \beta} y_e \leq b(\beta) \text{ for all brackets $\beta$ of }H \label{eq:BrAssoc}
\}.
\end{align}
\end{definition}

\begin{theorem} \cite{Devadoss, PPP}\label{thm:brassoc}
The normal fan of the bracket associahedron $\BrAssoc_H{\b}$ is the bracket associahedral fan of $H$. 
\end{theorem}

We remark that Devadoss \cite{Devadoss} works with the sequence $b(\beta) = - 3^{|\beta|-2}$ -- see Lemma \ref{lemma:bracketfactorization} --  but in fact Theorem \ref{thm:brassoc} holds precisely for concave sequences. To see this, we use Padrol, Pilaud, and Poullot's description \cite[Corollary 2.12]{PPP} of the deformation cone of graph associahedral fans, and in particular bracket associahedral fans. This cone is given by the equality $b(G) = 0$ and inequalities $b(\beta - e) + b(\beta - e') < b(\beta) + b(\beta - e - e')$ for any bracket $\beta$ and any non-disconnecting edges $e,e'$ of $\beta$. This is equivalent to the concavity condition above.\footnote{Padrol, Pilaud, and Poullot require $b(G) = 0$, but we can remove this condition by a suitable translation. For general graphs, some care is required when $\beta-e-e'$ is disconnected, but this does not happen for trees.}

\begin{lemma} \label{lem:vertbrassoc}
Let $\b: \Brackets(H) \rightarrow \R$ be concave. 
The vertex $b_B$ of $\BrAssoc_H(\b)$ associated to a maximal bracketing $B$ of $H$ is given by 
\begin{equation} \label{eq:Devadoss}
(b_{B})_e = b(\bracket(e)) - \sum_{f \lessdot e} b(\bracket(f)) \qquad \text{ for } e \in E(H),
\end{equation}
summing over the edges $f$ that $e$ covers in $\tree(B)$. 
\end{lemma}

In a maximal bracketing, each bracket contains one or two maximal sub-brackets, so the expression above has at most three terms.

\begin{proof}
For each maximal bracketing $B$, we can compute the entries of vertex $b_{B}$  recursively by $(b_{B})_e = b(e)$ for any edge $e$ that is a bracket, and $\sum_{e \in \beta} (b_{B})_e = b(\beta)$ for every bracket $\beta$ of $B$. In the tree of $B$, this says
$b(\bracket(e)) = \sum_{f \leq e} (b_{B})_e$.
Adding a minimum element $\widehat{0}$ with $b(\widehat{0}) = b_{B}(\widehat{0})=0$, and applying M\"obius inversion in the poset $\tree(B) \cup \widehat{0}$, we get
$
(b_{B})_e = \sum_{f \leq e} \mu(f,e) \, b(\bracket(f))
$
Since the interval $[f,e]$ is a path, $\mu(f,e)$ equals $1$ if $f=e$, it equals $-1$ if $f \lessdot e$, and it is $0$ otherwise; so \eqref{eq:Devadoss} does indeed describe Devadoss's bracket associahedron. 
\end{proof}

Slightly differently, we will work with the \emph{extended set of brackets} $\text{Brackets}^+(H) = \text{Brackets}(H) \cup V(H)$, including also a \emph{small bracket} around each individual vertex. 
Consider any concave function $\b: \text{Brackets}^+(T) \rightarrow \Z$; in the concavity condition, if $\beta_1 \cap \beta_2$ is an edge-less graph on vertices $v_1,\ldots,v_k$, we set $b(\beta_1 \cap \beta_2) = b(v_1) + \cdots + \beta(v_k)$.
The function $\b$ induces a concave function $\b^-: \Brackets(H) \rightarrow \Z$ defined by
\begin{equation}\label{eq:b-}
b^-(\beta) = b(\beta) - \sum_{v \in V(\beta)} b(v),
\end{equation}
whose bracket associahedron we now describe.
A bracketing $B$ of $H$ now gives an \emph{extended} $\tree^+(B)$ that also includes the small brackets.

\begin{lemma} \label{lem:vertbrassoc2}
Let $\b: \Brackets^+(H) \rightarrow \R$ be concave and $\b^-:\Brackets(H) \rightarrow \R$ be given by \eqref{eq:b-}. 
The vertex $b_B$ of $\BrAssoc_H(\b^-)$ associated to a maximal bracketing $B$ of $H$ is given by 
\begin{equation} \label{eq:Devadoss2}
(b_{B})_e = b(\bracket(e)) - \sum_{f \lessdot e} b(\bracket(f)) \qquad \text{ for } e \in E(H).
\end{equation}
summing over the edges $f$ that $e$ covers in the extended $\tree^+(B)$. 
\end{lemma}

\begin{proof}
    This follows by applying Lemma \ref{lem:vertbrassoc} to $\b^-$.
\end{proof}

\noindent
\subsection{\textsf{The cosmohedron} }

In this section we construct a \emph{cosmohedron}: a polytope $\Cosmo_{n-1}(\a,\b)$ whose normal fan is the cosmohedral fan $\CosmoFan_{n-1}$. 
We first give two vertex descriptions of $\Cosmo_{n-1}(\a,\b)$ and use them to prove that it has the correct normal fan. We then use the normal fan to compute the inequality description of $\Cosmo_{n-1}(\a,\b)$. Finally, we use those inequalities to prove that $\Cosmo_{n-1}(\a,\b)$ is linearly isomorphic to the cosmohedron proposed by Arkani-Hamed, Figueiredo, and Vaz\~ao \cite{AFV}. This gives a proof of the correctness of their construction. 

Recall that $\diag^+(\pentagon_{n+2}) = \diag(\pentagon_{n+2}) \cup e$ is the set of diagonals of the $(n+2)$-gon augmented with the base $e$. 
We say a function $\b = (b(P) \, : \, P \text{ is a subpolygon of } \pentagon_{{n+2}}) \in \R^{\binom{n+2}{\geq 3}}$ is \emph{concave} if $b(P_1) + b(P_2) > b(P_1 \cap P_2) + b(P_1 \cup P_2)$ whenever the intersection and union are also subpolygons. An example is a concave function of the number of vertices of the polygon. 
 
\medskip

Throughout this section, we fix:

$\bullet$ an arbitrary positive vector $\a = (a_{ij} \, : \, ij \in \diag^+(\pentagon_{n+2})) \in \R^{\diag^+(\pentagon_{n+2})} \cong \R^{\binom{n+1}{2}}$,

$\bullet$ a concave vector $\b = (b(P) \, : \, P \text{ is a subpolygon of } \pentagon_{{n+2}}) \in \R^{\binom{n+2}{\geq 3}}$, 

\smallskip

$\bullet$ a small constant $0 < \varepsilon < \frac{m}{24M}$ where $m = \min a_{ij}$ and $M=\max |b(P)|$. 

\noindent The cosmohedron $\Cosmo_{n-1}(\a,\b)$ is constructed from those choices.

\subsubsection{\textsf{Vertex description}}

For a rooted plane tree $T$ let
\begin{eqnarray*}
f^T \, : \,  \R^{E(T)} &\longrightarrow& \R^{V(T)}_0\\
e_{ij} & \longmapsto & e_i-e_j
\end{eqnarray*}
be the isomorphism between $\R^{E(T)}$ and $\R^{V(T)}_0 := \{x \in \R^{V(T)} \, : \, \sum_v x_v=0\}$ dual to the isomorphism $f_T$ of Lemma \ref{lem:f_T}.
Also, our chosen vector $\b \in \R^{\binom{n + 2}{\geq 3}}$ induces a map $\b:\Brackets^+(T) \rightarrow \R$, since the extended brackets of $T$ correspond to the polygons in the triangulation of $M$. In turn this induces a map $\b^-:\Brackets(T) \rightarrow \R$ by \eqref{eq:b-}.

\begin{definition}\label{def:vertmat}
For a maximal Matryoshka $M$ with bracketed tree $(T,B)$ we define
\[
c_M := a_T + \varepsilon \, f^T(b_{B})
\]
where $a_T$ is the vertex of $\Loday_{n-1}(\a)$ corresponding to the plane tree $T$ and $b_B$ is the vertex of $\BrAssoc_T(\b^-)$ corresponding to its bracketing $B$. 
The \emph{cosmohedron} is
\[
\Cosmo_{n-1}(\a, \b):= \conv(c_M \, : \, M \text{ is a maximal Matryoshka on the $(n+2)$-gon}) \subset \R^n. \label{eq:rootcosmo}
\]
\end{definition}

Before we prove the correctness of this construction, let us give an explicit formula for the vertex $c_M$ directly in terms of the Matryoshka $M$.

\begin{proposition} \label{prop:cosmovertex}
Let $M$ be a maximal Matryoshka on $\pentagon_{n+2}$.
For each $v \in [n]$  define
\[
(a_M)_v = \sum_{ij \rightarrow v} a_{ij} = \sum_{u \leq i < v} \sum_{v < j \leq w} a_{ij}
\]
summing over the diagonals $ij$ that cross triangle $v$ (with vertices $uvw$) of the triangulation of $M$ transversally. 

For each polygon of the Matryoshka $M$ that is subdivided into maximal subpolygons $Q$ and $R$ in $M$, let
\[
b_M(P) = b(P) - b(Q) - b(R).
\]
For each $v \in [n]$, where triangle $v$ has edges $u_v$ (opposite to vertex $v$), $l_v$, and $r_v$, define
\[
(b_M)_v = b_M(P(l_v)) + b_M(P(r_v)) - b_M(P(u_v))
\]
where $P(e)$ is the minimal polygon in $M$ containing diagonal $e$; when $e$ is an edge of $\pentagon_{n+2}$, we set $b_M(P(e))=0$. Then  the vertex $c_M$ of a Matryoshka $M$ has coordinates
\[
(c_M)_v = (a_M)_v + \varepsilon (b_M)_v \qquad \text{ for } v \in [n].
\]
\end{proposition}

\begin{proof}
This follows readily from Lemmas \ref{lem:vertassoc} and  \ref{lem:vertbrassoc} and Definition \ref{def:vertmat}.
\end{proof}

\begin{remark}
Our preferred choice of $\a$ and $\b$ will be
\[
a_{ij} = 1 \text{ for all } i,j, \qquad 
b(P) = \begin{cases}
    0 & \text{ if } |P| \leq 4 \\ 
    3^{|P|-5} & \text{ if } |P| \geq 5. \\ 
\end{cases}
\]
This choice simplifies computations: we have
\[
(a_M)_v = |l_v| \cdot |r_v|
\]
where $|f|$ is the number of edges  of $\pentagon_{n+2}$ that are on the side of $f$ opposite to $e$. Also, $b_M(P)=0$ for all triangles and quadrilaterals $P$.
\end{remark}

\begin{example} 
Let us compute the vertex $c_{M_2}$ for the Matryoshka $M_2$ of Figure \ref{fig:Matryoshkas}.
We have $a_{M_2} = (1 \cdot 2, 1\cdot 1, 3 \cdot 2, 1 \cdot 1) = (2,1,6,1)$
The non-zero coordinates of $b_{M_2}$ are $b_{M_2}(\text{hexagon}) = 3-1 -0 = 2$ and $b_{M_2}(\text{pentagon}) = 1-0-0=1$. Therefore $b_{M_2}=(0+0-1, 0+0-0, 1+2-0, 0+0-2) = (-1,0,3,-2)$ and 
\[
c_{M_2} = (2,1,6,1) + \varepsilon(-1,0,3,2).
\]
\end{example}

\subsubsection{\textsf{The cosmohedron: proof of correctness}}

\begin{theorem} \label{thm:cosmohedron}
The normal fan of the cosmohedron $\Cosmo_{n-1}(\a,\b)$ is the cosmohedral fan $\CosmoFan_{n-1}$.
\end{theorem}

\begin{proof}
Let $\N_P(v) := \{w \in V^* \, : \, w(v) \geq w(p) \text{ for all } p \in P \}$ denote the normal cone of vertex $v$ in polytope $P \subset V$.
Since the cosmohedron lives on a hyperplane with constant coordinate sum, its dual space is  $\R^n/\R\1$. We need to prove that 
\[
\N_{\Cosmo_{n-1}(\a,\b)}(c_M) = \cone(M)
\]
for all maximal Matryoshkas $M$. Since both the left-hand sides and the right-hand sides subdivide $\R^n/\R\1$ with no interior overlap as $M$ ranges over the maximal Matryoshkas, it suffices to prove the inclusion $\subseteq$ for all $M$.

Let $M$ be a maximal Matryoshka with bracketed tree $(T,B)$, and take $w \in \N_{\Cosmo_{n-1}}(c_M)$. To prove that $w \in \cone(M)$, we proceed in two steps. 

\smallskip

\noindent \textsf{Step 1.} Consider a covering relation $\pentagon_M(ij^\perp) \gtrdot \pentagon_M(kl^\perp)$ in $M$. We need to show that $w$ satisfies the corresponding inequality $w_i-w_j \geq w_k-w_l$ for $\cone(M)$.

Edge $kl^\perp$ creates a subdivision $\pentagon_M(kl^\perp) = P_1 \cup P_2$ and edge $ij^\perp$ creates a subdivision $
\pentagon_M(ij^\perp) = \pentagon_M(kl^\perp) \cup P_3$.
Consider the Matryoshka $M'$ obtained by removing $\pentagon_M(kl^\perp)$ and adding a new polygon around $P_2$ and $P_3$.

\begin{figure}[h]
\begin{center}
\includegraphics[width=\textwidth]{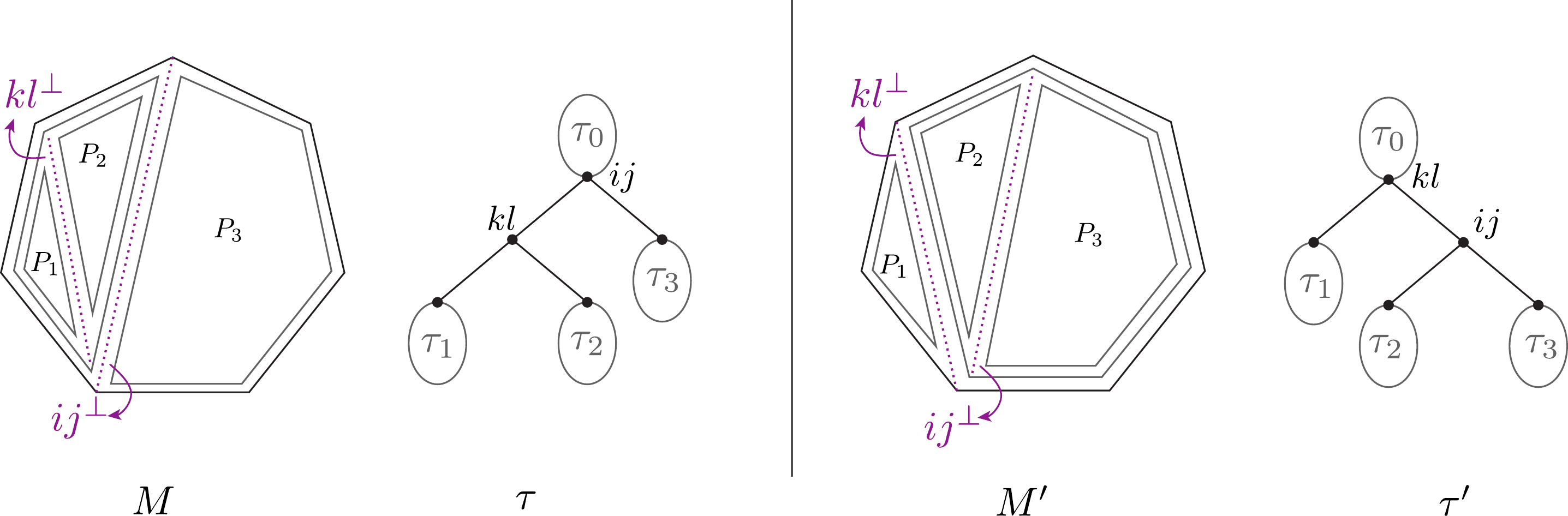}
\end{center}
\caption{A swap of a non-minimal polygon in a Matryoshka, and the corresponding effect on the polygon containment tree $\tau$.  \label{fig:step1}}
\end{figure}

The bracketed tree of $M'$ is $(T,B')$ for the same tree $T$ and a different bracketing. The containment trees $\tau= \tree(B)$ and $\tau' = \tree(B')$ differ by a simple mutation, as shown in Figure \ref{fig:step1}. 
Since Devadoss's bracket associahedron is a generalized permutahedron, its adjacent vertices $b_B$ and $b_{B'}$ satisfy $b_B - b_B' = L(e_{ij} - e_{kl})$ for some $L>0$\cite{AguiarArdila, Postnikov}. Now  $w \in \N_{\Cosmo_{n-1}}(c_M)$ gives
$w(c_M) \geq w(c_{M'})$, so  
\begin{eqnarray*}
w(a_T + \varepsilon \, f^T(b_B)) & \geq & w(a_{T} + \varepsilon \, f^T(b_{B'}))  \\
wf^T(L(e_{ij} - e_{kl})) & \geq & 0   \\
w((e_i-e_j) - (e_k-e_l)) & \geq & 0  
\end{eqnarray*}
as desired.

\smallskip

\noindent \textsf{Step 2.} 
Consider a polygon $\pentagon_M(kl^\perp)$ that does not contain a smaller $\pentagon_M(k'l'^\perp)$; it must be a quadrilateral split into two triangles  $k$ and $l$. Consider the Matryoshka $M'$ obtained by flipping those two triangles in the quadrilateral. 

\begin{figure}[h]
\begin{center}
\includegraphics[width=\textwidth]{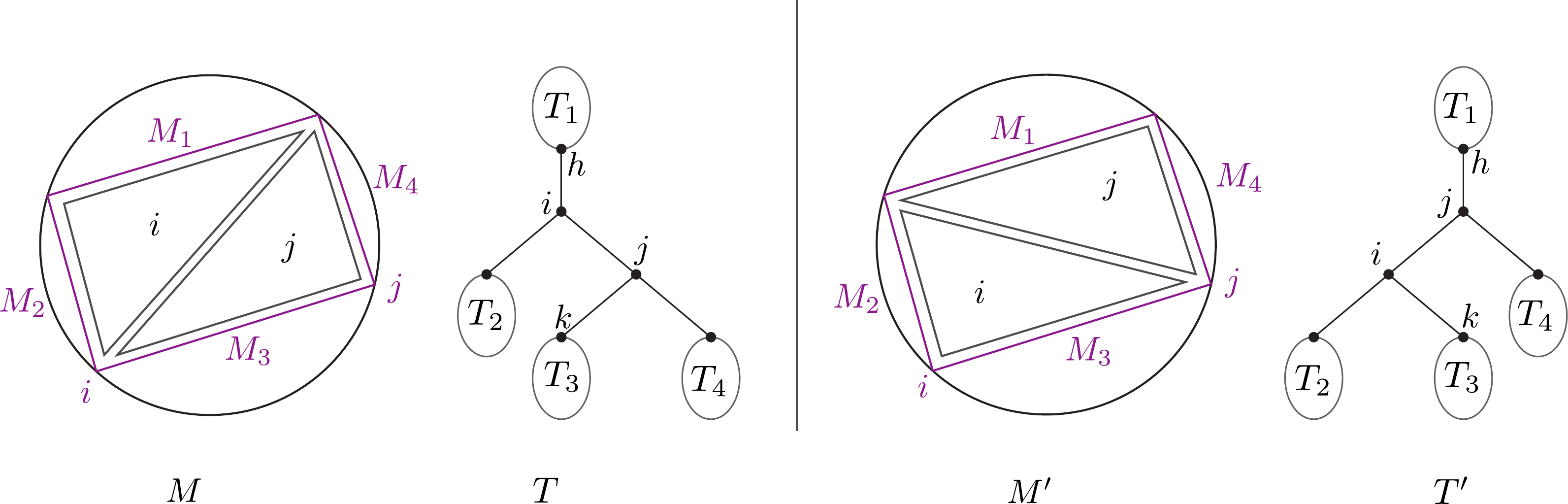}
\end{center}
\caption{A flip of the triangulation in a Matryoshka, and the corresponding effect on the tree $T$ of the triangulation.  \label{fig:step2}}
\end{figure}

The tree $T'$ of (the triangulation of) $M'$ is obtained from $T$ by a simple edge mutation, as shown in Figure \ref{fig:step2}. Since Loday's associahedron is a generalized permutahedron, its adjacent vertices $a_T$ and $a_{T'}$ satisfy $a_T-a_T' = L(e_{i} - e_{j})$ for some $L > 0$ \cite{AguiarArdila, Postnikov}. 
The bracketings $B$ and $B'$ both contain bracket $ij$, so they contain the same set of polygons, and give isomorphic posets $\tau=\tree(B) \cong \tree(B')=\tau'$. This means that $b_B = b_{B'}$  up to the resulting reordering of the edges $(hi, ij, jk) \mapsto (hj, ji, ik)$. 
Therefore if we write $H=(b_B)_{hi},  I=(b_B)_{ij}, J=(b_B)_{jk}$,
\begin{eqnarray*}
f^T(b_B)- f^{T'}(b_{B'}) &=&  (H(e_h-e_i)+I(e_i-e_j)+J(e_j-e_k)) \\
& & \quad - (H(e_h-e_j) + I(e_j-e_i) + J(e_i-e_k))  \\
&=& (-H+2I-J)(e_i-e_j).
\end{eqnarray*}
Now since $w \in \N_{\Cosmo_{n-1}}(c_M)$, we have $w(c_M) \geq w(c_{M'})$, so  
\begin{eqnarray*}
w\left((a_T - a_{T'}) + \varepsilon \, (f^T- f^{T'})(b_B) \right)& \geq & 0 \\
w((L+2 \varepsilon(-H+2I-J))(e_i-e_j) & \geq & 0 \\ 
w(e_i-e_j) & \geq & 0  
\end{eqnarray*}
as desired. 
In the last step we are observing that each entry of $(b_B)$ is at most $3M$ because the formula in Lemma \ref{lem:vertbrassoc} contains at most three terms, so $|2\varepsilon(-H+2I-J)| < 24M\varepsilon$,
whereas each edge $a_T-a_{T'}$ of $\Loday_{n-1}(\a)$ is a sum of edges of its Minkowski summands, so its length $L$ is at least $m$.
\end{proof}

\begin{remark} \label{rem:b=b}
In Step 2 of the proof above, it is essential that  $b_B$ and $b_{B'}$ are equal up to the natural reordering of their coordinates. This is a subtle requirement, because $b_B$ and $b_{B'}$ are vertices of different bracket associahedra $\BrAssoc_T$ and $\BrAssoc_{T'}$. The key property of Devadoss's realization of bracket associahedra that makes this work is that the coordinates of a vertex $b_B$ only depend on the poset $\tau=\tree(B)$ of the bracketing,
and in the above situation we have $\tau=\tree(B) = \tree(B') = \tau'$.
We note that Postnikov's realization \cite{Postnikov} of bracket associahedra does \textbf{not} have this property, so we cannot use it to realize the cosmohedron in this manner. \end{remark}

\noindent
\subsubsection{\textsf{Inequality description.} }

For a chord $C$ and a vertex $i$ or diagonal $ij$ of $\pentagon_{n+2}$, let us write 
$C \succ i$ if the chord $C$ is strictly above vertex $i$ and 
$C \succeq ij$ if the chord $C$ is weakly above diagonal $ij$, with respect to the base edge $e$ at the top.

For a subdivision $\S$ of $\pentagon_{n+2}$ (or equivalently a non-crossing set of chords $\C$), we define some parameters of $\S$ and $\C$ in a few equivalent ways. Each description is sufficient, but it can be useful at times to have all of them available. Let the \emph{depth} of a vertex $i$ or a diagonal $ij$ of $\pentagon_{n+2}$ be 
\begin{eqnarray*}
d_{\S}(i) &=& \# \text{  of chords $C$ of $\C$ with $C \succ i$} \\
&=& \text{distance from vertex $i$ to the root in $\tree(\S)$} \\
&=& \text{depth of vertex $i$ in $\tree_{\square}(\S)$}, \text{ and}\\ 
d_{\S}(ij) &=& \# \text{ of chords $C$ of $\C$ with $C \succeq ij$} \\
&=& \text{depth of cell $ij$ in $\tree_{\square}(\S)$}.
\end{eqnarray*}
where $\tree_{\square}(\S)$ is the canonical embedding of the tree of $S$ in a square grid, and the depth of a vertex or a cell in $\tree_{\square}(\S)$ is the number of non-root vertices of the tree whose southwest/southeast shadow contains it; see Figure \ref{fig:inequality}.

\begin{figure}[h]
    \centering
    \includegraphics[width=5in]{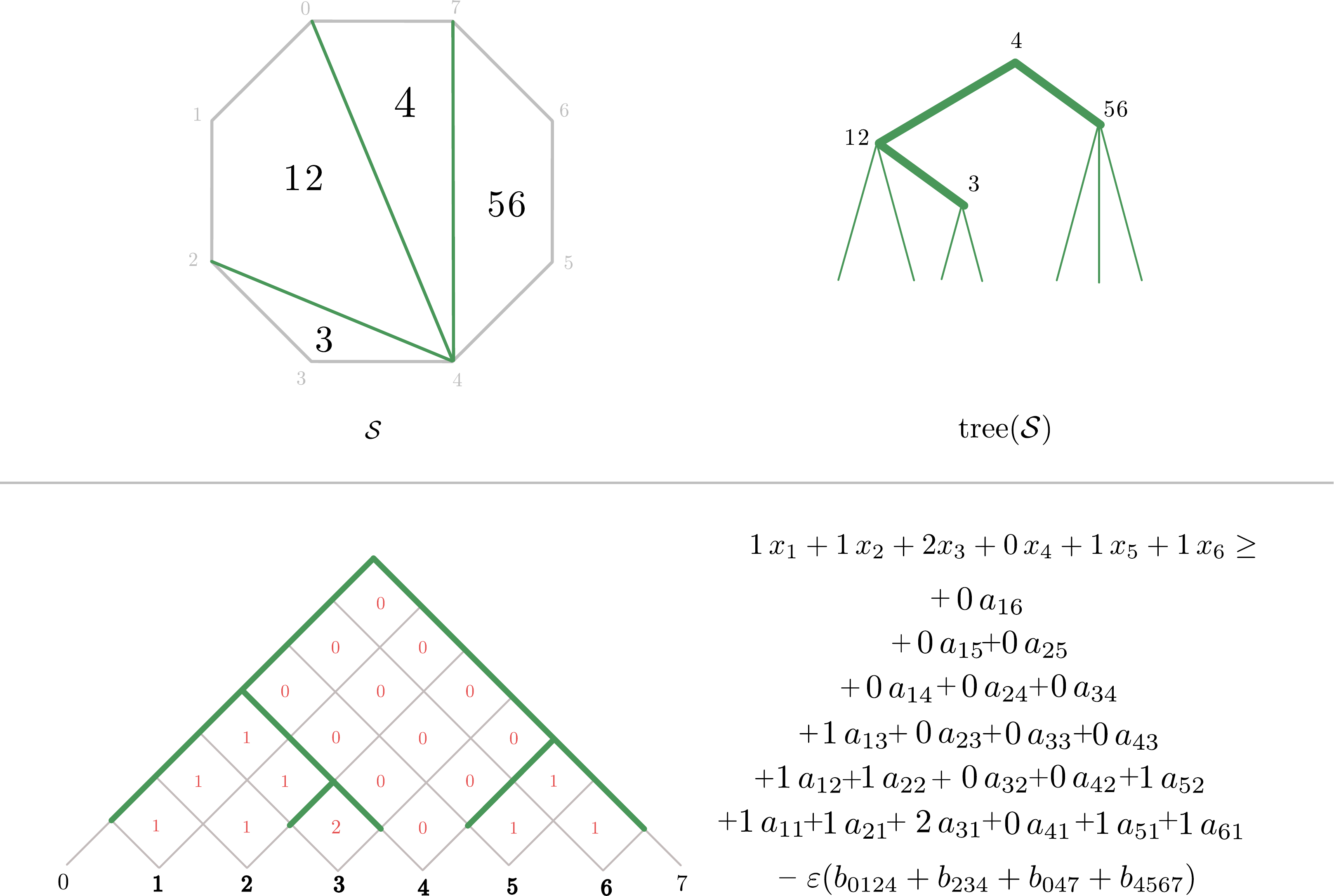}
    \caption{A subdivision $\S$, its corresponding tree, its grid embedding with the corresponding cell depths, and the corresponding cosmohedral facet.} \label{fig:inequality}
\end{figure}

\begin{theorem} \label{thm:cosmoineqs}
The cosmohedron $\Cosmo_{n-1}(\a, \b) \subset \R^n$ is described by the irredundant system of inequalities
\begin{eqnarray*}
\sum_{1 \leq i \leq n} x_i  &=&  \sum_{ij \in \diag(\pentagon_{n+2})} a_{ij},  \\ 
\sum_{1 \leq i \leq n} d_{\S}(i) \ x_i   &\geq&  \sum_{ij \in \diag(\pentagon_{n+2})} d_{\S}(ij)\ a_{ij} - \varepsilon \, \sum_{P \in \S} b(P) \text{ for any subdivision } \S.
\end{eqnarray*}
\end{theorem}

\begin{proof}
The lineality space of $\CosmoFan_{n-1}$ is $\R(1,\ldots, 1)$, so the only equality satisfied by $\Cosmo_{n-1}(\a,\b)$ is the given one.

The irredundant inequalities of $\Cosmo_{n-1}(\a,\b)$ correspond to the rays of $\CosmoFan_{n-1}$, which are in bijection with the minimal non-trivial Matryoshkas by Theorem \ref{thm:cosmofan3}. These are the subdivisions $\S$ of $\pentagon_{n+2}$ -- with no nesting -- obtained from the non-crossing sets of chords $\C$. By \eqref{eq:cone(M)}, the ray $\R_{\geq 0}( \rho_{\S}) = \cone(\S)$ is given by the inequalities $x_v-x_w \geq 0$, and hence $(\rho_{\S})_v - (\rho_{\S})_w =1$, for all chords $vw^\perp$ of $\S$. Therefore $\rho_{\S} = (d_{\S}(1), \ldots, d_{\S}(n))$. 
Notice that $\displaystyle \rho_{\S} = \sum_{C \in \C} {\rho_{\{C\}}}$ and $\rho_{\{C\}} = e_{i+1} + \cdots + e_{j-1}$ for $C = ij$.

To determine the extreme value of $(\rho_{\S}, x)$ for $x \in \Cosmo_{n-1}(\a,\b)$, we need to compute 
$(\rho_{\S},c_M)$ for any maximal Matryoshka $M=(T,B)$ that refines $\S$; see Figure \ref{fig:inequalityatavertex}. 
A chord $C$ gets a label $vw^\perp$ for a pair of adjacent triangles $v,w$ in the triangulation $T$, and the definition in Lemma \ref{lem:f_T} gives $f_T(e_{vw}) = \rho_{\{C\}} +\R\1$; this is also a ray of the associahedral fan $\AssocFan_{n-1}$. 
Therefore
\begin{eqnarray*}
(\rho_{\C}, c_M) &=& \sum_{C \in \C} \left(\rho_C\, , \, a_T + \varepsilon \, f^T(b_B)\right) \quad \text{ by Definition \ref{def:vertmat},} \\
 &=& \sum_{C \in \C} (\rho_C, a_T)  + \varepsilon \, \sum_{C=vw^\perp \in \C}  (f_T(e_{vw}), f^T(b_B)) 
\end{eqnarray*}
The first sum comes from the Loday associahedron; using \eqref{eq:Loday} we obtain
\[
\sum_{C \in \C} (\rho_C, a_T) 
= \sum_C \sum_{ij \, : \, ij \preceq C} a_{ij}\\
 = \sum_{ij} d_{\C}(ij) a_{ij}
\]
Since $f_T$ and $f^T$ are dual, the second sum equals
\[
\sum_{vw^\perp \in \C}  (b_B)_{vw} =\sum_{vw^\perp \in \C} \left( b(\bracket_B(vw))  - \sum_{tu \lessdot_B vw} b(\bracket_B(tu))\right)
\]
using Lemma \ref{lem:vertbrassoc2} for the vertices of the bracket associahedron.
Here $tu \lessdot_B vw$ means that $vw$ is covers $tu$ in the extended $\tree^+(B)$ of the bracketing $B$. 
 
Now, every diagonal $d$ of the triangulation $T$ that is not in $\C$ is inside a polygon $P \in \S$, so  $\bracket_B(d) \subset P$ cannot contain any chord in $\C$. It follows that $\C$ is an upper ideal of $\tree^+(B)$, and that the brackets of the maximal non-elements of $\C$ are precisely the polygons in $\S$. Also, $\tree(B)$ has a unique maximal element $VW = \tree(B)_{max}$. Therefore 
\begin{eqnarray*}
\sum_{vw^\perp \in \C}  (b_B)_{vw} 
&=& b(\bracket_B(VW)) - \sum_{tu \in (\tree(B)-\C)_{\text{max}}} b(\bracket_B(tu))\\
&=&
 b(\pentagon) - \sum_{P \in \S} b(P) \\
 &=&
 - \sum_{P \in \S} b(P) 
\end{eqnarray*}
using that $b(\pentagon)=0$ . The desired result follows.
\end{proof}

\begin{figure}[h]
\begin{center}
\includegraphics[width=5in]{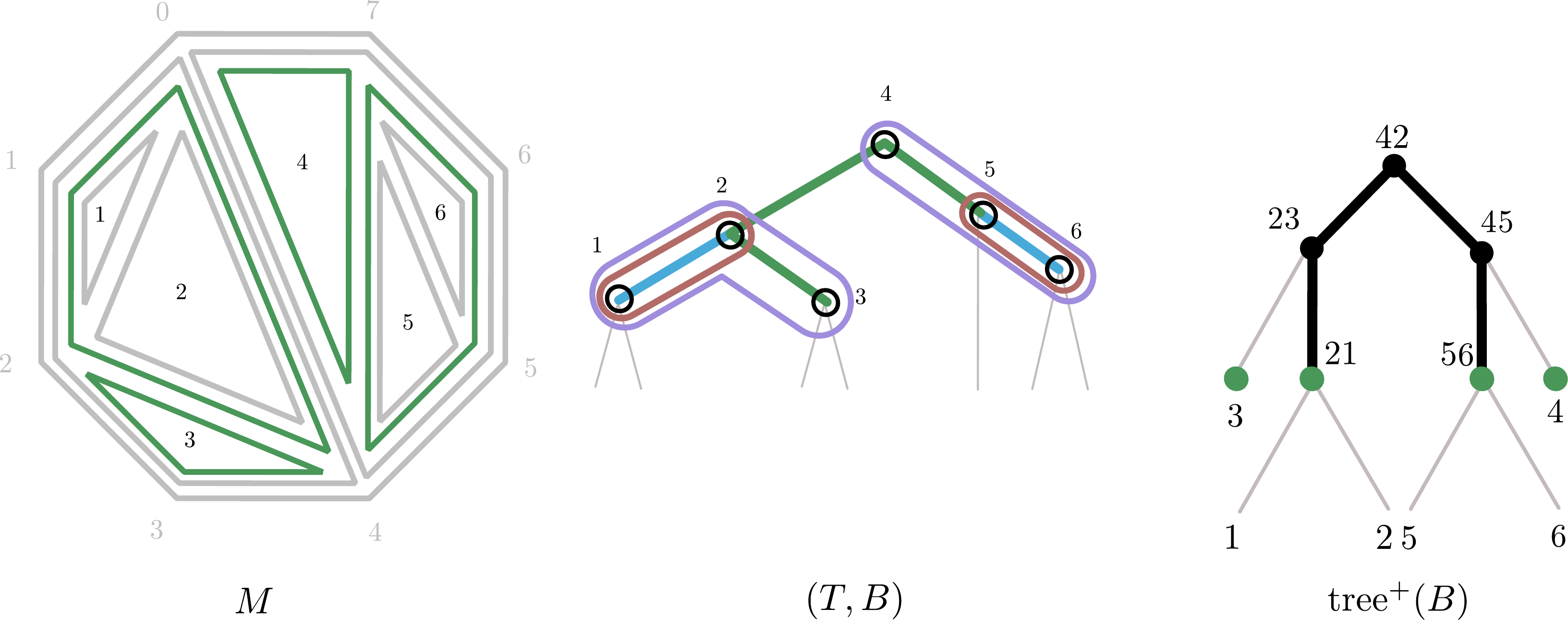}
\end{center}
\caption{A vertex on the facet given by Figure \ref{fig:inequality}. \label{fig:inequalityatavertex}}
\end{figure}

\begin{example}
Figure \ref{fig:inequalityatavertex} shows a Matryoshka $M$ that refines the subdivision $\S$ of Figure \ref{fig:inequality}, so vertex $c_M$ is on the facet that maximizes $\rho_{\C}$. In the computation of $(\rho_\C, c_M)$ above, the left sum is simply the dot product of the cell labels of $\tree_\square(\S)$ and $\a$.
Writing $b(\cdot) = b(\bracket_B(\cdot))$ for legibility, the right sum is 
\begin{eqnarray*}
    &&
[b(42))-b(23)-b(45)]
+
[b(23)-b(3)-b(21)]
+
[b(45)-b(56)-b(4)]
\\
&=& 
0 - b(3)-b(21)-b(56)-b(4) \\
&=& -b_{0124}-b_{234}-b_{047}-b_{4567}
\end{eqnarray*}
\end{example}

\subsubsection{\textsf{Isomorphism with the Arkani--Hamed-Figueiredo-Vaz\~ao cosmohedron.}\label{sec:AFV}}

In \emph{kinematic space} one assigns a real number to each diagonal of $\pentagon_{n+2}$:
\[
\R^{\diag(\pentagon_{n+2})} = 
\{(X_{ij})_{ij \in \diag(\pentagon_{n+2})}\} \cong \R^{(n+2)(n-1)/2}
\]
%of assignments of a real number to each diagonal of $\pentagon_{n+2}$. 
We assign $X_{0,n+1} = 0$ to the base edge $e=0,n+1$.

Let us fix:

$\bullet$ an arbitrary assignment $\A = (A_{ij} \, : \, 0 \leq i << j \leq n)$ of ${\binom{n}{2}}$ positive numbers to the diagonals of $\pentagon_{n+2}$ that do not involve $n+1$, and 

$\bullet$ a small convex vector $\B = (B(P) \, : \, P \text{ is a subpolygon of } \pentagon_{n+2})$.

\noindent Note that the vector $\b$ of the previous section was concave, and $B$ is convex.

\smallskip

Arkani-Hamed, Bai, He, and Yan \cite{ABHY} defined the \emph{associahedron}\footnote{In \cite{ABHY} this is denoted $\mathcal{A}_{n+2}$. We use the subscript $n-1$ to match the root system $A_{n-1}$ and the dimension of the polytope.} $\ABHY_{n-1}(\A)$  to be the $(n-1)$-dimensional polytope 
\begin{align}
 \{X \in \R^{\diag(\pentagon_{n+2})} 
\, : \, & \,   X_{r,s} + X_{r+1,s+1} - X_{r, s+1} - X_{r+1, s} = A_{rs} \text { for  } 0 \leq r  \ll  s \leq n 
\nonumber \\
&  
X_{r,s} \geq 0 \, \text{ for } \, 0 \leq r  \ll  s \leq n+1, 
\} \label{eq:ABHY}
\end{align}
where $X_{0,n+1} = X_{r,r+1}:=0$ \, for \, $0 \leq r \leq n$, and again, $r \ll s$ means that $s-r \geq 2$.

Arkani-Hamed, Figueiredo, and Vaz\~ao \cite{AFV} defined the following polytope, and conjectured that its face structure matches the combinatorics of Matryoshkas.

\begin{definition}\label{def:AFVcosmo}
 The \emph{cosmohedron} $\AFV_{n-1}(\A, \B)$ is the $(n-1)$-dimensional polytope 
\begin{align}
 \{X \in \R^{\diag(\pentagon_{n+2})} 
\, : \, &   \,  X_{r,s} \geq 0 \text{ for } 0 \leq r  \ll  s \leq n+1, \nonumber  \\
&  X_{r,s} + X_{r+1,s+1} - X_{r, s+1} - X_{r+1, s} = A_{rs} \text { for  } 0 \leq r  \ll  s \leq n, \nonumber \\
&  \sum_{C \in \C} X_{C} \geq \sum_{P \in \S(\C)} B(P) \, : \, \text{ for non-crossing } \C \subset \diag(\pentagon_{n+2}) \label{eq:AFVcosmo} \} 
\end{align}
where $\S(\C)$ is the set of polygons in the subdivision of $\pentagon_{n+2}$ induced by $\C$.
\end{definition}

\begin{theorem} \label{thm:samecosmo}
The map $\displaystyle X_{r,s} = \sum_{r < i < s} x_i - \sum_{r \leq i << j \leq s} a_{ij}$ gives linear isomorphisms
\begin{enumerate}
\item
between the associahedra $\Loday_{n-1}(\a)$ and $\ABHY_{n-1}(\A)$, and  
\item
between the cosmohedra $\Cosmo_{n-1}(\a,\b)$ and $\AFV_{n-1}(\A,\B)$, 
\end{enumerate}
where $A_{ij} = a_{i, j+1}$ for each diagonal $ij \in \diag(\pentagon_{n+2})$ and $B(P) = -\varepsilon \, b(P)$ for sufficiently small $\varepsilon > 0$.
\end{theorem}

Before proving this, we note that $\a$ has size ${\binom{n+1}{2}}$ whereas $\A$ has size ${\binom{n}{2}}$; this is because the $n$ entries $a_{i-1, i+1}$ do not enter the definition of $\AFV_{n-1}(\A,\B)$. These entries are the coefficients of the points $\Delta_{\{i\}} = e_i$ in the Minkowski sum of $\Loday_{n-1}(\a)$, so they only cause a translation of the polytope.

\begin{proof}[Proof of Theorem \ref{thm:samecosmo}]
1. A vector $x \in \R^n$ determines a vector $X$ in kinematic space that satisfies the linearly independent equations
\[
X_{r,s} + X_{r+1,s+1} - X_{r, s+1} - X_{r+1, s} = a_{r, s+1} = A_{r,s} \qquad \text{ for } 0 \leq r << s \leq n. 
\]
as one readily verifies. These $\binom{n}{2}$ equations restrict $X$ to a vector space of dimension $(n+2)(n-1)/2 - \binom{n}{2} = n-1$, which is precisely the dimension of the associahedron. The inequalities $X_{r-1,s+1} \geq 0$ are restatements of the defining inequalities of $\Loday_{n-1}(\a)$ as written in \eqref{eq:Loday}, proving the desired isomorphism. See also \cite{early}.

2. It remains to translate the inequalities of the cosmohedron $\Cosmo_{n-1}(\a,\b)$:
\[
\sum_{1 \leq i \leq n} d_{\C}(i) \ x_i   \geq  \sum_{ij \in \diag(\pentagon_{n+2})} d_{\C}(ij)\ a_{ij} - \varepsilon \, \sum_{P \in \S} b(P) 
\]
for any subdivision $\S$ with chords $\C$. We rewrite this as
\[
\sum_{C \in \C} \left(\sum_{1 \leq i \leq n} d_{\{C\}}(i) \ x_i \right)  \geq  \sum_{C \in \C} \left(\sum_{ij \in \diag(\pentagon_{n+2})} d_{\{C\}}(ij)\ a_{ij} \right)- \varepsilon \, \sum_{P \in \S} b(P). 
\]
Now notice that $d_{\{C\}}(i)$ is $1$ if $C \succ i$ and $0$ otherwise, and $d_{\{C\}}(ij)$ is $1$ if $C \succeq ij$ and $0$ otherwise, so
\[
X_C = \sum_{1 \leq i \leq n} d_{\{C\}}(i) \ x_i  - \sum_{ij \in \diag(\pentagon_{n+2})} d_{\{C\}}(ij)\ a_{ij}.
\]
Thus the inequalities of $\Cosmo_{n-1}(\a,\b)$ in $x$-space are equivalent to the inequalities of $\AFV_{n-1}(\A,\B)$ in $X$-space.
We note that $\B$ is concave if and only if $\b$ is concave, and $\B$ is made sufficiently small -- as \cite{AFV} requires -- by making $\varepsilon$ sufficiently small. 
\end{proof}

In particular, this proves one of our main results: the correctness of Arkani-Hamed, Figueiredo, and Vaz\~ao's construction. 

\begin{corollary}
The face poset of Arkani-Hamed, Figueiredo, and Vaz\~ao's cosmohedron $\AFV_{n-1}$ is anti-isomorphic to the poset of Matryoshkas of an $(n+2)$-gon.
\end{corollary}

It also lets us determine precisely what their requirement of a ``small enough" $B$ means.
Theorem \ref{thm:cosmohedron} proves that 
$\Cosmo_{n-1}(\a,\b)$ is a cosmohedron for $0 < \varepsilon < (\min a_{ij})/24(\max b(P))$; it would be interesting to trace the proof carefully and find the optimal constant.

\section{\textsf{Face structure and enumeration}\label{sec:enum}}

\noindent
\textsf{\textbf{The dimension of a face.}} 
Recall that $M_{\text{not min}}$ is the number of polygons in a Matryoshka $M$ that are not inclusion-minimal. 

\begin{lemma} \label{lem:dim2}
The cone of a Matryoshka $M$ in the cosmohedral fan has dimension
\[
\dim \cone(M) = |M_{\text{not min}}| + 1.
\]
\end{lemma}

\begin{proof}
If $(T,B)$ is the bracketed tree corresponding to $M$ from Lemma \ref{lem:bijection}, we have $\cone(M) = \cone(T,B) = f_T(\cone_T(B)_{\geq 0})$  and its dimension equals the dimension of $\cone_T(B)$, which is the number of brackets in $B$. This equals the number of non-minimal polygons of the Matryoshka $M$.
\end{proof}

\noindent \textbf{\textsf{Faces of the cosmohedron factor combinatorially, but not geometrically.}} 
The faces of the cosmohedron have an interesting factorization property that is subtler and weaker than the analogous property for bracket associahedra.

\medskip

Bracket associahedra\footnote{and, more broadly, generalized permutahedra \cite{AguiarArdila, Edmonds}} have the property that any facet factors canonically as a Cartesian product of two bracket associahedra, as we now describe. Let $H$ be a graph, $\beta$ be a bracket of $H$, and $H/\beta$ the contraction of $\beta$ in $H$, whose edges are in bijection with $E(H)-E(\beta)$.  Also let $b|_\beta(\beta') = b(\beta')$ for each bracket $\beta'$ of $\beta$ and $b/_\beta(\beta'') = b(\beta\cup \beta'') - b(\beta)$ for each bracket $\beta''$ of $H/\beta$. Note that if $b$ is concave on $\Brackets(H)$, then $b|_\beta$ and $b/_\beta$ are concave on $\Brackets(\beta)$ and $\Brackets(H/\beta)$, respectively.

\begin{lemma} \label{lemma:bracketfactorization} \cite{AguiarArdila}
For any graph $H$ and any bracket $\beta$ of $H$, the $e_\beta$-maximal facet of $\BrAssoc_{H}(\b)$ is linearly isomorphic to the product of bracket associahedra
\[
(\BrAssoc_{H}(\b))_\beta \cong 
\BrAssoc_{\beta}(\b|_\beta) \times
\BrAssoc_{H/\beta}(\b/_\beta) \subset \R^{E(\beta)} \times \R^{E(H/\beta)}
\]
\end{lemma}

\medskip

Facets, and hence faces, of cosmohedra have an analogous factorization property -- but it is significantly weaker. The first statement below was observed in \cite{AFV}, but it holds an unexpected surprise.

\begin{proposition} \label{prop:factorization} 
Let $\S$ be a subdivision of the $(n+2)-gon$ into polygons $P_1, \ldots, P_k$ where $P_i$ is an $(n_i+2)$-gon. Let $T(\S)$ be its dual graph, which is a tree. Then the facet of $\Cosmo_n$ corresponding to $\S$ is \textbf{combinatorially} isomorphic to the product
\[
(\Cosmo_n)_{\S} \cong^{comb} \BrAssoc_{T(\S)} \times (\Cosmo_{n_1}) \times \cdots \times (\Cosmo_{n_k}).
\]
However, this factorization \textbf{does not hold geometrically}. 
\end{proposition}

\begin{proof}
The face poset of $L({\Cosmo_n})$ is the opposite of the poset $M_n$ of Matryoshkas of the $(n+2)$-gon. The interval $[\widehat{0},(\Cosmo_n)_{\S}]$ corresponds to the subposet $(M_n)_{\geq{\S}}$ of Matryoshkas refining $\S$. Such a Matryoshka is obtained by making $k+1$ independent choices:

(i) refining each of the $k$ polygons $P_i$ internally with its own Matryoshka -- the ways of doing this are in order-reversing bijection with the faces of $\Cosmo_{n_i}$ --  and 

(ii) refining $\S$
 externally with compatible polygons that are unions of the $P_i$s -- the ways of doing this are in order-reversing bijection with the bracketings of $T(\S)$.

\noindent
 These independent choices are compatible with the order relations of the posets, giving the desired result.
\end{proof}
The above isomorphism is not, in general, a linear isomorphism of polyhedra. 
We can already see this in the cosmohedron $\Cosmo_3$.
Consider the subdivision of a hexagon by a long and a short diagonal, as illustrated in Figure \ref{fig:quadrilateral}. Proposition \ref{prop:factorization} predicts that this facet is combinatorially isomorphic to the product $T(P_2) \times \Cosmo_{2}  \times \Cosmo_{1} \times \Cosmo_{1} = | \times | \times \cdot \times \cdot$ which is a quadrilateral. We claim that this face is a trapezoid; we can see this in Figure \ref{fig:cosmohedron}, but let us verify it algebraically. We use Proposition \ref{prop:cosmovertex} to compute:
\begin{eqnarray*}
c_{M_1} = (1, 2, 6, 1) + \varepsilon(0,-3,3,0)& \qquad &  
 c_{N_1} = (2, 1, 6, 1) + \varepsilon(-3,0,3,0) \\
c_{M_2} = (2, 1, 6, 1)  + \varepsilon(-1,0,3,-2) & \qquad & 
 c_{N_2} = (1, 2, 6, 1) + \varepsilon(0,-1,3,-2) 
\end{eqnarray*}
and notice that the opposite sides
\[
c_{M_1} - c_{N_1} = (1-3 \varepsilon)(e_2-e_1), \qquad 
c_{M_2} - c_{N_2} = (1- \varepsilon)(e_1-e_2), \qquad 
\]
are parallel but not equal. Since it is not a parallelogram, this quadrilateral face cannot be expressed as a Cartesian product of two edges.

\begin{figure}[h]
\begin{center}
\includegraphics[width=4in]{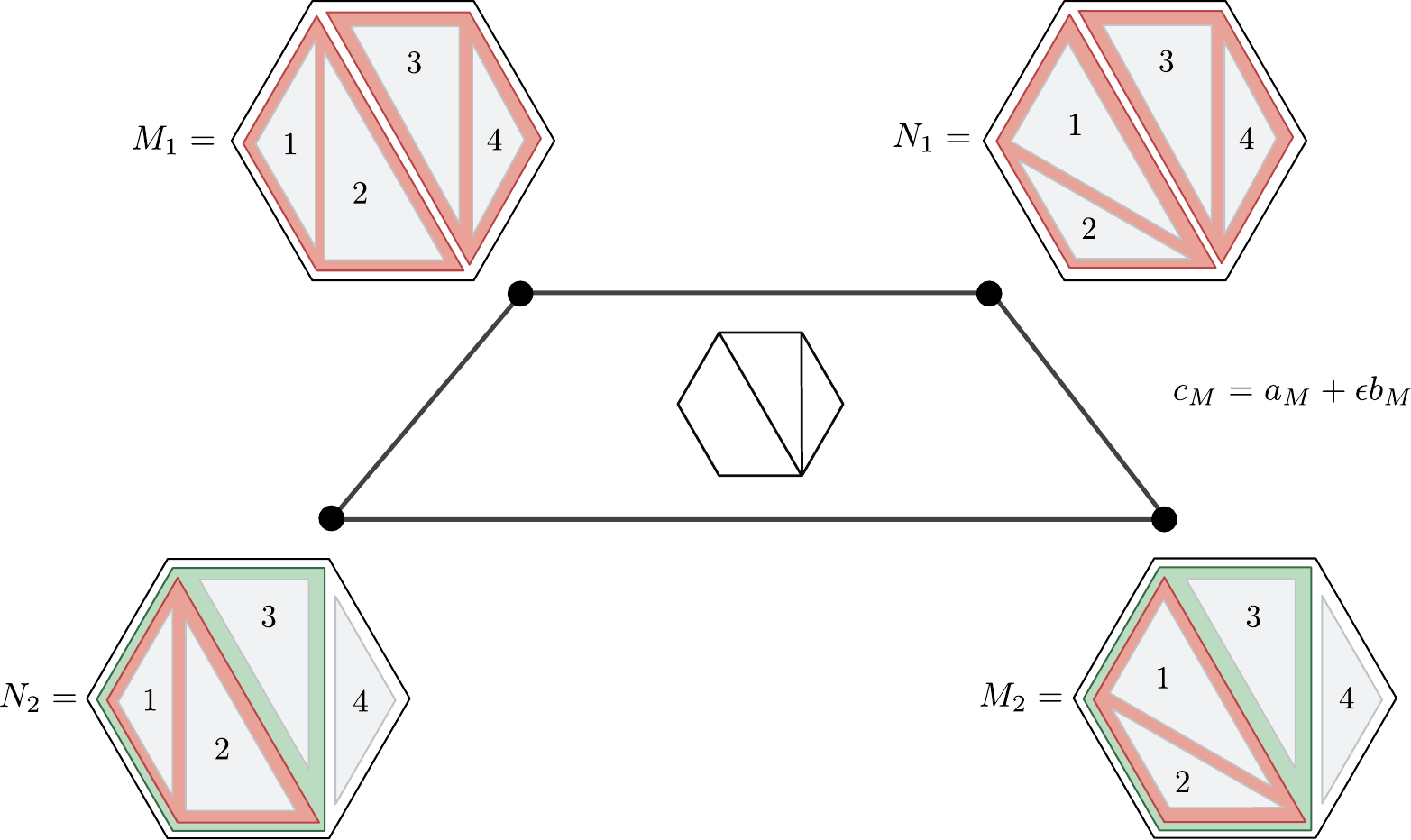}
\end{center}
\caption{A trapezoidal face of the cosmohedron.}
\label{fig:quadrilateral}
\end{figure}

\subsection{\textsf{Enumeration}}

Having understood the combinatorial structure of the cosmohedron, we are now ready to enumerate its faces. We also compute the face numbers of the \emph{correlatron} of \cite{Correlatron}, which are closely related in an elegant and surprising way.

\subsection{\textsf{Counting the faces of the cosmohedron}}

Our next goal is to compute the number of faces of the cosmohedron in each dimension.  We list these numbers in Table \ref{tab:cosmohedron} for $n \geq 6$.

\begin{table}[h]
\centering
\begin{tabular}{|c||c|c|c|c|c|c|c|c|}
\hline
 cosmohedron & 0 & 1 & 2 & 3 & 4 & 5 & 6 \\
\hline
\hline
0 & 1 &   &    &  &  &  &  \\
1 & 1 & 2 &  &  &  &  &  \\
2 & 1 & 10 & 10 &  &  &  &  \\
3 & 1 & 44  & 114 & 72 &  &  &  \\
4 & 1 & 196 & 952 & 1400 & 644 &  &   \\
5 & 1 & 902 & 7116 & 18040 & 18528 & 6704 &  \\
6 & 1 & 4278 & 50550 & 194616 & 332664 & 262728 & 78408 \\
\hline
\end{tabular}
\caption{The number of $d$-codimensional faces of the $n$th cosmohedron on $n+2$ points.}
\label{tab:cosmohedron}
\end{table}

It is useful to organize them in the $f$-polynomial $f_n(t) = f_{\AssocFan_{n-1}}(t)$ and the modified $f$-polynomial $g_n(t)=[(1+t)f_n(t)-1]/t$, whose coefficients are the consecutive pair sums of the coefficients of $f_n(t)$:
\begin{eqnarray*}
&& f_1(t)=1, \,\, f_2(t)=1+2t, \,\,f_3(t)=1+10t+10t^2, \,\,f_4(t)=1+44t+114t^2+72t^3, \ldots \\
&& g_1(t)=1, \,\,g_2(t)=3+2t, \,\,g_3(t)=11+20t+10t^2, \,\,g_4(t)=45+148t+186t^2+72t^3, \ldots. 
\end{eqnarray*}

\begin{theorem} \label{thm:cosmofvector} The
$f$-polynomials of the cosmohedral fans form the unique sequence $\{f_n(t)\}_{n \geq 1}$ of polynomials such that the power series in $(\R[t])[x]$
\begin{eqnarray*}
&& x - f_1(t)x^2 - f_2(t)x^3 - f_3(t)x^4 - \ldots \\
&& x + g_1(t)x^2 + g_2(t)x^3 + g_3(t)x^3 + \ldots 
\end{eqnarray*}
are compositional inverses, where  $g_n(t) = [(1+t)f_n(t) - 1]/t$. Equivalently, the power series $A_t(x) = x - f_1(t)x^2 - f_2(t)x^3  + \ldots$ satisfies
\[
A_t\left(\frac{(2t+1)x-(2t+2)x^2}{t(1-x)} - \left(1+\frac1t\right)A_t(x)\right) = x.
\]
\end{theorem}

\begin{proof}
Every face $F$ of the cosmohedral fan is contained in a unique minimal face of the associahedral fan $A(F)$. If $F$ corresponds to a Matryoshka $M$ on an $(n+2)$-gon then $A(F)$ corresponds to the subdivision $M_{max}$ given by the maximal polygons $P_1, \ldots, P_k$ of $M$. 
If $P_i$ is an $(n_i+2)$-gon then $n = n_1 + \cdots + n_k$.

Conversely, consider a face $A$ of the associahedral fan, corresponding to a non-trivial subdivision $\S$ of the $(n+2)$-gon into an $(n_i+2)$-gon $P_i$ for $1 \leq i \leq k$.
The faces $F$ of the cosmohedron such that $A(F)=A$ correspond to the Matryoshkas with $M_{max} = \S$; these are obtained by choosing arbitrary Matryoshkas $M_1, \ldots, M_k$  on $P_1, \ldots, P_k$, that is, faces $F_1, \ldots, F_k$ on the corresponding cosmohedral fans. Furthermore, we have
\[
\dim F = 1+|M_{\text{not min}}| 
= 1+\sum_{i=1}^k (|(M_i)_{\text{not min}}| +1) = 1+ \sum_{i=1}^k \dim F_i.
\]
so the $f$-numbers of the cosmohedron $\Cosmo_{n-1}$ are given by the recurrence
\begin{equation} \label{eq:f}
f_n(t) = \sum_{F \leq \Cosmo_{n-1}} t^{\codim F} =
1 + t \sum_{A \lneq \Assoc_{n-1}} f_{n_1}(t) \cdots f_{n_k}(t).
\end{equation}
Adding the full face $\Assoc_{n-1}$ to both sides, we obtain
\begin{equation} \label{eq:fg}
g_{n}(t) =  \sum_{A \leq \Assoc_{n-1}} f_A(t)
\end{equation}
using the notation of \cite{AguiarArdila}: if $x_1, x_2, \ldots, $ is a sequence and a face $A$ of the associahedron $\Assoc_{n-1}$ is combinatorially a product of associahedra $\Assoc_{n_1-1} \times \cdots \Assoc_{n_k-1}$, then we write $x_A = x_{n_1}\cdots x_{n_k}$.

Arriving at \eqref{eq:fg} is a pleasant d\'ej\`a vu: the associahedral version of Lagrange inversion obtained in \cite[Theorem 2.4.3]{AguiarArdila} tells us that this recurrence is precisely equivalent to the assertion that $x - f_1(t)x^2 - f_2(t)x^3 -  \ldots$ and $ x + g_1(t)x^2 + g_2(t)x^3 + \ldots$ are compositional inverses. 

The second formulation follows from the first by a straightforward computation.
\end{proof}

A pleasant feature of this proof is that Aguiar and Ardila discovered the associahedral version of Lagrange inversion in 
\cite[Theorem 2.4.3]{AguiarArdila} within the Hopf-algebraic structure of Loday's associahedra, specifically. Is it coincidental that this is the same associahedron we constructed? Might there be a deeper underlying algebraic explanation?
Cosmohedra feature an operadic structure that is the subject of an upcoming paper.

\begin{corollary}
The $f$-polynomials of the cosmohedral fans are given by the recurrence $f_1(t)=1$ and 
\[
f_n(t) = 1 + t \sum_{1^{a_1}2^{a_2}\cdots \vdash n} \frac{(n+a_1+a_2+\cdots)!}{(n+1)!} \cdot  \frac{f_1(t)^{a_1}}{a_1!} \frac{f_2(t)^{a_2}}{a_2!} \cdots \qquad \text{for } n \geq 2.
\]
summing over the non-trivial partitions of $n$ as a sum $n=a_1 \cdot 1 + a_2 \cdot 2 + \cdots + a_{n-1} \cdot(n-1)$.
\end{corollary}

\begin{proof}
This is a reformulation of the recurrence formula \eqref{eq:f}, taking into account that the number of subdivisions of an $(n+2)$-gon into $a_2$ triangles, $a_3$ quadrilaterals, etc. is $(n+a_1+a_2+\cdots)! / (n+1)!a_1!a_2! \cdots$ \cite[Theorem 5.3.10]{EC2}.
\end{proof}

\subsection{\textsf{Asymptotic growth}}
\noindent \textbf{\textsf{Facets.}}
The facets of the cosmohedron $\Cosmo_{n-1}$ are in bijection with the non-trivial subdivisions of an $(n+2)$-gon, or equivalently the non-trivial parenthesizations of a string of $n+1$ letters.
Their number, listed in column 1 of Table \ref{tab:cosmohedron}, is one less than the $n$-th \emph{Hipparchus--Schr\"oder number}. This is one of the oldest known combinatorial quantities, studied by astronomer and mathematician Hipparchus in the second century B.C.; for a fascinating historical account, see \cite{EC2, StanleyHipparchus}. 
These numbers are sequence A001003 in the Online Encyclopedia of Integer Sequences; their generating function is 
\[
x+x^2+3x^3+11x^4+45x^5+197x^6+\cdots = \frac14\left(1+x-\sqrt{1-6x+x^2}\right)
\]
The asymptotic growth of the sequence is controlled by $3 + \sqrt8$, the inverse of the smallest root of $1-6x+x^2$. More precisely \cite{Knuth}, the number $r_n$ of rays grows like 
\[
r_n \sim \frac14\sqrt{(\sqrt{18}-4)/\pi}  \cdot (3+\sqrt8)^n \cdot n^{-3/2}.
\]

\bigskip

\noindent \textbf{\textsf{Vertices.}}
The vertices of the cosmohedron $\Cosmo_{n-1}$ are in bijection with the maximal Matryoshkas on an $(n+2)$-gon.
The following lemma is stated in \cite{AFV} with a minor typo; we include a proof.
\begin{lemma}
The number $m_n$ of maximal Matryoshkas on  an $(n+2)$-gon is given by the recurrence relation
\[
m_1 = 1,  \qquad m_n = \sum_{k=1}^{n-1} (k+1)m_{k}m_{n-k} \text{ for } n \geq 2.
\]
The formal power series $M(x) = \sum_{n \geq 2} m_{n-1}x^n = x^2 + 2x^3 + 10x^4+72x^5+\cdots$ is given by the differential equation 
\[
M'(x) = 1 - \frac{x^2}{M(x)}.
\]
\end{lemma}

\begin{proof}
Fix an edge $e$ of the $(n+2)$-gon. 
A maximal Matryoshka is obtained by splitting an $(n+2)$-gon into a $(k+2)$-gon that contains $e$ and an $(n-k+2)$-gon, and then putting a maximal Matryoshka in each one of them. There are $k+1$ possible positions for the outer $(k+2)$-gon that contains $e$, and $m_k$ and $m_{n-k}$ choices for the two inner Matryoshkas, respectively. The differential equation follows by a straightforward computation. 
\end{proof}

In particular, this means that $M(x)$ is a $D$-algebraic series; that is, its derivatives satisfy a polynomial equation with coefficients in $\mathbb{C}[x]$; in fact, one of the simplest such equations: $M(x) = x^2 + M(x)M'(x)$. Compare this with the algebraic power series for Catalan numbers, which satisfies one of the simplest algebraic equations $C(x) = 1+xC(x)^2$.

Melczer \cite{Melczer} writes that ``\emph{$D$-algebraic series seem to be on the border between decidability and undecidability,
perhaps tipped to the side of undecidability. [...] Compared to D-finite functions, not much is
known about $D$-algebraic asymptotics, and in general not much can be said.}". Since   $M(x)$ is one of the simplest $D$-algebraic series, the following conjecture of Kotesovic, based on numerical evidence, is an interesting challenge.

\begin{conjecture}\cite{OEIS}\label{conj:asymptotic} The number of Matryoshkas grows asymptotically like 
\[
m_n \sim  c \cdot n!  \cdot  n^4
\]
for $c = 0.005428\ldots$. 
\end{conjecture}

\bigskip

\noindent
\textbf{\textsf{All faces.}}
A more ambitious goal would be to find the asymptotics for the total number of faces,  or even the limit distribution of the normalized $f$-vector. We do not know how to do this.

\subsection{\textsf{Counting the faces of the  correlatron}}

Another object that naturally arises in physics, and that is intimately related to the cosmohedron, is the \emph{correlatron} introduced by Figueiredo and Vaz\~ao in \cite{Correlatron}. This polytope captures the combinatorics of the correlator, a physical observable that can be computed from the wavefunction. The combinatorial structure of this object is described by the interacting combinatorics of chords (the combinatorics of associahedra) and subpolygons (the combinatorics of cosmohedra). 

\begin{figure}[h]
    \centering    
    \includegraphics[width=0.8\linewidth]{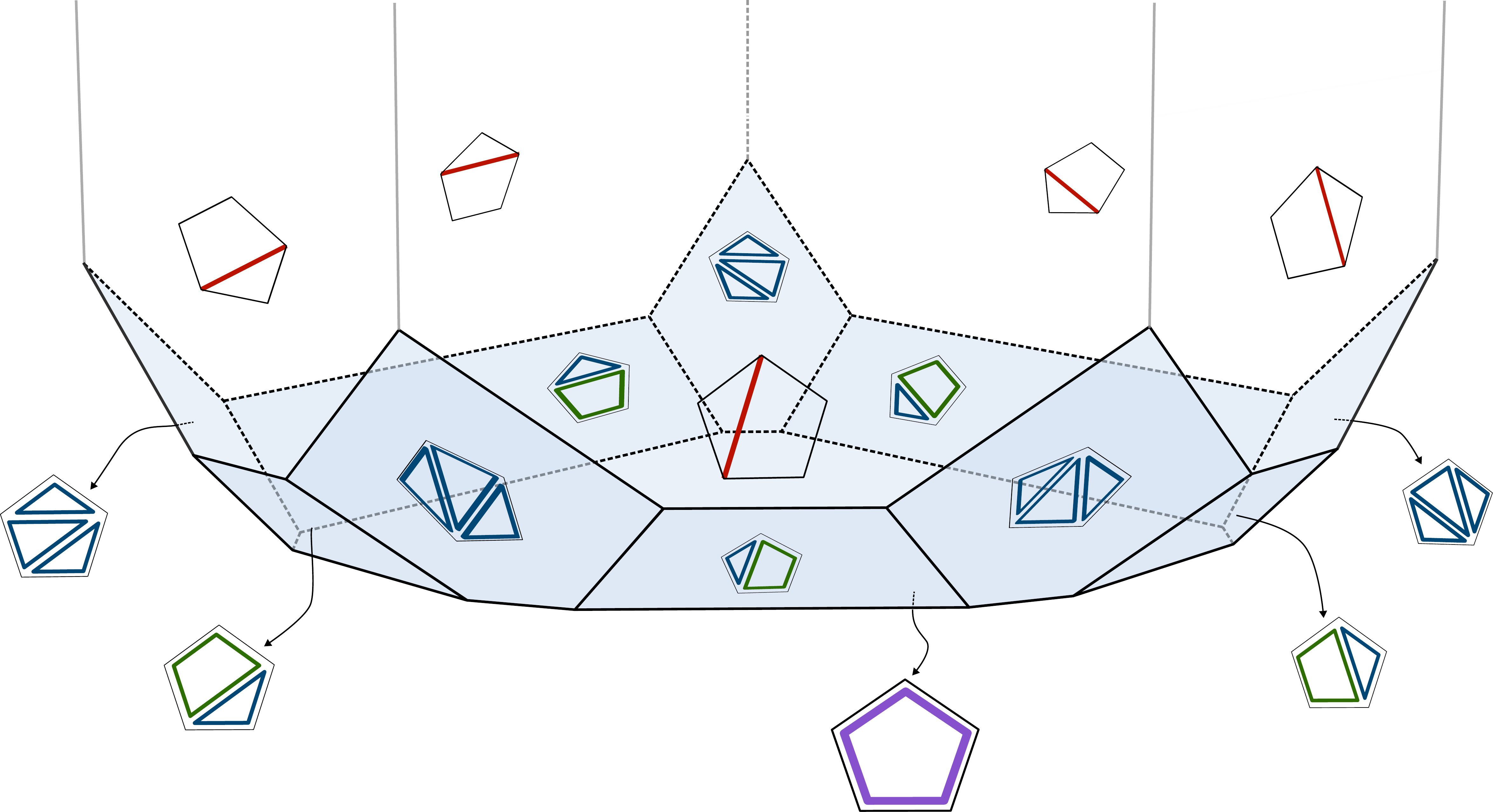}
    \caption{The three-dimensional correlatron.}
    \label{fig:correlatron}
\end{figure}

The correlatron is an unbounded $n$-dimensional polyhedron, with faces of two kinds:

1. The bounded faces of the correlatron are labelled by the compatible pairs $(\C,M)$, consisting of a set $\C$ of non-crossing chords and a Matryoshka $M$ on an $(n+2)$-gon. Here $\C$ and $M$ are \emph{compatible} if no chord in $\C$ intersects a polygon in $M$; or equivalently, if $\C$ is a subset of the diagonals in the underlying subdivision $M_{\text{max}}$ of $M$. The bounded face inclusions of the correlatron are given by the reverse inclusion of chords and Matryoshkas:
\begin{equation*}
    \{\mathcal{C}^\prime, M^\prime \} \text{ is a face of }\{\mathcal{C},M\} \text{ in the correlatron if and only if  } \mathcal{C} \subseteq \mathcal{C}^\prime \text{ and } M \leq M^\prime.
\end{equation*}

2. The unbounded faces are labeled by chords $\underline{\C}$, which we underline to distinguish them easily from the earlier faces. An unbounded face $\underline{\C}$ is contained in $\underline{\C'}$ if and only if $\C \subseteq \C'$. An unbounded face $\underline{\C}$ also contains the bounded faces $(\C',M')$ such that $\C \subseteq \C'$.

Note in particular that the correlatron has a bottom cosmohedral facet $\Cosmo_{n-1}$ corresponding to the pairs of the form $(\emptyset, M)$. Its unbounded directions look like a cone over  $\Assoc_{n-1}$.
Since the unbounded faces of the correlatron are enumerated by the well-known $f$-vector of the associahedron, we focus on enumerating the bounded part.

\begin{table}[h]
\centering
\begin{tabular}{|c||c|c|c|c|c|c|c|c|}
\hline
 correlatron & 0 & 1 & 2 & 3 & 4 & 5 & 6 \\
\hline
\hline
0 & 0 &  &   &  &  &  &  \\
1 & 0 & 1 &  &  &  &  &  \\
2 & 0 & 3 & 4 &  &  &  &  \\
3 & 0 & 11  & 35 & 25 &  &  &  \\
4 & 0 & 45 & 251 & 405 & 200 &  &   \\
5 & 0 & 197 & 1694 & 4592 & 4984 & 1890 &  \\
6 & 0 & 903 & 11158 & 44932 & 80036 & 65606 & 20248 \\
\hline
\end{tabular}
\caption{The number of $d$-codimensional faces of the $n$th bounded correlatron.}
\label{tab:correlatron}
\end{table}

\begin{theorem} \label{thm:correlafvector}
The reverse $f$-polynomials  of the bounded correlatrons $h_1(t)=t, h_2(t)=3t+4t^2, h_3(t)=11t+35t^2+25t^3, \ldots$
form the unique sequence $\{h_n(t)\}_{n \geq 1}$ of polynonmials such that the power series in $(\R[t])[x]$
\begin{eqnarray*}
&& x + h_1(t)x^2 + h_2(t)x^3 + h_3(t)x^4 + \ldots \\
&& x - tg_1(t)x^2 - tg_2(t)x^3 - tg_3(t)x^3 + \ldots 
\end{eqnarray*}
are compositional inverses, where  $g_n(t) = [(1+t)f_n(t) - 1]/t$ and $f_n$ is the $f$-polynomial of the $n$th cosmohedral fan.
Equivalently, the power series $H_t(x) = x + h_1(t)x^2 + h_2(t)x^3 + \ldots$ satisfies
\[
H_t\left(\frac{(t+1)x^2-tx}{1-x} + (t+1)A_t(x)\right) = x
\]
where $A_t(x)$ is the generating function for the $f$-vectors of cosmohedra in Theorem \ref{thm:cosmofvector}. 
\end{theorem}

\begin{proof}
Every face $F$ of the truncated correlatron corresponds to a pair $(\C,M)$ where $M$ is non-trivial. Let $\C$ create a subdivision into polygons $P_1, \ldots, P_k$, and let $A(F)$ be the corresponding face of the associahedal fan corresponding to the set of chords $\C$.
A Matryoshka $M$ compatible with $\C$ restricts to Matryoshkas $M_1, \ldots, M_k$ on $P_1, \ldots, P_k$. 
One can almost recover $M$ from the $M_i$s: the only additional information is whether $M$ contains the polygon $P_i$ (in which case we let $\varepsilon_i=1$) or it does not (in which case $\varepsilon_i=0$). We note that if $M_i$ is the trivial Matryoshka on $P_i$, then $\varepsilon_i=1$ necessarily; if $M_i$ is non-trivial, then $\varepsilon_i$ can be chosen freely. Finally notice that
\[
\codim F = |\C| + |M_{\text{non-min}}| = k + \sum_{i=1}^k |(P_i)_{\text{non-min}}| + \sum_{i=1}^k (\varepsilon_i - 1)
\]
Therefore
\begin{eqnarray*}
h_n(t) &=& \sum_{F \leq \BdCorrela_n} t^{\codim F} \\
&=& 
\sum_{\substack{A \leq \Assoc_n \\ \S(A) = \{P_1, \ldots, P_k\}}}
t^{k} \sum_{\substack{M_1, \ldots, M_k \\ \varepsilon_1, \ldots, \varepsilon_k}} \left(\prod_{i=1}^k t^{|(P_i)_{\text{non-min}}|  + \varepsilon_i -1}\right) \\
&=& \sum_{A \leq \Assoc_n} 
 t^{|\C|} \prod_{i=1}^k \left(1 + \sum_{M_i \neq \{P_i\}} t^{|(P_i)_{\text{non-min}}|}\left(1+\frac1t\right)\right) \\
&=& \sum_{A \leq \Assoc_n} 
t^{|\C|} \prod_{i=1}^k \left(1 + (f_{P_i}(t)-1)\left(1+\frac1t\right)\right) \\
&=& \sum_{A \leq \Assoc_n} 
t \, \prod_{i=1}^k \left((1 + t) f_{P_i}(t)-1\right) \\
&=& \sum_{A \leq \Assoc_n} t g_A(t),
\end{eqnarray*}
again using the notation of \cite{AguiarArdila}. The associahedral version of Lagrange inversion \cite[Theorem 2.4.3]{AguiarArdila} then tells us  that $x - tg_1(t)x^2 - tg_2(t)x^3 - tg_3(t)x^4 + \ldots$ and $ x + h_1(t)x^2 + h_2(t)x^3 + h_3(t)x^3 + \ldots$ are compositional inverses. 

The second formulation follows from a straightforward computation.
\end{proof}

Note the remarkable similarity between the formulas in Theorems \ref{thm:cosmofvector} and \ref{thm:correlafvector}
for the $f$-vectors of the cosmohedron and bounded correlatron. Is this an indication of something deeper?

\section{\textsf{Elusive qualities of the cosmohedron} \label{sec:subtleties}}

The cosmohedron is related to several beautiful polytopes in the literature, either through direct connections or indirect analogies. These include the permutahedron \cite{BooleStott, Schoute}, the associahedron \cite{Loday, Stasheff}, graph associahedra \cite{CarrDevadoss}, the nested permutahedron \cite{CastilloLiu}, the permutoassociahedron \cite{CastilloLiu2}, and the
permutonestohedron \cite{Gaiffi},
among others. 

However, there are numerous ways in which the combinatorial structure of the cosmohedron is more subtle than its predecessors. In turn, this makes it significantly harder to prove the correctness of this construction. In this section, we discuss some of the intricacies that arise -- and that remain in future polytopes of interest.

\bigskip
\noindent
\textsf{\textbf{The cosmohedron is not very symmetric.}}
The permutahedron, the nested permutahedron, the permuto-associahedron, and the permutonestohedron are symmetric under the natural action of $S_n$ onto $\R^n$. This allows us to express them as the convex hull of the $S_n$-orbit of a point, or of a carefully placed permutahedron or associahedron, respectively. The cosmohedron only has a natural action of the dihedral group. 

The cosmohedron is made of an associahedron's worth of different bracket associahedra. It requires much greater care to choose a suitable realization of each bracket associahedron that are compatible with each other, and then place each one of them at the correct position, in order to make sure that they glue correctly into a well-behaved global polytope. This is a subtle matter.

\bigskip
\noindent
\textsf{\textbf{The cosmohedron is not simplicial.}} This is not a surprise; none of the polytopes on the above list are simplicial. 
Proposition \ref{prop:factorization} shows that the only simplicial faces of the cosmohedron are the vertices and edges.

\bigskip
\noindent
\textsf{\textbf{The cosmohedron is not simple.}} 
This is more surprising, because all the other polytopes in the list above are simple.
The only simple vertices of the cosmohedron are those whose Matryoshka is nested linearly. To see this, let $M=(T,B)$ be a maximal Matryoshka.
The facets of $\cone(M)$ correspond to the facets of $\cone_T(B)_{\geq 0}$. These in turn correspond to the $n-2$ facets of $\cone_T(B)$ (which is simplicial), and the facets of the form $y_e \geq 0$ for a minimal bracket in $B$, which is an edge in $B$, and corresponds to a quadrilateral in $M$. Therefore 
\[
(\text{number of facets of } \cone(M)) = n - 2 + (\text{number of quadrilaterals in } M).
\]
This can be any number between $n-1$ and $\lfloor 3n/2 \rfloor - 2 $. It equals $n-1$ precisely if $M$ has exactly one quadrilateral, which happens precisely when $M$ is linearly nested. 

\textsf{Why is this a challenge?}
The fact that polytopes like (bracket) associahedra are simple allows for a flexible construction. One can start with an easy polytope (in these cases a simplex), and then create new facets one at a time by making cuts that are ``not too deep", so that the facets meet transversally. The required care only involves depth inequalities; see \cite{Loday, Devadoss, PPP}.

We \textbf{do} construct the cosmohedron through a chiseling of faces of the associahedron, but we have to be much more careful now: the chiseling has to line up perfectly to create the required non-transversal crossings. The required care now involves many equalities as well as inequalities. In $\AFV_{n-1}$, these equalities manifest in the fact that the linear forms $\sum_{C \in \C}X_C$ satisfy many linear relations between them. 
In $\Cosmo_{n-1}$, they manifest in the many  requirements $b_B = b_{B'}$ in Remark \ref{rem:b=b}.

\bigskip 
\noindent
\textsf{\textbf{The cosmohedron depends delicately on the chiseling.}}
To further illustrate the point above, recall that in our construction of the cosmohedron, we chiseled a Devadoss bracket associahedron at each vertex of the associahedron. This choice is important. For instance, despite their usefulness in other settings, the Postnikov bracket associahedra do not produce the cosmohedron upon chiseling.

\bigskip
\noindent
\textsf{\textbf{We don't know the cosmohedron to be a deformation of an easy polytope.}} 
The \emph{deformations of the permutahedron}, also known as \emph{generalized permutahedra} or \emph{polymatroids}, are obtained from the permutahedron by moving facets without pushing them past vertices. They are characterized by having the same edge directions as the permutahedron, namely  $e_i - e_j$ for $i \neq j$ \cite[Prop. 2.6]{ACEP}, \cite{Postnikov}.

Generalizing this setup, Castillo and Liu \cite{CastilloLiu} defined the \emph{nested permutahedron} $\Perm(\alpha, \beta)$ by placing an $(n-1)$-permutahedron at each corner of the $n$-permutahedron. They studied its deformations, and showed that Kapranov's  \emph{permutoassociahedron} \cite{Kapranov} -- which has an $(n-1)$-associahedron at each vertex of the $n$-permutahedron, can be obtained as a deformation of $\Perm(\alpha, \beta)$.

The cosmohedron would seem to fit in the same framework; our proof of Theorem \ref{thm:cosmofan3} even shows that the cosmohedron has the same edge directions as $\Perm(\alpha, \beta)$, namely $e_i-e_j$ and $e_i-e_j+e_k-e_l$ for distinct $i,j,k,l$.
This makes it natural to wonder: Is our cosmohedron a deformation of the nested permutahedron? \textbf{It is not.}

Dually, we claim that our cosmohedral fan is not a coarsening of the nested braid fan  $\Br_n^2$ of Castillo and Liu \cite{CastilloLiu}.
The $n!\cdot(n-1)!$ facets of $\Br_n^2$ correspond to the choices of an order $\sigma \in S_n$ of the coordinates of $x \in \R^n$ and an order $\tau \in S_{n-1}$ of their consecutive differences. For example, the set of points $x \in \R^4$ with

 \qquad $x_3 > x_1 > x_4 > x_2$ \qquad and \qquad 
 $x_3-x_1 > x_1-x_4 > x_4-x_2 > 0$

\noindent is a maximal open cone of $\Br_4^2$. Such a point must satisfy two inequalities $x_3-x_4 > x_3-x_1$ and $x_1-x_2 > 0$ of the Matryoshkal $\cone(M_2)$ of Figure \ref{fig:twoMatryoshkas}, but it may or may not satisfy the third inequality $x_3-x_1 > x_1-x_2$. Therefore, this maximal cone of $\Br_4^2$ intersects $\cone(M_2)$ full-dimensionally, but is not contained in it.
This proves our claim. 

\textsf{Why is this a challenge?} Generalized permutahedra are in bijection with submodular functions on the subsets of $[n]$. Thus to construct them, it is sufficient to choose the correct submodular function, and to use the theory of polymatroids and generalized permutahedra to prove the correct combinatorial structure; see \cite{AguiarArdila, Postnikov}. 

We do not know an analog to the permutahedron for this setting.
It would be very interesting to find a nice polytope such that (i) it can be deformed to obtain the cosmohedron, (ii) its deformation cone can be nicely described combinatorially, and (iii) the face structure of its deformations is combinatorially well-behaved.

\bigskip
\noindent
\textsf{\textbf{We don't know the cosmohedron to be a Minkowski sum of easy polytopes.}} 
The permutahedron, Loday's associahedron, and Postnikov's graph and bracket associahedra can all be expressed as Minkowski sums of faces of the standard simplex. This gives a simple proof that their face structure-- the common refinement of the easy face structures of the summands -- is correct. For details, see \cite{Postnikov}. 

We do not know how to do this in our setting. It would be interesting to find a realization of the cosmohedron that can be decomposed easily and usefully into a Minkowski sum.

\bigskip 
\noindent
\textsf{\textbf{The faces of the cosmohedron only factor combinatorially.}}
As discussed in Lemma \ref{lemma:bracketfactorization}, faces of a generalized permutahedra factor into smaller generalized permutahedra \cite{AguiarArdila, Edmonds}. For permutahedra, associahedra, and graph associahedra, any facet of a polytope in these families is a product of smaller polytopes in the same family. This gives an inductive strategy to prove their combinatorial structure, introduced for associahedra in \cite{MarklShniderStasheff}. One simply needs to guarantee that the facets are placed correctly; again, because these polytopes are simple, this step only requires verifying the appropriate depth inequalities. 
Once the facets are laid out correctly, thanks to their factorization, the rest of the face structure can be proved recursively. 

In our realization of the cosmohedron, each one of the facets is  \textbf{combinatorially, but not geo\-metrically} a product of cosmohedra and a bracket associahedron. Therefore, its combinatorial structure cannot be proved using this inductive strategy.
It would be interesting to find a realization of the cosmohedron where the face factorization of Proposition \ref{prop:factorization} holds geometrically.

\bigskip 
\noindent
\textsf{\textbf{Outlook: chiseled polytopes.}}
In an upcoming paper, we present a general family of \emph{$\X$ in $Y$} polytopes that contains all of the polytopes above, and numerous other polytopes of interest, old and new. Our proof strategy for the combinatorial structure of the cosmohedron is very well suited to this general setting.
In the Appendix, we provide some physical motivation for this general construction, by sketching a new motivating example: we introduce the \emph{one-loop $\mathcal{U}$-polytope}, which captures the combinatorics of all one-loop cubic (trivalent) graphs, as well as the Feynman polytope for the Symanzik  $\mathcal{U}$-polynomial of each graph, all in one polytope.

\section{\textsf{Acknowledgements}}

The work of Federico Ardila--Mantilla was supported by the  National Science Foundation (Grant DMS 2154279) in the United States, a Wolfson Fellowship of The Royal Society in the United Kingdom, and 
the Friends of the Institute for Advanced Study Fund at the IAS.
The work of Nima Arkani-Hamed was supported by the DOE (Grant No. DE-SC0009988), the Simons Collaboration on Celestial Holography,  the European Union (ERC, UNIVERSE PLUS, 101118787), and the Carl B. Feinberg cross-disciplinary program in innovation at the IAS.
The work of Carolina Figueiredo was supported by FCT/Portugal (Grant No. 2023.01221.BD). 
The work of Francisco Vaz\~ao was supported by the Jonathan M. Nelson Center for Collaborative Research. 
Views and opinions expressed are those of the author(s) only and do not necessarily reflect those of the European Union or the European Research Council Executive Agency. Neither the European Union nor the granting authority can be held responsible for them. 

\appendix

\section{\textsf{Appendix: $\X$ in Y polytopes}} \label{sec:XinY}

The main focus of this paper has been to explore and prove the mathematical properties of cosmohedra, which encode the cosmological wavefunction of $\Tr(\phi^3)$ theory. We use this appendix to discuss how the principles that guided the construction of cosmohedra are of much broader relevance. 

\subsection{\textsf{The cosmohedron as a guiding example for $\X$ in $Y$ polytopes}}

As we have learned by now, cosmohedra encode geometrically the combinatorics of maximal Matryoshkas of an $n$-gon. One way to organize all such Matryoshkas is by dividing them into classes labeled by the underlying collection of triangles -- \textit{i.e.} the underlying triangulation. This is what naturally occurs in physics. To write down the cosmological wavefunction, we start by listing all possible cubic (trivalent) graphs, which correspond to all possible triangulations. Then, for each cubic graph, we consider all possible ways of finding maximal collections of nested subgraphs, which give the possible maximal Matryoshkas compatible with the respective underlying triangulation. The cosmohedron is then the object that puts both combinatorial problems together. In particular, to find the cosmohedron we first needed to find two other polytopes: 
\begin{enumerate}[itemsep=0em]
  \item[Y.] The underlying polytope whose vertices label all possible triangulations/diagrams -- which at tree-level is simply the associahedron;
  \item[$\X$.] A class of polytopes, one associated to each triangulation, whose vertices label all possible Matryoshkas compatible with the respective triangulation -- which is done by the bracket associahedon $X_T$ for each triangulation $T$.
\end{enumerate}

Once we have these two objects under control, we ``blow up" each vertex $v_T$ of $Y$ into the respective $\X$ polytope $X_T$, via the chiseling procedure described earlier. At the level of the normal fans, we subdivide each cone of the fan of $Y$ using the corresponding normal fan in the family $\X$. 

Therefore, for any given combinatorial problem, where on top of listing all possible diagrams ($Y$), we also want to keep track of some further information $X_G$ of \textbf{each} diagram $G$ ($\X$). Once we find the polytopes for $\X$ and $Y$ separately, we can seek to define the polytope that captures the full combinatorics ($\X$ in $Y$) by the procedure described above. 

As it turns out, such situations are ubiquitous in physics, and not just in the mathematical settings discussed in Section \ref{sec:subtleties}. In this appendix, we will discuss a particular example, simply to illustrate how this construction generalizes to other contexts.

\subsection{ \textsf{The $\U$-polytope for one-loop-integrated amplitudes}}

At loop-level, the amplitudes are expressed as a sum of \emph{Feynman integrals}, which often diverge at the boundaries of the integration domain. Thus, it is important to record the asymptotic behavior of the integrands. The most transparent way to do this is by using the projective representation of Feynman integrals in terms of \emph{Schwinger parameters} $\alpha_e$ -- one for each edge of the Feynman diagram $G$ -- where they take the following form:
\[
\label{eq:IG}
    \mathcal{I}_G = \Gamma(d)
    \int_{\R\mathbb{P}_{>0}^{E-1}} \frac{\prod_{e \in G} d \alpha_e }{\text{GL}(1)} \, \frac{1}{\mathcal{U}_G^{D/2-d} \mathcal{F}_G^d}. 
\]
Here $d= E - L D/2$ is the \emph{superficial degree of divergence} of a graph $G$ with $E$ edges and genus $L$ (number of independent loops) in spacetime dimension $D$, and $\mathcal{U}_G$ and $\mathcal{F}_G$ are the first and second \emph{Symanzik polynomials} of $G$:
\[
  \mathcal{U}_G = \sum_{T \text{ tree of } G} \, \prod_{e \notin T} \alpha_e, 
  \qquad
  \mathcal{F}_G = \sum_{T \text{ 2-tree of } G} \, \prod_{e \notin T} \alpha_e.
\]
The sum in $\U_G$ is over all spanning trees of  $G$, and the sum in $\F_G$ is over all spanning \emph{2-trees} $T_1 \sqcup T_2$, where $T_1$ and $T_2$ are two vertex-disjoint trees that span all vertices of $G$.
To correctly describe the divergences of $\mathcal{I}_G$, one needs to characterize the asymptotic behavior of $\mathcal{U}_G$ and $\mathcal{F}_G$.

Therefore, at one-loop level, we are looking for an object that captures not only all possible cubic (trivalent) one-loop graphs $G$ (collected in a polytope $Y$), but also the possible leading monomials in $\mathcal{U}_G$ and $\mathcal{F}_G$ for each graph $G$ (collected in a family of polytopes $\X$). As it turns out, the polytopes Y and $\X$ polytopes are known in the physics and mathematics literature \cite{NimasTalk, Arkani-Hamed:2022cqe, Backus:2025hpn, Edmonds}. We will show how to combine them together into a single chiseled polytope. This answers a question posed in \cite{Arkani-Hamed:2023lbd}.

In this Appendix, we illustrate the essential ideas behind the construction in a simple but already interesting case: the asymptotic behaviors of the $\U$-polynomials in the one-loop setting. 
In an upcoming paper, we will treat the general setting of higher loops, and consider both Symanzik polynomials $\mathcal{F}$ and $\mathcal{U}$.

\subsubsection{\textsf{The $Y$ polytope: the one-loop associahedron}}

The \emph{one-loop associahedron} is a polytope designed to parameterize the triangulations of a polygon $\Pentagon_n$ with one puncture. Consider an $n$-gon with vertices cyclically labeled $1,2, \ldots, n$ and a puncture labeled $O$. We consider curves of two kinds:
\begin{enumerate}
\itemsep0em
    \item \emph{chords} $x_{i,j}$ that go from $i$ to $j$ while keeping the puncture $O$ to the right, for $1 \leq i, j \leq n$. (Note that we have \emph{diagonals} $x_{i,j} \neq x_{j.i}$ for each $i \neq j$, and \emph{tadpoles} $x_{i,i}$ for each $i$.)

\item \emph{spokes} $y_k$ and $\tilde{y}_k$  that spiral from $k$ into the puncture clockwise and counterclockwise, respectively. (For simplicity, we draw straight lines from $k$ to $O$ tagged in blue and red, respectively.)
\end{enumerate}

We are interested in the triangulations of $\Pentagon_n$ using these curves. Notice that every triangulation uses exactly $n$ curves, and
contains at least one spoke $y_k$ or $\tilde{y}_k$.
Also, notice that a clockwise and a counterclockwise spiral always intersect, so every triangulation is monochromatic. The left panels of
Figures \ref{fig:2pt1LoopFan} and \ref{fig:3pt1LoopFan} show the $6$ triangulations (and six coarser subdivisions) for $n=2$, and the $10$ blue triangulations for $n=3$, respectively

The \emph{loop associahedron} was first introduced in \cite{NimasTalk}; it is a variant of the cluster polytopes in \cite{Arkani-Hamed:2019vag} that includes the \emph{tadpole} curves, $x_{i,i}$. For the punctured $n$-gon, we get an $n$-dimensional simplicial polytope whose vertices correspond to the triangulations of $\Pentagon_n$. Its faces correspond to the partial triangulations or subdivisions of $\Pentagon_n$, and the containment relations of the faces match the reverse inclusion relations of the sets of curves. The left panels of Figures  \ref{fig:embed2pt} and
\ref{fig:embed3pt} 
respectively illustrate the one-loop associahedra for $\Pentagon_2$ and $\Pentagon_3$.

\subsubsection{\textsf{The $\X$ polytopes:  Feynman polytopes}}

The asymptotic behavior of the Symanzik polynomials $\mathcal{U}_G$ and $\mathcal{F}_G$ is fully characterized by their Newton polytopes in $\R^{E(G)}$:
\[
U_G = \Newt(\U_G), \qquad F_G = \Newt(\F_G)
\]
called \emph{Feynman polytopes} and studied in \cite{Arkani-Hamed:2022cqe}.

Since $\U_G$ and $\F_G$ are squarefree, $U_G$ and $F_G$ are $0$-$1$ polytopes. 
The exponents of $\U_G$ (resp. $\F_G$) are the complements of the spanning trees (resp. $2$-trees) of graph $G$, which are the bases of the matroid $M(G)$ (resp. the truncation matroid $\Tr M(G)$). It follows that 
\[
U_G = P_{M(G)^\perp},\qquad F_G = P_{(\Tr M(G))^\perp}
\]
where $P_M$ denotes the matroid (basis) polytope of $M$ and $^\perp$
denotes matroid duality.

In this appendix, for clarity and simplicity, we focus on the first Symanzik polynomials in the one-loop setting. 
When $G$ only has one cycle, the polynomial $\mathcal{U}_G$ is simply the sum of all variables $\alpha_e$ associated to the edges $e$ of that cycle, and the Feynman polytope is just the standard simplex on those edges. We summarize this by writing that  $U_G = \Delta_{\Cycle(G)}$ in the one-loop setting. Note that, unlike in the setting of cosmohedra, this polytope is usually low-dimensional.

\subsubsection{\textsf{Towards a general construction: $\X$ in $Y$ fans in dimensions 2 and 3}}

To motivate the general construction, let us first carry it out in detail in dimensions 2 and 3, beginning at the level of fans.
The fan of the one-loop associahedron has a ray for each curve of $\Pentagon_n$, and a face $\cone(\C)$ for each collection of  $\mathcal{C} \subset \Curves(\Pentagon_n)$ of
compatible curves, or equivalently, each subdivision of $\Pentagon_n$. We need to subdivide the cones of this fan to keep track of the dominant terms of the corresponding $\U$-polynomials.

\begin{figure}[h]
    \centering    \includegraphics[width=.8\linewidth]{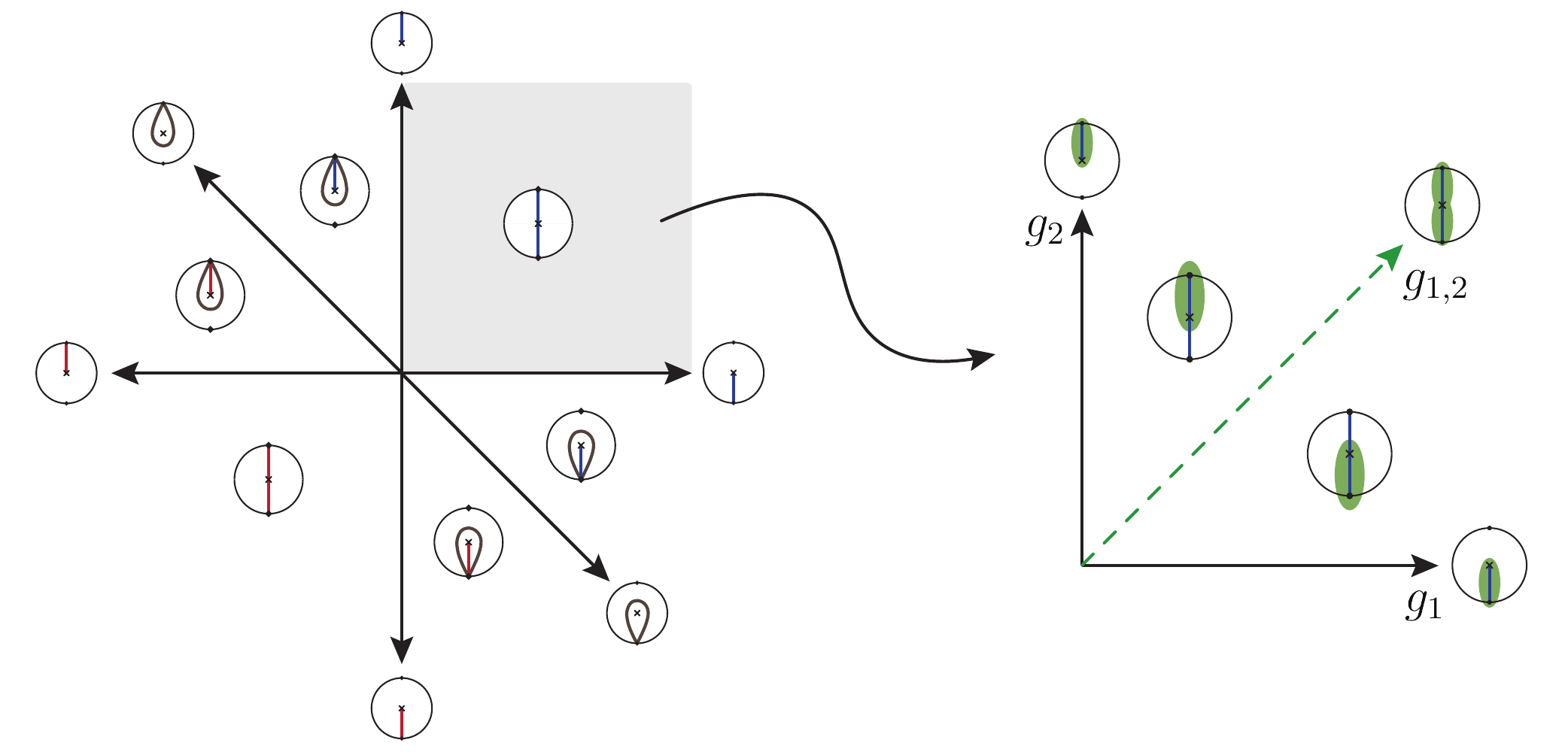}
    \caption{
    (a) Fan of the two-dimensional one-loop associahedron.
     (b) Refinement of a cone in the fan to encode the dominating terms of the $U$-polynomial.   \label{fig:2pt1LoopFan}
}
\end{figure}

 \begin{example} For $n=2$ there are three blue and three red triangulations of $\Pentagon_2$. Their cones naturally fit into a complete fan in the plane shown in the left panel of Figure \ref{fig:2pt1LoopFan}. The only diagrams with a non-trivial $\mathcal{U}$ polynomial are the ones with two spokes, corresponding to the positive and negative quadrants of the fan. 
The dual diagram of these triangulations is the bubble diagram with two parallel edges $\alpha_1$ and $\alpha_2$; its $\mathcal{U}$-polynomial is simply:
$
    \mathcal{U} 
    \left(
     \begin{gathered}
    \begin{tikzpicture}[line width=0.6,scale=0.4,baseline={([yshift=0.0ex]current bounding box.center)}]
        \coordinate (1L) at (-2.,0);
        \coordinate (1M) at (-1.,0);
        \coordinate (2R) at (2.,0);
        \coordinate (2M) at (1.,0);
        \draw[] (1L) -- (1M);
        \draw[] (2R) -- (2M);
        \draw (0,0) circle (1);
         \node[scale=0.7] at (0,0.6) {$\alpha_2$};
         \node[scale=0.7] at (0,-0.6) {$\alpha_1$};
    \end{tikzpicture}
\end{gathered}  
\right) =  \alpha_1 +\alpha_2,
$. 

We wish to divide each such cone into  two subcones, depending on which term of $\U$ dominates: $\alpha_1$ or $\alpha_2$. We can easily do this by subdividing them in half, adding the rays  $g_{1,2} = g_1 + g_2$ and $\tilde{g}_{1,2} = \tilde{g}_1 + \tilde{g}_2$ respectively. Here $g_i$ are the usual \emph{$g$-vectors} for the spokes curves, which are the rays of the loop associahedral fan. This subdivision is shown in the right panel of Figure \ref{fig:2pt1LoopFan}. 
\end{example}

\begin{example} 
For $n=3$ there are 10 blue (and  10 red) triangulations of the punctured disk with three marked points on the boundary $\Pentagon_3$. They label the maximal cones of the fan shown on the left of Figure \ref{fig:3pt1LoopFan}; this is the blue half of the one-loop associahedral fan. The red half is isomorphic, and is glued to this one along the boundary in the opposite orientation. Together they form a complete three-dimensional fan.

\begin{figure}[h]
    \centering
    \includegraphics[width=\linewidth]{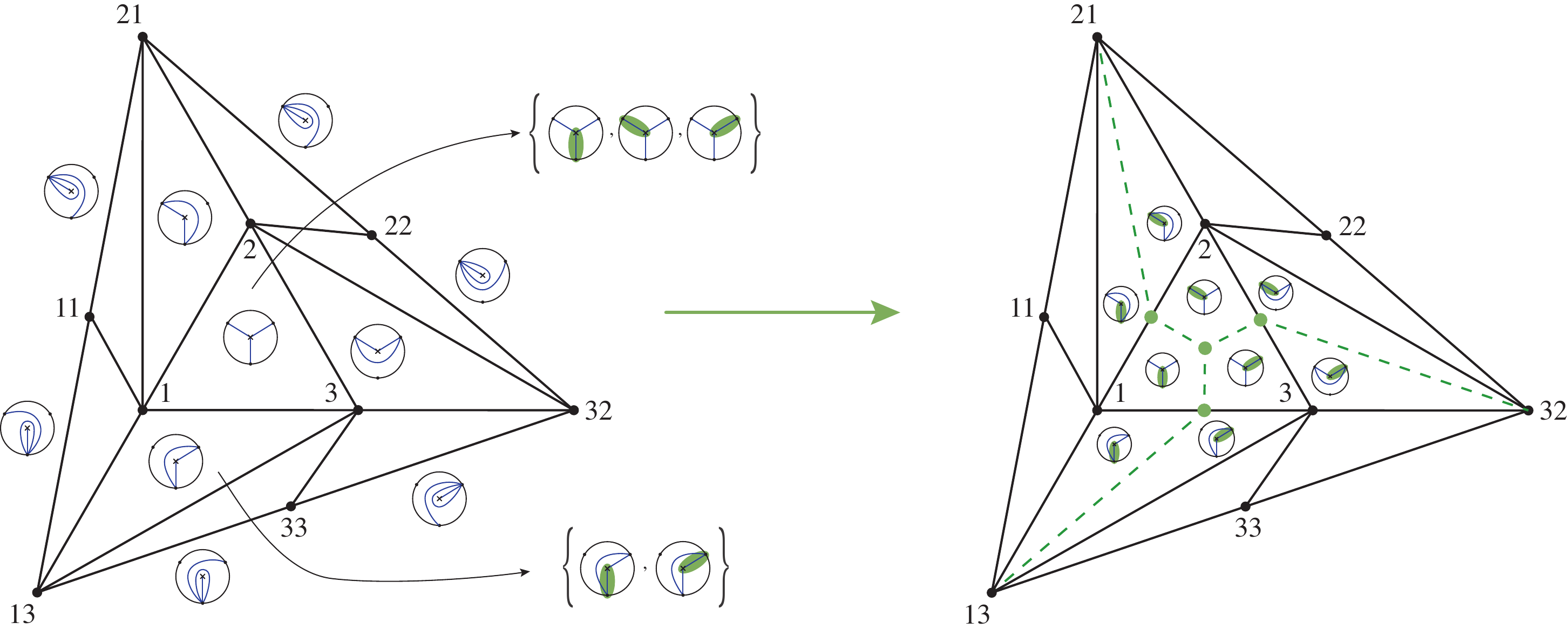}
    \caption{(a) Half-fan of the three-dimensional one-loop associahedron.  (b) Half-fan of the three-dimensional one-loop $\U$-polytope. }
    \label{fig:3pt1LoopFan}
\end{figure}

Again, for each triangulation $\C$ with two spokes $y_i$ and $y_j$, we introduce a new ray $g_{ij} = g_i + g_j$ that subdivides the corresponding cone in half. But now there is also a triangulation with three spokes $y_1, y_2, y_3$, with $\U$ polynomial $\alpha_1 + \alpha_2 + \alpha_3$. We subdivide its face $\cone(g_1,g_2,g_3)$ by introducing four new rays 
$    g_{1,2} = g_1 + g_2$, 
$g_{2,3} = g_2 + g_3,$
$ g_{1,3} = g_1 + g_3,$ and 
that span three full-dimensional subcones: 
$\cone_1 = \cone(g_1, g_{1,2}, g_{1,3}, g_{1,2,3})$,  
$\cone_2 = \cone(g_2, g_{1,2}, g_{2,3}, g_{1,2,3})$, and  
$\cone_3 = \cone(g_3, g_{1,3}, g_{2,3}, g_{1,2,3})$.  
This is illustrated on the right side of Figure \ref{fig:3pt1LoopFan}. Already in this example, we see that the fan is not simplicial, and therefore the  $\mathcal{U}$-polytope will not be simple.
\end{example}

\subsubsection{\textsf{The normal fan of the $\U$-polytope}}

Now we are ready to describe the general construction of the $\U$-fan. 

\bigskip
\noindent
\textsf{\textbf{Fan of the one-loop associahedron.}} 
The fan of the one-loop associahedron \cite{NimasTalk} has a ray for each curve of $\Pentagon_n$, and a face $\cone(\C)$ for each collection of compatible curves, $\mathcal{C} \subset \Curves(\Pentagon_n)$, or equivalently, each subdivision of $\Pentagon_n$. This is a simplicial fan whose cones are ordered by reverse inclusion:  
\[
\cone(\C) \subseteq \cone(\C') \quad \text{  if and only if } \quad  \C \subseteq \C'.
\]

 \bigskip
\noindent
\textsf{\textbf{Fan of the $\mathcal{U}$ polytope.}} 
Now consider a facet $\cone(\C)$ of the one-loop associahedron, corresponding to a triangulation $\C$ of $\Pentagon_n$. In the polytope, we wish to replace the vertex $v_{\C}$ with a simplex $\Delta_{\Spokes(\C)}$, since the spokes of $\C$ correspond to the edges in the unique loop of the cubic diagram dual to $\C$.
Therefore 
we need to subdivide the simplicial face $\cone(\C)$ into $|\Spokes(\C)|$ maximal cones. 
If $\cone(\C)$ is regular then the $s$-th cone $\cone(\C,s)$ consists of the points in $\cone(\C)$ whose closest spoke-ray is $g_s$. 
If it is not, we linearly transform it into a regular cone, and use that subdivision. Explicitly, we obtain new rays
\[
    g_\mathcal{S} = \sum_{s \in \mathcal{S}} g_s,
    \qquad    \tilde g_\mathcal{S} = \sum_{s \in {\mathcal{S}} }\tilde g_s 
\]
for each subset $\mathcal{S} \subseteq [n]$ and more generally, a new cone
\[
\cone(\C, \S) = 
\cone(
\{g_{ij} \, : \,  x_{ij} \in \Chords(\C)\} 
\cup 
\{g_{T} \, : \, S \subseteq  T \subseteq \Spokes(\C)\}
)
\]
for each non-crossing set of curves $\C$ and each non-empty subset $\emptyset \neq \S \subseteq \Spokes(\C)$ -- where we allow the exception $\S=\emptyset$ for the sets $\C$ with no spokes. 

Note that $\cone(\C, \S)$  is a $|\Chords(\C)|$-fold cone over a $(|\Spokes(\C)| - |S|)$--cube. 
We have 
\[
\cone(\C,\S) \subseteq \cone(\C',\S') \quad \text{  if and only if } \quad  \C \subseteq \C' \text{ and } \S \supseteq \S',
\]
 The blue and red fans glue along the ``equatorial complex" of (necessarily lower-dimensional) faces $\cone(\C,\emptyset)$ such that $\C$ has no spokes.

\subsubsection{\textsf{The $\U$-polytope}}

We now construct a geometric realization of the $\mathcal{U}$-polytope, given by the appropriate chiseling procedure that creates, at each vertex of the underlying $Y$-polytope, the respective $\X$-polytope.  We understood above how to subdivide the normal fan of the loop associahedron to get the normal fan of the $\mathcal{U}$-polytope -- and therefore we have its new facet normals $g_{\S}$ and $\tilde{g}_{\S}$ for $\S \subseteq [n]$ . However, as we saw for the cosmohedron, it is a subtle matter to determine how deeply we need to chisel in each direction to obtain the correct polytope. This is controlled by the constants $\epsilon_{\mathcal{S}}$ and $\tilde{\epsilon}_{\mathcal{S}}$.appearing in the new facet inequalities $g_{\S}\cdot x \geq \epsilon_{\mathcal{S}}$ and $\tilde{g}_{\S}\cdot x \geq \tilde{\epsilon}_{\mathcal{S}}$. 

Given a vector $\c$ of positive parameters $c_{rs}$, $c_r$ and $\tilde{c}_r$ for $1 \leq r, s \leq n$, the \emph{one-loop associahedron} $L_1(\c)$ defined in \cite{NimasTalk, Backus:2025hpn}
 is the $n$-dimensional polytope 
\begin{equation}
\begin{aligned}
   L_1(\c) =  \Bigg\{(x,y,\tilde{y})  &\in \R^{\Curves(\Pentagon_{n+2})} \, : \, \left[ x_{r,s} \geq 0, \, y_{r}\geq 0, \, \tilde{y}_r \geq 0 \,  \right], \\
    &
    \left[ \begin{aligned}
    &x_{r,r} + x_{r+1,r+1} -x_{r+1,r} -y_{r} -\tilde{y}_{r-1} = c_{r+1,r}, \\
    &x_{r,s+1} + x_{r-1,s} - x_{r, s} - x_{r-1, s-1} = c_{rs} \text{ for }r \neq \{s+1,s\}, \\
    &y_r + \tilde{y}_r - x_{r+1,r+1} = c_{r+1}.
    \end{aligned}\right] \Bigg\}
\end{aligned} 
\label{eq:loopassoc}
\end{equation}
where we set to $0$ any variables that are undefined because their subscripts are outside of the prescribed ranges. 
In this equation, we are writing vectors in 
$ 
\R^{\Curves(\Pentagon_{n+2})}$ 
in the form
$(x,y,\tilde{y}) \in \R^{\Chords(\Pentagon_{n+2})} \times \R^{\Spokes(\Pentagon_{n+2})} \times \R^{\widetilde{\Spokes}(\Pentagon_{n+2})}
$
where $x=(x_{r,s})_{1 \leq r, s \leq n}$, 
$y=(y_{r})_{1 \leq r  \leq n}$, and
$\tilde{y}=(\tilde{y}_{r})_{1 \leq r  \leq n}$.

To construct our $\U$-polytope, we need to introduce new facet inequalities of the form:
\begin{equation}
	 \sum_{k\in \mathcal{S}} y_k \geq \epsilon_{\mathcal{S}},  \quad   
    \sum_{k\in \mathcal{S}} \tilde{y}_k \geq \tilde{\epsilon}_{\mathcal{S}} \qquad \text{ for } \emptyset \neq \S \subseteq [n]
    \label{eq:facet_ineqs}
\end{equation}
for the possible subsets of spokes. 
A first condition on the  $\epsilon$s is that they should be small enough in relation to the $c$s that the chiseling will not remove any faces of $L_1(\c)$.
The parameters must also satisfy the equalities
\begin{equation}
    \epsilon_{\mathcal{S}_1} + \epsilon_{\mathcal{S}_2} = \epsilon_{\mathcal{S}_1\cup \mathcal{S}_2} +\epsilon_{\mathcal{S}_1\cap \mathcal{S}_2} \text{ for any subsets } \mathcal{S}_1,  \mathcal{S}_2 \text{ with } \mathcal{S}_1\cap \mathcal{S}_2 \neq \emptyset
\label{eq:modular} \end{equation}
since the corresponding rays of the normal fan satisfy  $g_{\mathcal{S}_1} + g_{\mathcal{S}_2} = g_{\mathcal{S}_1\cup \mathcal{S}_2} +g_{\mathcal{S}_1\cap \mathcal{S}_2}$, and they lie in the same cone $\cone([n], \{y_s\})$ of the fan for any $s \in \mathcal{S}_1\cap \mathcal{S}_2$.
In addition, the $\epsilon$s must  satisfy
the \emph{wall-crossing inequalities}
\begin{equation}
    \epsilon_{i,j} > \epsilon_i + \epsilon_j \qquad \text{ for } i \neq j
    \label{ineq:modular}
\end{equation}
since the corresponding rays satisfy $g_{i,j} = g_i + g_j$, and  $\cone(g_i, g_{i,j})$ and
$\cone(g_j, g_{i,j})$ are distinct cones in the fan.
The analogous conditions hold for the $\tilde{\epsilon}$'s. One may verify that these are all the necessary and sufficient conditions to obtain the desired $\U$-polytope.

The solutions to the system of equations \eqref{eq:modular} are of the form 
\begin{equation}
    \epsilon_{\{s_1, \ldots, s_k\}} = \epsilon_{s_1} + \cdots + \epsilon_{s_k} + (k-1)\epsilon 
\end{equation}
for any choice of $\epsilon_1, \ldots, \epsilon_n, \epsilon$, and the inequalities \eqref{ineq:modular} are satisfied when $\epsilon > 0$. 
The same analysis holds for the $\tilde{\epsilon}$'s.
Thus a particularly simple choice of parameters that cuts out a $\U$-polytope is
\begin{equation}
\epsilon_{S} = \tilde{\epsilon}_{S} = (|S|-1)\epsilon\, .
\end{equation}
for small $\epsilon >0$. 
Summarizing, the $\U$-polytope is
\[
\U(\c) = \{(x,y,\tilde y) \in L_1(\c) \, : \, \sum_{k\in \mathcal{S}} y_k \geq (|S|-1)\epsilon,  \quad   
    \sum_{k\in \mathcal{S}} \tilde{y}_k 
    \geq (|S|-1)\epsilon
    \text{ for } \emptyset \neq \S \subseteq [n]\}
\]

\begin{figure}[h]
\begin{center}
\includegraphics[width=3.5in]{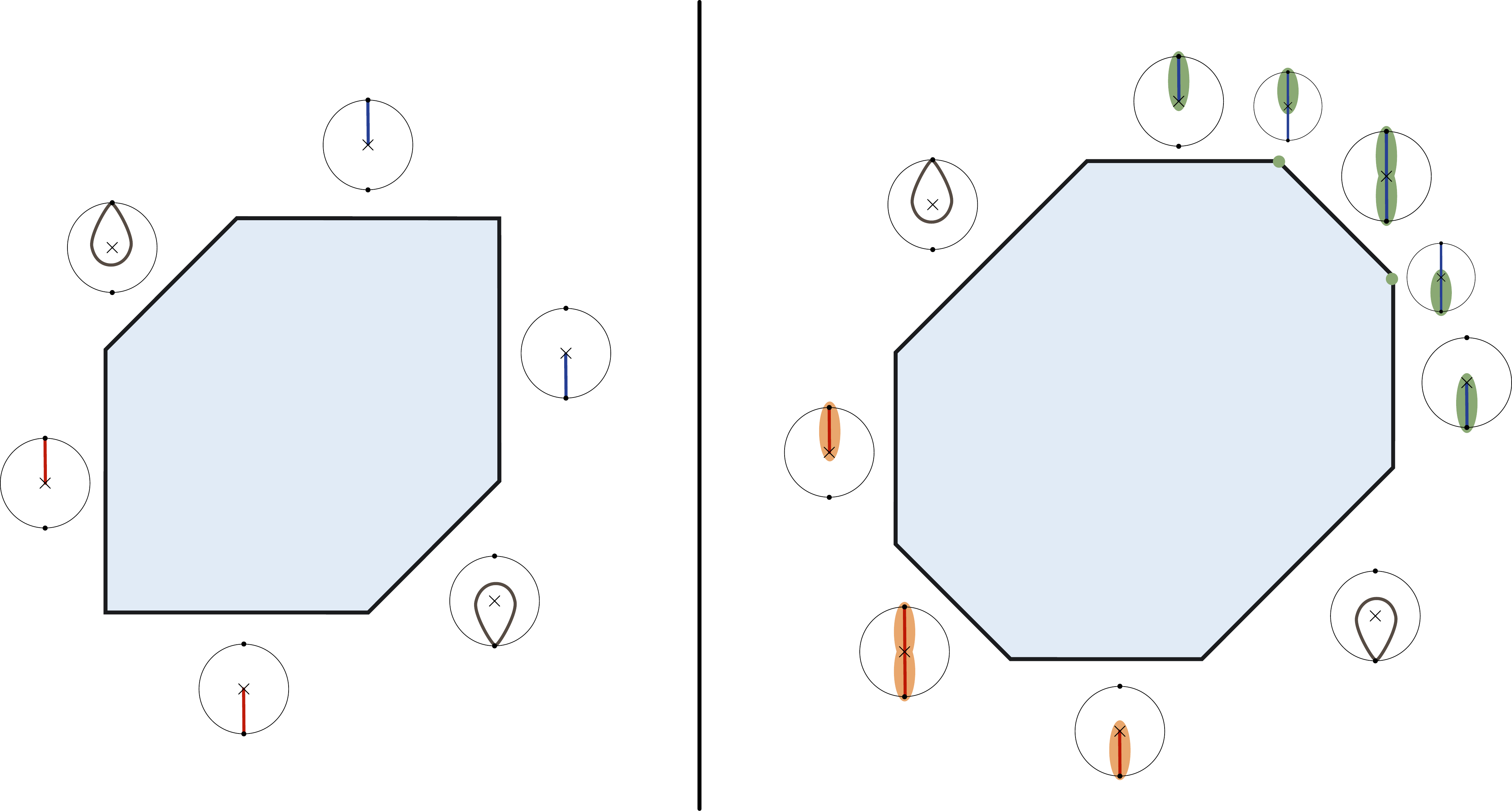}
\end{center}
\caption{(a) $Y$: Two-dimensional one-loop associahedron.  (b) $\X$ in $Y$: Two-dimensional one-loop $\mathcal{U}$-polytope. }
\label{fig:embed2pt}
\end{figure}

\begin{example}
For $n=2$, the one-loop associahedron is the intersection of the subspace
\begin{align}
    & y_1 +\tilde{y}_1 -x_{2,2} = c_2\,, \quad y_2 +\tilde{y}_2 -x_{1,1} = c_1\, , \nonumber \\
    & x_{1,1} + x_{2,2} -y_2 -\tilde{y}_1 = c_{1,2}\, ,\quad x_{1,1} + x_{2,2} -y_1 -\tilde{y}_2 = c_{2,1}\, ,
\end{align}
with the positive orthant $x_{1,1} \geq 0,\,  x_{2,2} \geq 0,\,  y_{1} \geq 0,\, y_{2} \geq 0$. To obtain the $\U$-polytope, we impose the stronger inequalities
\begin{align}
    x_{1,1} \geq 0\, , \,\, x_{2,2} \geq 0 \, , \,\, y_{1} \geq\epsilon_1\, , \,\, y_{2} \geq\epsilon_2 \, , \,\, y_{1} + y_{2}  \geq \epsilon_1 + \epsilon_2 + \epsilon \, .
\end{align}
for any $\epsilon_1, \epsilon_2, \epsilon >0$ and the analogous inequalities for $\tilde{y}$ for any $\tilde\epsilon_1, \tilde\epsilon_2, \tilde\epsilon$ with $\epsilon, \tilde \epsilon >0$.
Using the parameterization $\epsilon_1 = \epsilon_2 =0$ (and similarly for $\tilde{\epsilon}$) and $\epsilon=\tilde\epsilon$ we obtain the correct embedding as shown on the right of Figure \ref{fig:embed2pt}.

\begin{figure}[H]
\begin{center}
\includegraphics[width=5in]{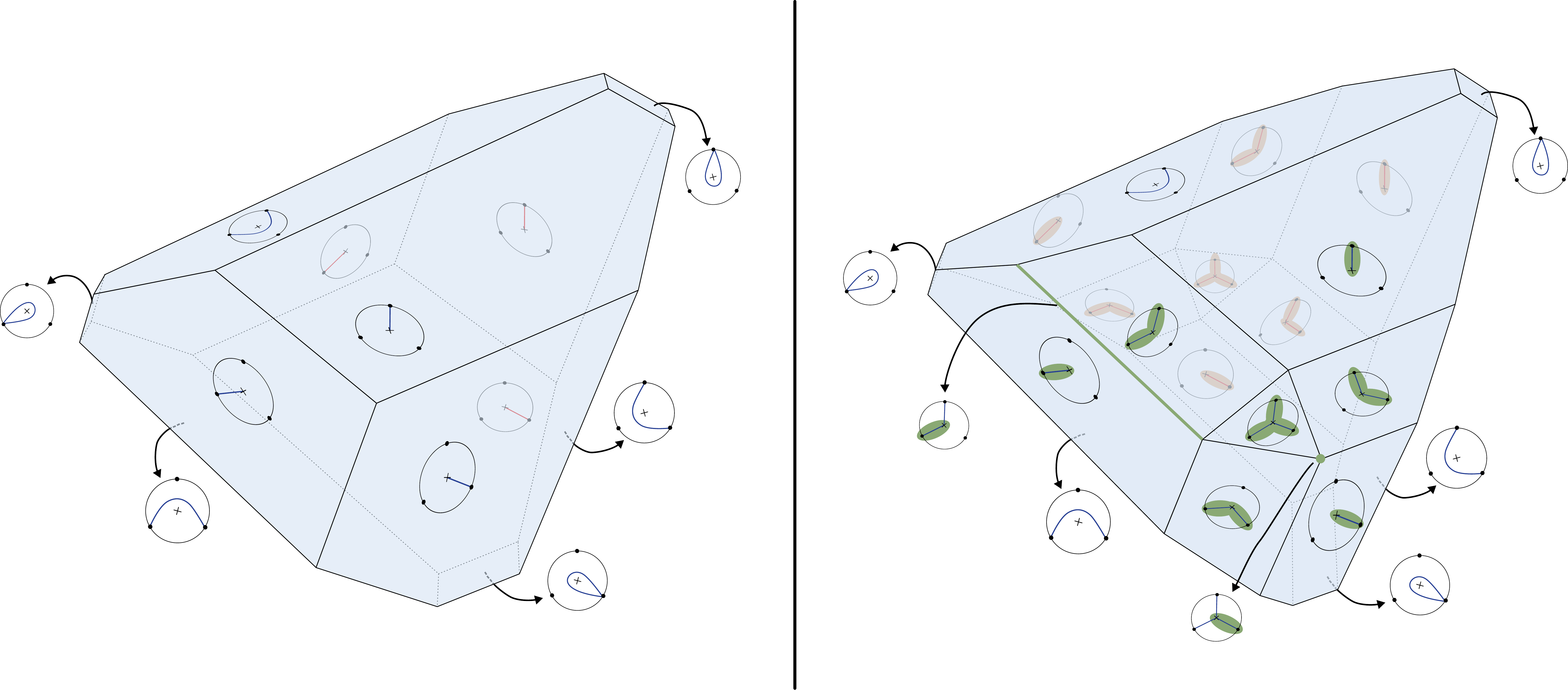}
\end{center}
\caption{(a) Y: Three-dimensional one-loop associahedron. 
(b) $\X$ in $Y$: Three dimensional one-loop $\mathcal{U}$-polytope.
\label{fig:embed3pt}}
\end{figure}

Similarly, for $n=3$, we intersect the subspace
\begin{align}
    & y_1 +\tilde{y}_1 -x_{2,2} = c_2\,, \quad y_2 +\tilde{y}_2 -x_{3,3} = c_3\, ,\quad y_3 +\tilde{y}_3 -x_{1,1} = c_1\, , \nonumber \\
    & x_{1,1} + x_{2,2}+ x_{2,1} -y_1 -\tilde{y}_3 = c_{2,1}\, ,\,\, \,  x_{2,2} + x_{3,3}+ x_{3,2} -y_2 -\tilde{y}_1 = c_{3,2} \nonumber \\
    & x_{3,3} + x_{1,1}+ x_{1,3} -y_3 -\tilde{y}_2 = c_{1,3}
    \, , \quad x_{1,3}+ x_{2,1} - x_{1,1} = c_{2,3}\, , \nonumber \\
    & x_{2,1}+ x_{3,2} - x_{2,2} = c_{3,1}\, , \quad x_{3,2}+ x_{1,3} - x_{3,3} = c_{1,2}\, ,
\end{align}
with the inequalities 
\begin{align}
    & x_{i,i}\geq 0\, ,\quad x_{2,1}\geq0\, ,\quad x_{3,2}\geq0\, ,\quad x_{1,3}\geq0, \nonumber \\
    & y_i \geq \epsilon_i\,, \quad y_i + y_j \geq \epsilon_i + \epsilon_j + \epsilon\,,\quad y_1 + y_2 + y_3 \geq \epsilon_1 + \epsilon_2 + \epsilon_3 + 2\epsilon
\end{align}
for $\epsilon_1,\epsilon_2,\epsilon_3,\epsilon>0$,  
and analogous inequalities for the $\tilde{y}$s.

With $\epsilon_i=\tilde\epsilon_i=0$ and $\epsilon=\tilde\epsilon=1$, we produce the correct embedding seen in Figure \ref{fig:embed3pt}. 
\end{example}

\bibliographystyle{alpha}
\bibliography{cosmohedra}

\end{document}